\documentclass{tran-l}
\usepackage{graphicx}
\usepackage{epstopdf}
\usepackage{subfigure}
\usepackage{verbatim}
 \allowdisplaybreaks[1]

\newtheorem{theorem}{Theorem}[section]
\newtheorem{lemma}[theorem]{Lemma}

\theoremstyle{definition}
\newtheorem{definition}{Definition}[section]

\newtheorem{example}[theorem]{Example}
\newtheorem{xca}[theorem]{Exercise}
\newtheorem{proposition}[theorem]{Proposition}
\theoremstyle{remark}
\newtheorem{remark}[theorem]{Remark}

\numberwithin{equation}{section}

\newcommand{\eps}{\varepsilon}

\newcommand{\lam}{\lambda}
\newcommand{\alp}{\alpha}

\newcommand{\dx}{\,\mathrm{d}x}

\newcommand{\R}{{\mathbb{R}}}
\newcommand{\h}{{\mathcal{H}}}

\newcommand{\inte}{\int_{\mathbb{R}^2}}
\newcommand{\intB}{\int _{B_\delta (x_{2,k})}}
\newcommand{\intPB}{\int _{\partial B_\delta (x_{2,k})}}

\newcommand{\abs}[1]{\lvert#1\rvert}

\newcommand{\blankbox}[2]{%
  \parbox{\columnwidth}{\centering
    \setlength{\fboxsep}{0pt}%
    \fbox{\raisebox{0pt}[#2]{\hspace{#1}}}%
  }%
}

\begin{document}

\title[Semi-trivial Limit Behavior
of BEC Systems]{Ground States of Two-Component Attractive Bose-Einstein Condensates II: Semi-trivial Limit Behavior}

\author{Yujin Guo}
\address{Wuhan Institute of Physics and Mathematics,
Chinese Academy of Sciences, P.O. Box 71010, Wuhan 430071,
    P. R. China. }
\email{yjguo@wipm.ac.cn}
\thanks{Y. J.  Guo and S. Li are partially supported by NSFC under Grant No. 11671394 and MOST under Grant No. 2017YFA0304500. J. C. Wei is partially supported by NSERC of Canada. X. Y. Zeng is partially supported by NSFC grant 11501555.}

\author{Shuai Li}
\address{University of Chinese Academy of Sciences, Beijing 100190, P. R. China;  Wuhan Institute of Physics and Mathematics,
    Chinese Academy of Sciences, P.O. Box 71010, Wuhan 430071,
    P. R. China.}
\email{lishuai\_wipm@outlook.com}

\author{Juncheng Wei}
\address{Department of Mathematics, University of British Columbia, Vancouver, BC V6T 1Z2,
    Canada. }
\email{jcwei@math.ubc.ca}

\author{Xiaoyu Zeng}
\address{Department of Mathematics, Wuhan University of Technology, Wuhan 430070, P. R. China. }
\email{xyzeng@whut.edu.cn}


\date{on August 15th, 2017; accepted on February 15th, 2018.}



\begin{abstract}
As a continuation of \cite{GLWZ}, we study new pattern formations of ground states $(u_1,u_2)$ for two-component Bose-Einstein condensates (BEC) with homogeneous trapping potentials in $\R^2$, where the intraspecies interaction $(-a,-b)$ and the interspecies interaction $-\beta$ are both attractive, $i.e,$ $a$, $b$ and $\beta$ are all positive. If $0<b<a^*:=\|w\|^2_2$ and $0<\beta <a^*$ are fixed, where $w$ is the unique positive solution of $\Delta w-w+w^3=0$ in $\R^2$, the semi-trivial behavior of $(u_1,u_2)$ as $a\nearrow a^*$ is proved in the sense that $u_1$ concentrates at a unique point and while $u_2\equiv 0$ in $\R^2$. However, if $0<b<a^*$ and $a^*\le\beta <\beta ^*=a^*+\sqrt{(a^*-a)(a^*-b)}$, the refined spike profile and the uniqueness  of $(u_1,u_2)$ as $a\nearrow a^*$ are analyzed, where $(u_1,u_2)$ must be unique, $u_1$ concentrates at a unique point, and meanwhile $u_2$ can either blow up or vanish, depending on how $\beta$ approaches to $a^*$.
\end{abstract}

\maketitle

\tableofcontents

\section{Introduction}
In this paper, we study the following coupled nonlinear Gross-Pitaevskii system
\begin{equation}\label{equ:CGPS}
\begin{cases}
-\Delta u_{1} +V_1(x)u_{1} =\mu u_{1} +au_{1}^3 +\beta u_{2}^2 u_{1}   \,\ \mbox{in}\,\  \R^2,\\
-\Delta u_{2} +V_2(x)u_{2} =\mu u_{2} +bu_{2}^3 +\beta u_{1}^2 u_{2}   \,\ \mbox{in}\,\  \R^2,\,\
\end{cases}
\end{equation}
where $(u_{1} ,u_{2})\in\mathcal{X}=\h_1(\R ^2)\times \h_2(\R ^2)$  and
the space
\begin{equation*}
 \h_i(\R ^2) = \Big \{u\in  H^1(\R ^2):\ \int _{\R ^2}  V_i(x)|u(x)|^2\dx<\infty \Big\}
\end{equation*}
is equipped with the norm $\|u\|_{_{\h_i}}=\big(\int _{\R ^2} \big[|\nabla u|^2+ V_i(x)|u(x)|^2\big] \dx\big)^{\frac{1}{2}}$ for $i=1, 2$. The system \eqref{equ:CGPS} is used (see \cite{BC,DW,EGBB,FM,HMEWC,LWCMP,LW,LW2,PW,Royo}) to describe two-component Bose-Einstein condensates (BEC) with trapping potentials $V_1(x) $ and $V_2(x) $, where $\mu\in\R$ is a chemical potential.  From the physical point of view, we assume
that the trapping potentials $ V_i(x)\ge 0$ ($i=1, 2$) satisfy
\begin{equation}\label{cond:V.1}
\lim_{|x|\to\infty} V_i(x) = \infty, \ \inf_{x\in\R^2}V_i(x)=0
\text{ and } \inf\limits_{x\in \R^2} \big(V_1(x)+V_2(x)\big) \text{ are attained.}
 \end{equation}
In the system \eqref{equ:CGPS}, $a>0$ and $b>0$ ($resp.$ $<0$) represent that the intraspecies interaction  of the atoms inside each component is attractive  ($resp.$ repulsive), and $\beta>0$ ($resp.$ $<0$) denotes that the interspecies interaction  between two components is attractive ($resp.$ repulsive).

As a continuation of \cite{GLWZ}, in this paper we study ground states of (\ref{equ:CGPS}) for the case where the intraspecies interaction and interspecies interaction are both attractive, $i.e.$ $a, b, \beta>0$.
As illustrated in \cite[Proposition A.1]{GLWZ}, ground states of (\ref{equ:CGPS}) in this case can be described equivalently by the minimizers of the following $L^2-$critical constraint variational problem
\begin{equation}\label{def:e}
e(a,b,\beta):=\inf_{\{(u_1,u_2)\in \mathcal{X}:\, \int_{\R ^2}(u_1^2+u_2^2) \dx=1\}} E_{a,b,\beta}(u_1,u_2),\,\, a>0,\,\, b>0,\,\, \beta>0,
\end{equation}
where the Gross-Pitaevskii (GP) energy functional $ E_{a,b,\beta}(u_1,u_2)$ is given by
\begin{equation}\label{def:E}
\begin{split}
E_{a, b, \beta}(u_1,u_2)
=&\int_{\R ^2} \big(|\nabla u_1|^2+|\nabla u_2|^2\big)\dx +\int_{\R ^2}\big( V_1(x)u_1^2+V_2(x)u_2^2\big) \dx\\
&-\int_{\R ^2} \big(\frac{a}{2}|u_1|^4+\frac{b}{2}|u_2|^4 +\beta |u_1|^2|u_2|^2\big) \dx \,,\quad (u_1,u_2)\in\mathcal{X}.
\end{split}
\end{equation}
To discuss equivalently ground states of (\ref{equ:CGPS}), throughout the whole paper we shall therefore focus on investigating  \eqref{def:e}, instead of (\ref{equ:CGPS}). Since the GP energy functional $ E_{a, b, \beta}(u_1,u_2)$ is even in $(u_1,u_2)$, any minimizer $(u_1, u_2)$ of $e(a,b,\beta)$ must be either nonnegative or nonpositive. Without loss of generality, in this paper we therefore restrict to study nonnegative minimizers of $e(a,b,\beta)$, which are called {\em ground states} of (\ref{equ:CGPS}).

Besides the assumption \eqref{cond:V.1}, for the physical correlation we shall consider the trapping potentials $V_1(x)$ and $V_2(x)$ in the class of homogeneous functions, for which we define
\begin{definition}\label{cond:h.homo}
  $h(x)\geq0$ in $\R^2$ is homogeneous of degree $p\in\R^+$ (about the origin), if $h(x)$ satisfies
\begin{equation}\label{def:h.homo}
  h(tx)=t^p h(x)\,\,\,\text{in}\,\,\, \R^2\,\,\,\text{for any}\,\,\,t>0.
\end{equation}
\end{definition}
\noindent The above definition implies that the homogeneous function $h(x)$ of degree $p\in\R^+$ satisfies
\begin{equation}\label{prop:h.est}
  0\leq h(x)\leq C|x|^p \,\,\,\text{in} \,\,\,\R^2,
\end{equation}
where $C>0$ denotes the maximum of $h(x)$ on $\partial B_1(0)$. Note that $\nabla h(x)=0$ if and only if $x=0$ for the case where $\lim\limits_{|x|\to\infty}h(x)=+\infty$.
We also use $w=w(|x|)$ to denote (cf. \cite{BW,GNN,Kwong,LN,MS,RS}) the unique (up to translations) positive radially symmetric solution of the following nonlinear scalar field equation
\begin{equation}\label{equ:w}
\Delta w-w+w^3=0,\,\,\  w \in H^1(\R^2).
\end{equation}
We remark that $w$ satisfies (cf. \cite{GS}) the following identifies
\begin{equation}\label{ide:w}
  \|w\|_2^2=\|\nabla w\|_2^2=\frac{1}{2} \|w\|_4^4,
\end{equation}
and it follows from \cite[Proposition 4.1]{GNN} that $w(x)$ also satisfies
 \begin{equation} \label{decay:w}
w(x) \, , \ |\nabla w(x)| = O(|x|^{-\frac{1}{2}}e^{-|x|}) \,\
\text{as} \,\ |x|\to \infty.
\end{equation}
Recall from    \cite{GLWZ} that the analysis of $e(a,b,\beta)$ depends strongly on the following Gagliardo-Nirenberg type inequality
 \begin{equation}\label{ineq:GN}
 \int_{\R ^2} \big(|u_1|^2+|u_2|^2  \big)^2 \dx
 \le  \frac{2}{ \|w\|_2^2}
   \int_{\R ^2} \big(|\nabla u_1|^2+|\nabla u_2|^2\big) \dx  \int_{\R ^2}\big(u_1^2+u_2^2\big) \dx,
 \end{equation}
where $(u_1,u_2)\in H^1(\R^2)\times H^1(\R^2)$. It is proved in \cite[Lemma A.2]{GLWZ} that $\frac{2}{ \|w\|_2^2} $ is the best constant of (\ref{ineq:GN}), where the equality is attained at $(w\sin \theta  , w\cos \theta )$ for any $\theta \in [0,2\pi)$.

When $V_i(x)\in C^2(\R^2)$ is homogeneous of degree $p_i\ge 2$ and satisfies \eqref{cond:V.1} for $i=1$ and $2$, it then follows immediately from \cite{GLWZ} the following existence and nonexistence.

\vspace {.2cm}

\noindent {\bf Theorem A (Theorems 1.1 and 1.2 in \cite{GLWZ})} {\em
Suppose $V_i(x)\ge 0$ satisfies \eqref{cond:V.1}  and there exists at least one common point $x_0\in\R^2$ such that $V_i(x_0)=\inf\limits_{x\in\mathbb{R}^2}  V_i(x) =0$, where $i=1, \,2$. Set
\begin{equation}\label{thmAAA:1.1}
\beta^*=\beta^*(a,b):= a^* + \sqrt {(a^*-a)(a^*-b)}, \,\ \text{where}\,\ 0< a, \,b<a^*:=\|w\|_2^2.
\end{equation}
Then  $e(a,b,\beta)$  admits minimizers if and only if $0<a<a^*$, $0<b<a^*$ and  $0<\beta<\beta^*$.}

\vspace {.2cm}

\noindent
The above Theorem A shows that $e(a, b, \beta)$ admits minimizers if and only if the point $(a, b,\beta)$ lies within the cuboid described by Figure 1(a) below. Following \cite[Proposition A.1]{GLWZ} on the equivalence between ground states of (\ref{equ:CGPS}) and minimizers of $e(a, b, \beta)$, one can further obtain that for any given $(a, b,\beta)$,  a minimizer of $e(a,b,\beta)$ is a ground state of (\ref{equ:CGPS}) for some  $\mu\in \R$; conversely, a ground state of (\ref{equ:CGPS}) for some $\mu\in \R$ is a minimizer of $e(a, b, \beta)$.

By employing the energy method and blow up analysis, the uniqueness and the refined blow up behavior of nonnegative minimizers $(u_1, u_2)$ for $e(a,b,\beta)$ are investigated in \cite{GLWZ} under different types of trapping potentials, where we consider $0<a<a^*$, $0<b<a^*$ and $\beta \nearrow \beta ^*:=a^*+\sqrt{(a^*-a)(a^*-b)}$. In such a limit case, it turns out in \cite{GLWZ} that $(u_1, u_2)$ must be unique and blows up at a unique point. This further implies the strict positivity of $(u_1, u_2)$ in such a limit case.

The main purpose of this paper is to investigate new pattern formations of nonnegative minimizers $(u_1, u_2)$ for $e(a,b,\beta)$, where $0<b<a^*$,  $\beta\in (0, \beta ^*)=(0, a^*+\sqrt{(a^*-a)(a^*-b)})$ and $a\nearrow a^*$. Different from those studied in \cite{GLWZ}, we shall analyze that $(u_1, u_2)$ may admit the semi-trivial limit behavior for this case, depending on how $\beta$ approaches to $a^*$, in the sense that $u_1>0$ and $u_2\equiv 0$ in $\R^2$.

\subsection{Main results}

In this subsection, we shall introduce the main results of this paper. Stimulated by \cite[Theorem 1.1]{Grossi}, we define
\begin{equation}\label{H}
H_i(y):=\int_{\R^2}V_i(x+y)w^2(x)\dx>0, \ \text{ where } \ i=1,\, 2.
\end{equation}
We remark that our analysis also makes full use of the following classical Gagliardo-Nirenberg type inequality
\begin{equation}\label{ineq:GNQ}
\frac {\|w\|_2^{2}}{2}
  =\inf\limits_{u\in H^1(\R^2)\setminus \{0\}}
\frac{\inte |\nabla u(x) |^2 \dx \inte |u(x)|^2 \dx}{\inte |u(x)|^4  \dx},
\end{equation}
where the equality is attained at $w$ (cf. \cite{W}).
Our first result is concerned with the following interesting  limit behavior of nonnegative minimizers.

\begin{theorem}\label{thm:1.1}
Suppose $0\le V_i(x)\in C^2(\R^2)$ is homogeneous of degree $p_i$ with $2\le p_1\leq p_2$, where $V_i(x)$ satisfies  \eqref{cond:V.1} and
\begin{equation}\label{1:H}
\text{$y_0 $ is the unique and non-degenerate critical point of $H_1(y)$.}
\end{equation}
Let $ (u_{1k},u_{2k})$ be a nonnegative minimizer of $e(a_k, b, \beta _k)$, where $0<b<a^*:=\|w\|^2_2$, $a_k\nearrow a^*$ as $ k\to\infty$ and
\begin{equation}\label{con:a1beta.ab}
a^*<\beta _k  <\beta_k^*=a^*+\sqrt{(a^*-a_k)(a^*-b)} \ \text{ and }  \
 a^*-a_k=o\big(\beta_k-a^*\big)
\end{equation}
as $k\to\infty$.
Then there exists a subsequence, still denoted by $\{a_k\}$, of $\{a_k\}$   such that
\begin{equation}\label{thm1.2:M}
\sqrt{a^*} \varepsilon_k   u_{1k}(\varepsilon_k  x+x_{1k})\to w(x)\ \,\,\text{and}\,\,\
  \sqrt{\frac{a^*(a^*-b)}{\beta _k -a^*}}    \varepsilon_k   u_{2k}(\varepsilon_k  x+x_{2k})\to w(x)
\end{equation}
uniformly in $\R^2$ as $ k\to\infty$, where $\varepsilon_k>0$ is given by
\begin{equation}\label{1.2:eps.rate}
  \varepsilon_k:=\frac{1}{\lambda}\Big[(a^*-a_k)(a^*-b)-(\beta_k-a^*)^2\Big]^\frac{1}{2+p_1},\,\,  \lambda =\Big[\frac{p_1(a^*-b)}{2}H_1(y_0 )\Big]^\frac{1}{2+p_1},
\end{equation}
and $x_{ik}$ is the unique maximum point of $u_{ik}$ satisfying
\begin{equation}\label{lim:betaa1.z.rate-D}
  \lim\limits_{k\to\infty}\frac{x_{ik}}{\varepsilon_k}=y_0,\ \ i=1, 2.
\end{equation}
\end{theorem}

We remark that the similar estimate of (\ref{lim:betaa1.z.rate-D}) appeared earlier in \cite{Grossi}, where  a singular perturbation  problem was studied.
Even though Theorem \ref{thm:1.1} is proved mainly by the variational methods and blow up analysis as employed in \cite{GLWZ,GZZ2,Lions,S,Wang}, there are some new difficulties appearing in its proof. Firstly, since the blow up rate (\ref{1.2:eps.rate}) of Theorem \ref{thm:1.1} is different from those in \cite{GS,GLWZ,GZZ2,MMP,Manakov} and references therein, as in Proposition \ref{Prop:a1beta.3} one needs to seek for a different type of test functions so that the optimal upper estimate of $e(a_k, b, \beta _k)$ can be derived. Secondly, since the existing argument only gives that $\varepsilon_k u_{2k}(\varepsilon_k x+x_{1k})\to 0$ uniformly in $\R^2$ as $ k\to\infty$, one needs to investigate an approach of addressing that $u_{2k}\not\equiv 0$ for sufficiently large $k>0$. As shown in Lemma \ref{lem3:u2}, we shall achieve this purpose by analyzing the more refined energy estimates of $e(a_k, b, \beta _k)$, for which we make full use of the refined spike profiles proved in \cite[Theorem 1.2]{GLW}. Once $u_{2k}\not\equiv 0$ holds for sufficiently large $k>0$, we define
\begin{equation}\label{A:4.1:PL}
\begin{split}
\bar u_{1k}(x)=\sqrt{a^*} \varepsilon_{k}  u_{1k}(\varepsilon_{k}x+x_{1k})\,\  \text{and}\,\ u_{2k}(\varepsilon_{k}x+x_{1k})=C_\infty \sigma _k\bar u_{2k}(x), \\
\text{where}\,\  \sigma _k=\|u_{2k}\|_\infty>0 \,\  \text{and}\,\ C_\infty=\frac{1}{\|w\|_\infty}> 0,\qquad\qquad\qquad\quad\quad
\end{split}
\end{equation}
and $x_{1k}$ is the unique maximum point of $u_{1k}$. To complete the proof of Theorem \ref{thm:1.1}, the rest key point is thus to analyze the limit behavior of $\bar u_{2k}$ and $\sigma _k$ as $ k\to\infty$, for which we shall carry out a very delicate analysis of the PDE system associated to $(\bar u_{1k}, \bar u_{2k})$. We also remark that if $\beta _k$ is close enough to $\beta ^*_k$, the limit behavior (\ref{thm1.2:M}) still holds without the non-degeneracy assumption of \eqref{1:H}, see Theorem \ref{thm3.7} for more details.

Under the assumptions of Theorem \ref{thm:1.1}, one can note from (\ref{thm1.2:M}) that the nonnegative minimizers of $e(a_k, b, \beta _k)$ exhibit interesting new pattern formations where $u_{1k}$ blows up at a unique point and however $u_{2k}$ can either blow up or vanish, depending on how $\beta _k$ approaches to $a^*$. More precisely, for given $(a_k, b)$, if $\beta _k$ goes closer to $\beta ^*_k$, then $u_{2k}$ prefers to blow up at a unique point; conversely, if $\beta _k$ goes far away from $\beta ^*_k$, then $u_{2k}$ tends to decrease its height. Especially, if $\beta _k\le a ^*$ we then have the following semi-trivial limit behavior of nonnegative minimizers,  in the sense that $u_{1k}$ blows up at a unique point and however $u_{2k} \equiv 0$ for sufficiently large $k>0$. We also comment that the authors in \cite{BSW16,BSW17} analyzed recently three-component Schr\"{o}dinger systems in which some similar semi-trivial limits were found.

\begin{figure}[tb]
\centering
\subfigure[]{\includegraphics[width = 6cm,height=4.8cm,clip]{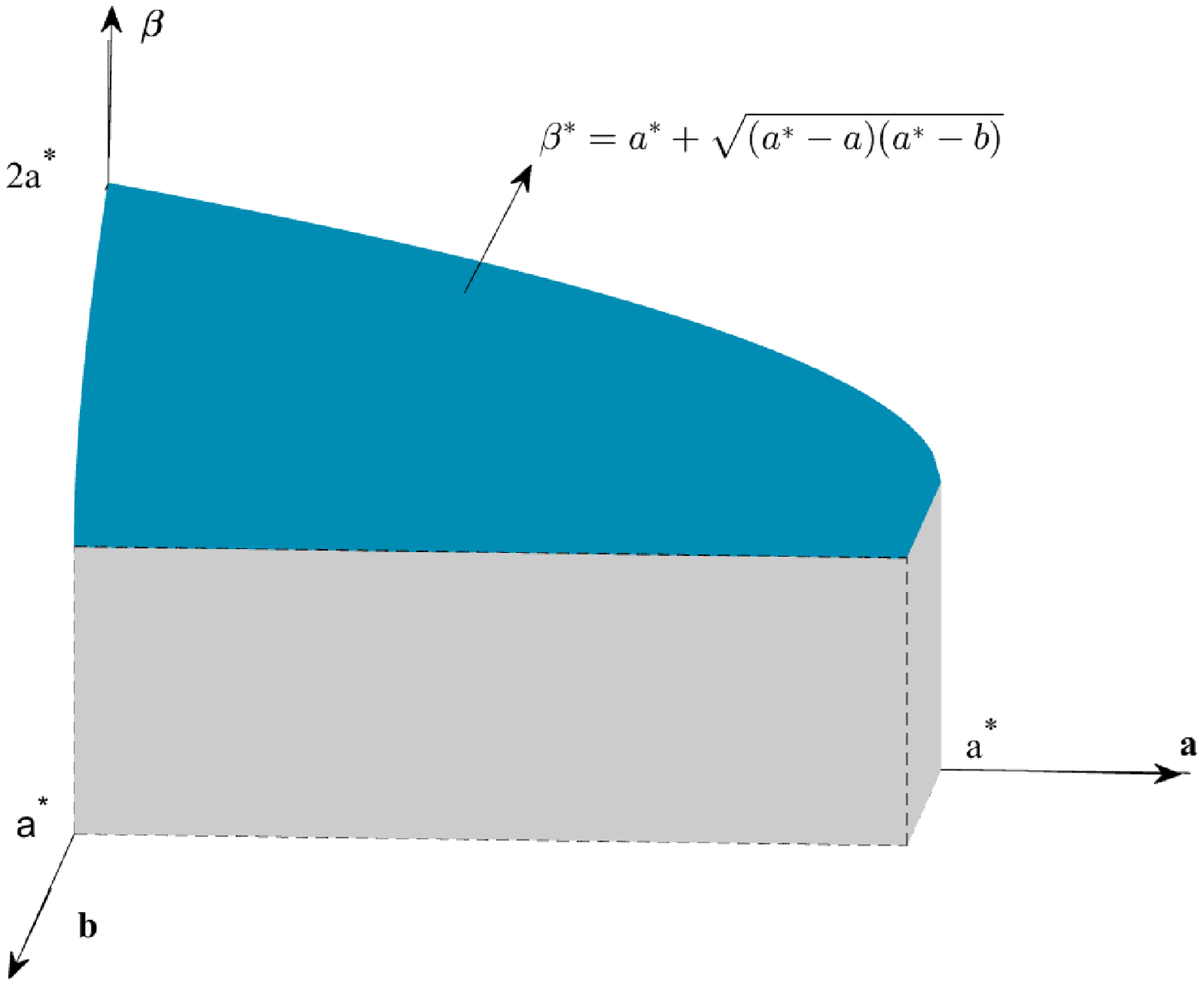}}
\subfigure[]{\includegraphics[width = 5cm,height=4.2cm,clip]{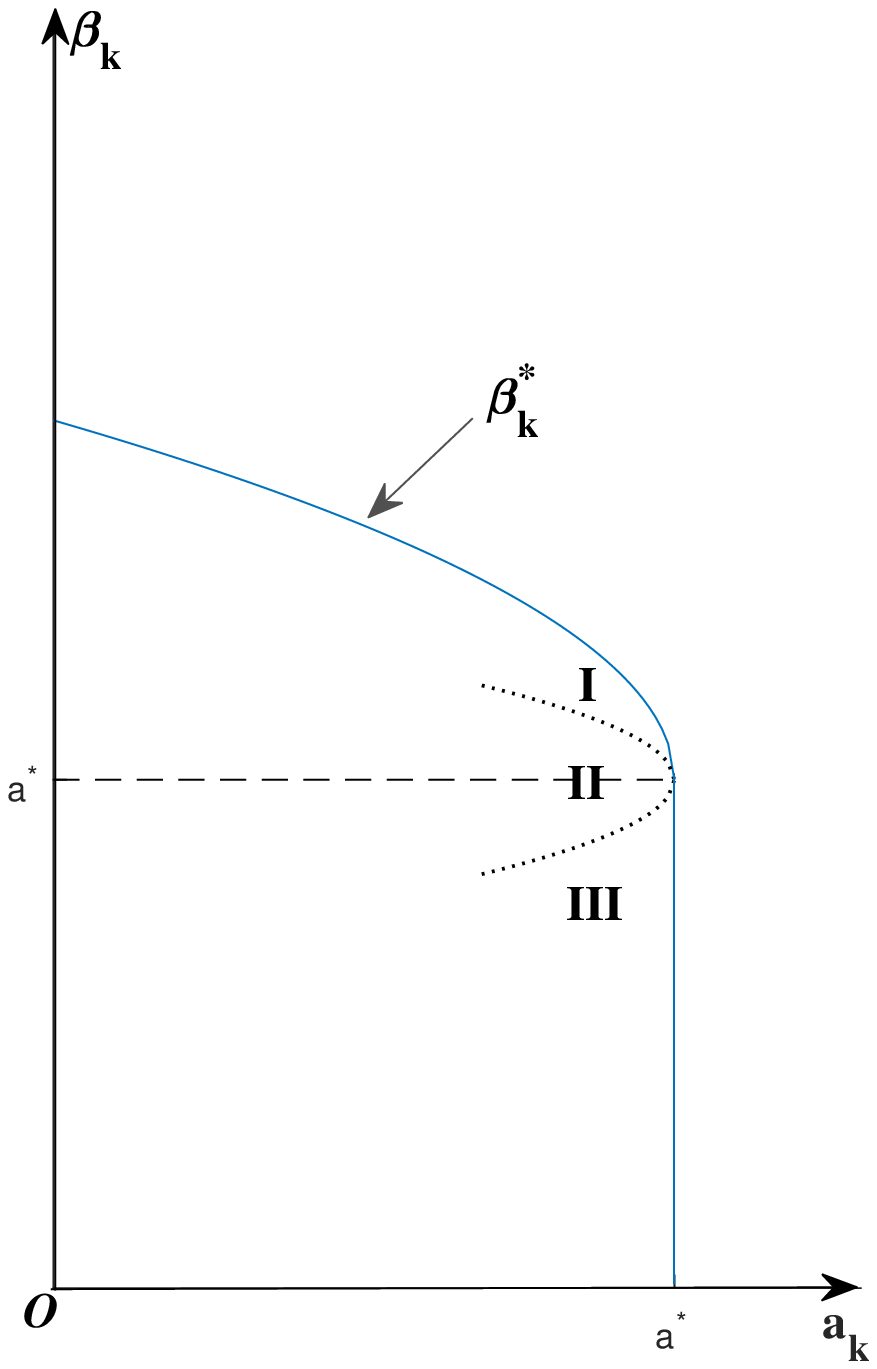}}
    \caption{\em
  Left figure: $e(a, b,\beta)$ has minimizers  if and only if  $(a, b,\beta)$ lies within the cuboid. Right figure:  For any given $0<b<a^*$, if $(a_k, \beta _k)$ lies within Region I, then $(u_{1k}, u_{2k})$ satisfies (\ref{thm1.2:M}); if   $(a_k, \beta _k)$ lies within Region III, then $(u_{1k}, u_{2k})$ satisfies (\ref{K:lim:a1.u.rate}).}
 \end{figure}

\begin{theorem}\label{thm:1.2}
Suppose $0\le V_i(x)\in C^2(\R^2)$ is homogeneous of degree $p_i\ge 2$ and satisfies \eqref{cond:V.1} for $i=1$ and $2$. Assume also that
\begin{equation}\label{cond:Vcritical.u1}
\text{$y_0 $ is the unique  critical point of $H_1(y)$.}
\end{equation}
Let $(u_{1k},u_{2k})$ be a nonnegative minimizer of $e(a_k,b,\beta _k)$, where $0<b<a^*$,  $a_k\nearrow a^*$ as $k\to\infty$ and $0<\beta _k<a^*$ satisfies
\begin{enumerate}
\item either $\beta _k\to \beta _*\in (0, a^*)$ as $k\to\infty$, or
\item $\beta_k\nearrow a^*$ and $a^*-a_k=o\big(a^*-\beta_k\big)$ as $k\to\infty$.
\end{enumerate}
 Then, up to a subsequence if necessary, we have
\begin{equation}\label{K:lim:a1.u.rate}
\begin{cases}
\lim\limits_{k\to\infty}\sqrt{a^*} \varepsilon_k  u_{1k}(\varepsilon_kx+x_{1k})=w(x)\,\, \text{uniformly in $\R^2$,}\\
u_{2k}(x)\equiv0\,\,\,\text{in $\R^2$, when $k>0$ is large enough},
\end{cases}
\end{equation}
where
\begin{equation}\label{K:def:a1.eps}
\varepsilon_k :=\frac{1}{\lambda_1}(a^*-a_k)^\frac{1}{p_1+2}>0,\,\,\
\lambda_1:=\Big[\frac{p_1}{2}H_1(y_0 )\Big]^\frac{1}{2+p_1},
\end{equation}
and the point $x_{1k}$ is the unique maximum point of $u_{1k}$ satisfying
\begin{equation}\label{K:lim:a1.z.rate}
  \lim\limits_{k\to\infty}\frac{x_{1k}}{\varepsilon_k} =y_0 .
\end{equation}
\end{theorem}

The challenging point of proving Theorem \ref{thm:1.2} is to prove that $u_{2k}\equiv 0$ for sufficiently large $k>0$. Roughly speaking, by contradiction if $u_{2k} \not \equiv 0$ for the case where $0<\beta _k<a^*$ satisfies $\beta _k\to \beta _*\in (0, a^*)$ as $k\to\infty$, a suitable transform of $u_{2k}$ then approaches to a nontrivial nonnegative solution of $\Delta u-u+\frac{\beta _*}{a^*}w^2u=0$ in $\R^2$, which is however a contradiction in view of \cite[Lemma 4.1]{Wei96}, see Theorem \ref{Th:b1} for details. However, if $u_{2k} \not \equiv 0$ for the case where  $0<\beta _k<a^*$ satisfies $\beta_k\nearrow a^*$ and $a^*-a_k=o\big(a^*-\beta_k\big)$ as $k\to\infty$, we shall consider (\ref{A:4.1:PL}) as a transform of $u_{2k}$, from which the argument of proving Theorem \ref{thm:1.1} finally leads to a contradiction.
As illustrated by Figure 1(b), we also mention that for any given $0<b<a^*$, if $(a_k, \beta _k)$ approaches to $(a^*, a^*)$ within Region I ($resp.$ Region III), then the limit behavior of $(u_{1k}, u_{2k})$ can be described by  Theorem  \ref{thm:1.1} ($resp.$ Theorem \ref{thm:1.2}). However, for any $0<b<a^*$, if $(a_k, \beta _k)$ approaches to $(a^*, a^*)$ within Region II, we expect that $u_{2k}$ can either blow up or vanish, depending on $V_i(x)$ and how $\beta _k$ approaches to $a^*$.

%
%

Under the non-degeneracy assumption of \eqref{1:H}, we finally address the following uniqueness of nonnegative minimizers for $e(a, b, \beta)$.

\begin{theorem}\label{thm1.3}
Suppose $0\le V_i(x)\in C^2(\R^2)$ is homogeneous of degree $p_i$ and satisfies  \eqref{cond:V.1}  and (\ref{1:H}), where $2\le p_1\le p_2$
and $H_2(y_0)\not =H_1(y_0)$ for the case $p_1= p_2$.
Then there exists a unique nonnegative minimizer for $e(a, b, \beta)$, where $(a, b, \beta)$ satisfies
\begin{equation}\label{5:beta}
0<b<a^*, \ \  a^*\le \beta <\beta ^*:=a^*+\sqrt{(a^*-a)(a^*-b)},\ \  a^*-a=o\big(\beta -a^*\big)   
\end{equation}
as $ a\nearrow a^*$.
\end{theorem}

Even though the uniqueness of nonnegative minimizers for $e(a, b, \beta)$ is also tackled in \cite[Theorem 1.5]{GLWZ}, there are some essential differences in the proof of Theorem \ref{thm1.3}. To prove Theorem \ref{thm1.3}, by contradiction we suppose $(u_{1,k}, v_{1,k})$ and $(u_{2,k}, v_{2,k})$ to be two different nonnegative minimizers of $e(a_k, b, \beta_k)$. The proof of Theorem \ref{thm:1.1} then motivates us to define
\begin{equation}\label{K:uniq:a-2}
\begin{split}
\bar u_{i,k}(x)=\sqrt{a^*} \varepsilon_{k}  u_{i,k}(\varepsilon_{k}x+x_{2,k})\,\  \text{and}\,\ v_{i,k}(\varepsilon_{k}x+x_{2,k})=C_\infty \sigma _k\bar v_{i,k}(x), \\
\text{where}\,\ i=1, 2,\,\  \sigma _k=\|v_{2,k}\|_\infty>0 \,\  \text{and}\,\ C_\infty=\frac{1}{\|w\|_\infty}> 0.\quad\quad\quad\quad\quad
\end{split}
\end{equation}
Here $\varepsilon_{k}>0$ is given by Proposition \ref{Prop:a1beta.3}, and $x_{2,k}$ is the unique maximum point of $u_{2,k}$.
Different from \cite[Theorem 1.5]{GLWZ}, we then need to consider the following difference function
\begin{equation}\label{K:uniq:a-7}\arraycolsep=1.5pt
\begin{array}{lll}
\xi_{1,k}(x)&=\displaystyle\frac{ \bar  u_{2,k}(x)-   \bar u_{1,k}(x)}{\|  \bar u_{2,k}-  \bar  u_{1,k}\|_{L^\infty(\R^2)}+\| \bar  v_{2,k}-  \bar v_{1,k}\|_{L^2(\R^2)}},\\[4mm]
\xi_{2,k}(x)&=\displaystyle\frac{ \bar  v_{2,k}(x)-  \bar  v_{1,k}(x)}{\| \bar  u_{2,k}-   \bar u_{1,k}\|_{L^\infty(\R^2)}+\| \bar  v_{2,k}- \bar  v_{1,k}\|_{L^2(\R^2)}}.
\end{array}
\end{equation}
By using more delicate analysis, the limit behavior $(\xi_{10}, \xi_{20})$ of $(\xi_{1,k}, \xi_{2,k})$ as $k\to\infty$ further turns out to satisfy the following non-degenerate system: as proved in (\ref{uniq:limit-A3}), the solution set of
\begin{equation} \label{K:Auniq:limit-2} \arraycolsep=1.5pt
\left\{\begin{array}{lll}
\qquad   \mathcal{L}_1\xi_{10}&:=&-\Delta \xi_{10}+\big[1-3w^2\big] \xi_{10} =\displaystyle \frac{2}{a^*} \Big( \inte u_0^3\xi_{10}\Big)w   \,\ \mbox{in}\,\  \R^2, \\ [2.5mm]
 \mathcal{L}_2(\xi_{10})\xi_{20}&:=&-\Delta \xi_{20}+(1- w^2 )\xi_{20} -2w^2\xi_{10} =\displaystyle \frac{2}{a^*} \Big( \inte u_0^3\xi_{10}\Big)w   \,\ \mbox{in}\,\  \R^2,
 \end{array}\right.\end{equation}
satisfies
\begin{equation}\label{K:Auniq:limit-3}
\left(\begin{array}{cc} \xi_{10}\\[2mm]
 \xi_{20}\end{array}\right)=b_0\left(\begin{array}{cc} 0\\[2mm]
w\end{array}\right)+\sum _{j=1}^{2}b_j\left(\begin{array}{cc} \frac{\partial w}{\partial x_j}\\[2mm] \frac{\partial w}{\partial x_j}
\end{array}\right)+c_0 \left(\begin{array}{cc}  w+x\cdot \nabla w \\[2mm]
w+x\cdot \nabla w \end{array}\right)
\end{equation}
for some constants $c_0$ and $b_j$ with $j=0, 1, 2$, which is more involved than those in \cite{GLW,GLWZ}. By deriving local Pohozaev identities (cf. \cite{Cao,Deng,Grossi,GLW}), we shall first prove that $c_0=0$, based on which we shall derive that $b_1=b_2=0$ in (\ref{K:Auniq:limit-3}). Following these, we shall prove that $\xi_{1,k}(x)\to \xi_{10}\equiv 0$ uniformly in $\R^2$ as $k\to\infty$. To reach a contradiction by further showing $b_0=0$, one then needs to derive a refined expansion of $\xi_{1,k} $ in terms of $\sigma _k$ and $\eps_k$.

When $H_1(y)$ has $N$ non-degenerate critical points, it was proved in \cite{Grossi} that the number of single peak solutions for some scalar equations equals exactly to $N$, where $N\ge 1$.  Our results show that the uniqueness of Theorem \ref{thm1.3} is true for the case where $N=1$, and it seems more complicated for the general case where $N>1$.

This paper is organized as follows. The main purpose of Section 2 is to establish Theorem \ref{Th:b1}. In Section 3 we shall first establish Proposition \ref{Prop:a1beta.3}, based on which we then finish the proof of Theorems \ref{thm:1.1} and \ref{thm:1.2} in Subsection 3.1. Following Proposition \ref{Prop:a1beta.3}, in Section 4 we shall complete the proof of Theorem \ref{thm1.3}. The proofs of Lemma \ref{lem4.3} and (\ref{5.3:exponent}) are given in Appendix A.

 \section{Limit Behavior of Nonnegative Minimizers: $0< \beta < a^*$  }

In this section, we mainly establish the following Theorem \ref{Th:b1} on the semi-trivial limit behavior of nonnegative minimizers for $e(a,b,\beta)$.

\begin{theorem}\label{Th:b1}
Suppose $V_i(x)\in C^2(\R^2)$ is homogeneous of degree $p_i\ge 2$ and satisfies \eqref{cond:V.1} and (\ref{cond:Vcritical.u1}) for $i=1$ and $2$. Let $(u_{1k},u_{2k})$ be a nonnegative minimizer of $e(a_k,b,\beta_k)$, where $0<b<a^*$,  $a_k\nearrow a^*$ and $0<\beta_k<a^*$ satisfies $\beta_k\to\beta _*\in (0, a^*)$ as $k\to\infty$. Then, up to a subsequence if necessary, we have
\begin{equation}\label{lim:a1.u.rate}
\begin{cases}
   \lim\limits_{k\to\infty}\sqrt{a^*} \varepsilon_k  u_{1k}(\varepsilon_kx+x_{1k})=w(x)\,\, \text{uniformly in $\R^2$,}\\
   u_{2k}(x)\equiv0\,\,\,\text{in $\R^2$, when $k$ is large enough},
\end{cases}
\end{equation}
where
\begin{equation}\label{def:a1.eps}
\varepsilon_k :=\frac{1}{\lambda_1}(a^*-a_k)^\frac{1}{p_1+2}>0,\,\,\
\lambda_1:=\Big[\frac{p_1}{2}H_1(y_0 )\Big]^\frac{1}{2+p_1},
\end{equation}
and the point $x_{1k}$ is the unique maximum point of $u_{1k}$ satisfying
\begin{equation}\label{lim:a1.z.rate}
  \lim\limits_{k\to\infty}\frac{x_{1k}}{\varepsilon_k} =y_0
\end{equation}
for $y_0\in\R^2$ given in (\ref{cond:Vcritical.u1}).
\end{theorem}

In order to prove Theorem \ref{Th:b1},  let $(u_{1k},u_{2k})$ be a nonnegative minimizer of $e(a_k,b,\beta _k)$ in view of Theorem A, so that the expression of $e(a_k,b,\beta _k)$ can be rewritten as
\begin{equation}\label{exp2:ea}
\begin{split}
  &e(a_k,b,\beta _k) =E_{a_k,b,\beta _k}(u_{1k},u_{2k})\\
  =&\int_{\R^2} \big(|\nabla u_{1k}(x)|^2 + |\nabla u_{2k}(x)|^2\big)\dx - \frac{a^*}{2}\int_{\R^2}(|u_{1k}(x)|^2+|u_{2k}(x)|^2)^2\dx \\
  & + \int_{\R^2} V_1(x)|u_{1k}(x)|^2\dx + \int_{\R^2} V_2(x)|u_{2k}(x)|^2\dx \\
  & + \frac{a^*-a_k}{2}\int_{\R^2}|u_{1k}(x)|^4\dx+\frac{a^*-b}{2}\int_{\R^2}|u_{2k}(x)|^4\dx \\
  & + (a^* - \beta _k) \int_{\R^2} |u_{1k}(x)|^2|u_{2k}(x)|^2\dx.
\end{split}
\end{equation}
Applying \cite[Theorem 1.2]{GLWZ}, we have
$e(a_k,b,\beta _k)\to0$ as $k\to\infty$ by choosing $k$ large enough that $a^*> \beta _k$.
It then follows from (\ref{ineq:GN}) and \eqref{exp2:ea} that
\begin{equation}\label{lim.b1.V}
  \lim\limits_{k\to\infty}\int_{\R^2}V_1(x)u_{1k}^2\dx=
  \lim\limits_{k\to\infty}\int_{\R^2}V_2(x)u_{2k}^2\dx=0,
\end{equation}
and
\begin{equation}\label{lim.b1.multi}
  \lim\limits_{k\to\infty}\int_{\R^2}u_{2k}^4\dx=\lim\limits_{k\to\infty}\int_{\R^2}u_{1k}^2u_{2k}^2\dx=0.
\end{equation}
Recall that
\begin{equation}\label{exp1:ea}
\begin{split}
  e(a_k,b,\beta _k) =
         &\int_{\R^2} |\nabla u_{1k}|^2\dx -\frac{a_k}{2}\int_{\R^2} |u_{1k}|^4\dx + \int_{\R^2} V_1(x)|u_{1k}|^2\dx \\
         &+\int_{\R^2} |\nabla u_{2k}|^2\dx -\frac{b}{2}\int_{\R^2} |u_{2k}|^4\dx + \int_{\R^2} V_2(x)|u_{2k}|^2\dx \\
         &-\beta _k \int_{\R^2} |u_{1k}|^2|u_{2k}|^2\dx.
\end{split}
\end{equation}
Hence, using the Gagliardo-Nirenberg inequality \eqref{ineq:GNQ}, one can deduce from (\ref{lim.b1.V})-\eqref{exp1:ea} that
\begin{equation*}\label{lim:b1.nabla1}
 \lim\limits_{k\to\infty}\int_{\R^2}|\nabla u_{2k}|^2\dx =0\,\,\,\text{and}\,\,\, \lim\limits_{k\to\infty}\Big(1-\frac{a_k}{a^*}\|u_{1k}\|_2^2\Big)\int_{\R^2}|\nabla u_{1k}|^2\dx =0.
\end{equation*}

On the other hand, the argument of proving \cite[Lemma 3.1(1)]{GLWZ} gives that $\int_{\R^2}(|\nabla u_{1k}|^2+|\nabla u_{2k}|^2)\dx\to+\infty$ as $k\to\infty$. Together with the fact $\int_{\R^2}(| u_{1k}|^2+| u_{2k}|^2)\dx=1$, one can derive from the above two equalities that
\begin{equation}\label{lim:b1.kinetic}
   \lim\limits_{k\to\infty}\int_{\R^2}|\nabla u_{1k}|^2\dx=+\infty, \,\
   \lim\limits_{k\to\infty}\int_{\R^2}|u_{1k}|^2\dx=1,
\end{equation}
and
\begin{equation}\label{lim:b1.L2}
   \lim\limits_{k\to\infty}\int_{\R^2}|\nabla u_{2k}|^2\dx=\lim\limits_{k\to\infty}\int_{\R^2}| u_{2k}|^2\dx=0.
\end{equation}
Moreover, using \eqref{lim.b1.V}-\eqref{lim:b1.L2}, we deduce  from \eqref{exp1:ea} that
\begin{equation}\label{lim.b1.GN}
\lim\limits_{k\to\infty}\frac{\int_{\R^2}|\nabla u_{1k}|^2\dx}
                                  {\int_{\R^2}| u_{1k}|^4\dx}=\frac{a^*}{2}.
\end{equation}
Following above estimates, we now address the proof of Theorem \ref{Th:b1}.

\begin{proof}[\bf Proof of Theorem \ref{Th:b1}.]
Setting $\bar{\varepsilon}_{k}:=\big(\int_{\R^2}|\nabla u_{1k}(x)|^2\dx\big)^{-\frac{1}{2}}>0$, it then follows from \eqref{lim:b1.kinetic} that  $\bar{\varepsilon}_{k}\to0$ as $k\to\infty$. Denote
\begin{equation}\label{eq2.6}
\bar{w}_{ik}(x):=\bar{\varepsilon}_{k} u_{ik}(\bar{\varepsilon}_{k}x+x_{1k}), \ i=1,2,\end{equation}
where $x_{1k}$ is a global   maximum point of $u_{1k}$. Since $(u_{1k},u_{2k})$ satisfies the system \eqref{equ:CGPS}, $(\bar w_{1k},\bar w_{2k})$ satisfies
\begin{equation}\label{eq2.7}
\begin{cases}
 -\Delta \bar w_{1k} +\bar \varepsilon_k^2V_1(\bar \varepsilon_k x+x_{1k})\bar w_{1k}
 =\bar \varepsilon_k^2\mu_{k} \bar w_{1k} +a_k\bar w_{1k} ^3+\beta _k \bar w_{2k} ^2\bar w_{1k}    \,\, \mbox{in}\,\,  \R^2,\\
 -\Delta \bar w_{2k} +\bar \varepsilon_k^2V_2(\bar \varepsilon_k x+x_{1k})\bar w_{2k}
 =\bar\varepsilon_k^2\mu_{k} \bar w_{2k} +b\bar w_{2k} ^3+\beta _k \bar w_{1k} ^2\bar w_{2k} \ \   \,\, \mbox{in}\,\,  \R^2,
\end{cases}
\end{equation}
where $\mu_k\in\R$ is a suitable Lagrange multiplier.
Note from \eqref{lim:b1.kinetic} and \eqref{lim:b1.L2} that for any sequence $\{a_k\}$ with $a_k\nearrow a^*$ as $k\to\infty$, $\bar w_{1k}$ is bounded uniformly in $H^1(\R^2)$ and $\bar w_{2k}\to 0$ in $H^1(\R^2)$ as  $k\to\infty$.  By the argument of proving (4.6) and (4.7) in \cite{GLWZ}, one can also obtain that $\bar w_{ik}$ and $\nabla \bar w_{ik}$ decay exponentially as $|x|\to\infty$ for $i=1, 2.$
Using the standard elliptic regularity theory, one can further derive from (\ref{eq2.7})  that
 \begin{equation}\label{eq2.9}
 \bar w_{2k}(x)\overset{k}\to0\text{ in }L^{\infty}(\R^2). \end{equation}
Therefore, the system (\ref{eq2.7}) must degenerate into a single equation of the form
\begin{equation}\label{equ:a1.u1}
-\Delta \bar w_{1k} +\bar \varepsilon_k^2V_1(\bar \varepsilon_k x+x_{1k})\bar w_{1k}
 =\big(\bar \varepsilon_k^2\mu_{k} +o(1)\big)\bar w_{1k} +a_k\bar w_{1k} ^3   \,\ \mbox{in}\,\  \R^2
\end{equation}
as $k \to\infty$.
Following the proof of \cite[Theorem 1.2]{GWZZ} (see also  \cite[Theorem 1.3]{GLWZ}), one can conclude from \eqref{lim.b1.V} and \eqref{lim.b1.GN} that, passing to a subsequence if necessary, $\bar{w}_{1k}$ satisfies
\begin{equation}\label{lim:a1.w1}
\bar{w}_{1k}(x)
   \xrightarrow{k}\frac{w(x)}{\|w\|_2}
\,\,\,\, \text{strongly in} \,\,\,H^1(\R^2),
\end{equation}
and $x_{1k}$ is the unique  maximum point of $u_{1k}$.

In order to determine the  convergence rate $\bar{\varepsilon}_k>0$, motivated by \cite{GLW}-\cite{GZZ2}, we next analyze a refined estimate of the energy $e(a_k,b,\beta _k)$ as $k\to\infty$.
Specifically, here we claim that
\begin{equation}\label{lim:a.e.rate}
\lim\limits_{k\to\infty}\frac{e(a_k,b,\beta _k)}{(a^*-a_k)^\frac{p_1}{p_1+2}}
=\frac{\lambda_1^2}{a^*}\frac{p_1+2}{p_1}.
\end{equation}
Actually,  by taking the following test function
\begin{equation*}\label{def:b.trial}
  \phi_1(x)=\frac{\tau}{\| w\|_2} w(\tau x-y_0 ) \ \ \text { and }\ \
  \phi_2(x)=0,
\end{equation*}
where $\tau=\lambda_1(a^*-a_k)^\frac{-1}{p_1+2}>0$, $\lambda_1>0$ is defined in \eqref{def:a1.eps} and $y_0$ is the unique  critical point of $H_1(y):=\int_{\R^2}V_1(x+y)w^2(x)\dx$,
the calculations yield  the following upper bound
\begin{equation}\label{psup:b1.e}
e(a_k,b,\beta _k)\leq E_{a_k,b,\beta _k}(\phi_1,\phi_2)
  =\frac{\lambda_1^2}{a^*}\frac{p_1+2}{p_1}(a^*-a_k)^\frac{p_1}{p_1+2}\ \ \text{as} \ \ k\to\infty.
\end{equation}
On the other hand, let $(u_{1k},u_{2k})$ be a nonnegative minimizer of $e(a_k,b,\beta _k)$ as $k\to\infty$.
It follows from (\ref{exp2:ea}) and \eqref{lim.b1.GN} that
\begin{equation*}\label{estsub:b1.e}
\begin{split}
  e(a_k,b,\beta _k)
  \geq& \frac{a^*-a_k}{2}\int_{\R^2}| u_{1k}|^4\dx + \int_{\R^2} V_1(x)|u_{1k}|^2\dx\\
  =&\frac{a^*-a_k}{a^*}(\bar{\varepsilon}_k)^{-2} + (\bar{\varepsilon}_k)^{p_1}\int_{\R^2} V_1\Big(x+\frac{x_{1k}}{\bar{\varepsilon}_k}\Big)\bar{w}_{1k}^2\dx.
\end{split}
\end{equation*}
By the argument of proving (3.35) in \cite{GLWZ}, it then yields from  above that
\begin{equation}\label{estsub:a1.e}
\begin{split}
\liminf\limits_{k\to\infty}\frac{e(a_k,b,\beta _k)}{(a^*-a_k)^\frac{p_1}{p_1+2}}
\geq\frac{\lambda_1^2}{a^*} \frac{p_1+2}{p_1},
\end{split}
\end{equation}
where the equality holds if and only if
\begin{equation}\label{estsub:a1.e.y}
\lim\limits_{k\to\infty}\frac{x_{1k}}{\bar{\varepsilon}_k}=y_0,\,\,\,\text{where $y_0\in\R^2$ is defined in \eqref{cond:Vcritical.u1}},
\end{equation}
and
\begin{equation}\label{estsub:a1.eps}
\lim_{k\to\infty}\bar{\varepsilon}_{k}/\varepsilon_k =1, \,\,\,\text{where $\varepsilon_k=\frac{1}{\lambda_1}(a^*-a_k)^\frac{1}{p_1+2}>0$ is defined in \eqref{def:a1.eps}}.
\end{equation}
Therefore, we conclude  \eqref{lim:a.e.rate} from \eqref{psup:b1.e} and \eqref{estsub:a1.e}.

The above proof of \eqref{lim:a.e.rate}  implies that the equality of \eqref{estsub:a1.e} holds true. This further implies that both \eqref{estsub:a1.e.y} and \eqref{estsub:a1.eps} are true, and therefore \eqref{lim:a1.z.rate} follows.
Furthermore, we obtain from \eqref{lim:a1.w1} and \eqref{estsub:a1.eps} that
\begin{equation*}
   \lim\limits_{k\to\infty}\sqrt{a^*} \varepsilon_k  u_{1k}(\varepsilon_kx+x_{1k})=w(x)\,\, \text{strongly in} \,\,H^1(\R^2).
\end{equation*}
Since we have as before that $w(x)$ and $\bar{w}_{1k}$ decay exponentially as $|x|\to\infty$, the standard elliptic regularity theory yields that the first limit of \eqref{lim:a1.u.rate} holds uniformly in $\R^2$ (see \cite[Lemma 4.9]{M} for similar arguments).

The rest is to prove that $u_{2k}(x)\equiv 0$ in $\R^2$ when $k>0$ is large enough.
On the contrary, suppose this is false.
Let $y_k$ be a global maximum point of $u_{2k}$, and set $\bar{u}_{2k}(x):=\frac{1}{\delta_k}u_{2k}(\varepsilon_kx+y_k)$, where $\delta_k:=\|u_{2k}\|_{\infty}$  and $\varepsilon_k>0$ is given in \eqref{def:a1.eps}. Then $\delta_k>0$ and $\bar{u}_{2k}(x)$ satisfies
\begin{equation}\label{equ:a1.w2}
\begin{split}
  &-\Delta \bar{u}_{2k}+\varepsilon_k^2V_2(\varepsilon_kx+y_k)\bar{u}_{2k}\\
  =&\mu_k\varepsilon_k^2\bar{u}_{2k}+b\delta^2_k\varepsilon_k^2\bar{u}_{2k}^3+\beta _k  {w}_{1k}^2\Big(x+\frac{y_k-x_{1k}}{\varepsilon_k}\Big) \bar{u}_{2k}\,\, \text{in} \,\,\R^2,
\end{split}
\end{equation}
where ${w}_{1k}(x):=\varepsilon_k  u_{1k}(\varepsilon_kx+x_{1k})$  and $x_{1k}$ is the unique maximum point of $u_{1k}$.
Note from (\ref{eq2.6}), (\ref{eq2.9}) and \eqref{estsub:a1.eps} that
\begin{equation}\label{eq2.18}
\delta_k\eps_k\to0\,\,\,\text{as}\,\,\,k\to\infty.
\end{equation}
It also follows from \eqref{exp2:ea} and \eqref{equ:a1.u1} that
\begin{equation}\label{lim:a1.mu}
\varepsilon_k^2\mu_k=\varepsilon_k^2\big(e(a_k,b,\beta _k)-o(1)\big)-\frac{a_k}{2}\int_{\R^2}w_{1k}^4\dx\to-1\,\,\,\text{as}\,\,\,k\to\infty.
\end{equation}
Since the origin is a global maximum point of $\bar{u}_{2k}$ and $\bar{u}_{2k}(0)=\frac{u_{2k}(y_k)}{\|u_{2k}\|_{\infty}}=1$, we then derive from  \eqref{equ:a1.w2} that $ {w}_{1k}^2(\frac{y_k-x_{1k}}{\varepsilon_k})\geq \frac{1}{2\beta _*}$. Since ${w}_{1k}$ decays exponentially as $|x|\to\infty$, applying the maximum principle to (\ref{equ:a1.w2}) then gives that $\{\frac{y_k-x_{1k}}{\varepsilon_k}\}$ is bounded  uniformly in $k$, where (\ref{lim:a1.mu}) is also used.
Thus,  passing to a subsequence if necessary, one can get that
\begin{equation}\label{lim:a1.xy}
  \lim\limits_{k\to\infty}\frac{y_k-x_{1k}}{\varepsilon_k}=y^0\,\,\,\text{for some}\,\,\,y^0\in\R^2.
\end{equation}
Furthermore, the standard elliptic regularity implies that $\|\bar{u}_{2k}\|_{C^{2,\alpha}_{loc}{(\R^2)}}\leq C$ for some $\alpha\in(0,1)$, where the constant $C>0$ is independent of $k$.
Then there exist a subsequence of $\{\bar{u}_{2k}\}$ (still denoted by $\{\bar{u}_{2k}\}$) and some $\bar{u}_{20}\in C^2_{loc}{(\R^2)}$ such that $\bar{u}_{2k}\to \bar{u}_{20}$ in $C^2_{loc}{(\R^2)}$ as $k\to\infty$. Especially, we have
\begin{equation}\label{lim:a1.u2.y}
\bar{u}_{20}(y^0)=\lim\limits_{k\to\infty}\bar{u}_{2k}\Big(\frac{y_k-x_{1k}}{\varepsilon_k}\Big)=1.
\end{equation}

On the other hand,
 one can derive from \eqref{lim:a1.w1} and (\ref{eq2.18})-\eqref{lim:a1.xy} that $\bar{u}_{20}$ satisfies
\begin{equation}\label{equ:2:22}
  -\Delta \bar{u}_{20}+\bar{u}_{20}-\frac{\beta _*}{a^*}w^2(x+y^0)\bar{u}_{20} =0\,\, \text{in} \,\,\R^2,
\end{equation}
where $0<\frac{\beta _*}{a^*}<1$ and $w$ is the unique positive solution of \eqref{equ:w}.
However, since it follows from  \cite[Lemma 4.1]{Wei96} that
$$\int_{\R^2}|\nabla u|^2\dx+\int_{\R^2}u^2\dx\geq\int_{\R^2}w^2u^2\dx \ \text{ for any }\  u\in H^1(\R^2),$$
we then reduce from (\ref{equ:2:22}) that
\begin{equation*}\label{equiv.a1.w20}
\bar{u}_{20}\equiv0 \ \text{ in }\,   \R^2,
\end{equation*}
which however contradicts to \eqref{lim:a1.u2.y}.
Therefore, we conclude that $u_{2k}(x)\equiv0$ in $\R^2$ when $k>0$ is large enough. This completes the proof of Theorem \ref{Th:b1}.
\end{proof}

\section{Limit Behavior of Nonnegative Minimizers: $a^*\le\beta <\beta ^*$}

In this section we shall prove Theorem \ref{thm:1.1} for the case where $(a_k, b, \beta_k)$ satisfies \eqref{con:a1beta.ab}. As a byproduct, we then complete the proof of Theorem \ref{thm:1.2}. We begin with the following lemma under the general assumption \eqref{cond:V.1}.

\begin{lemma}\label{lem:betaa1.u}
Suppose $V_1 (x)$ and $V_2 (x)$ satisfy \eqref{cond:V.1}.
Let $ (u_{1k},u_{2k})$ be a nonnegative minimizer of $e(a_k, b, \beta _k)$ satisfying
\begin{equation}\label{cond:-Vcritical} 0<b<a^*, \, \  a^*\leq \beta _k<\beta_k^*=a^*+\sqrt{(a^*-a_k)(a^*-b)},
\end{equation}
where $a_k\nearrow a^*$ as $k\to\infty$.
Then we have
\begin{itemize}
  \item [\rm(i).] $ (u_{1k},u_{2k})$ satisfies
\begin{equation}\label{lim.betaa1.V}
  \lim\limits_{k\to\infty}\int_{\R^2}V_1(x) u_{1k} ^2\dx=
  \lim\limits_{k\to\infty}\int_{\R^2}V_2(x) u_{2k} ^2\dx=0,
\end{equation}
\begin{equation}\label{lim:betaa1.u14}
   \lim\limits_{k\to\infty}\int_{\R^2}| u_{1k} |^4\dx=\infty.
\end{equation}

  \item [\rm(ii).] $ (u_{1k},u_{2k})$ also satisfies
  \begin{equation}\label{lim:betaa1.u2:u1.mult}
   \lim\limits_{k\to\infty}
   \frac{\int_{\R^2}|u_{2k} |^4\dx}{\int_{\R^2}|u_{1k} |^4\dx}=0,\,\,\,
   \lim\limits_{k\to\infty}
   \frac{\int_{\R^2}|u_{1k} |^2|u_{2k} |^2\dx}{\int_{\R^2}| u_{1k} |^4\dx}=0,
\end{equation}
\begin{equation}\label{lim:betaa1.u2}
   \lim\limits_{k\to\infty}
   \frac{\int_{\R^2}|\nabla u_{2k} |^2\dx}{\int_{\R^2}| u_{1k} |^4\dx}=
   \lim\limits_{k\to\infty}
   \int_{\R^2}| u_{2k} |^2\dx=0,
\end{equation}
and
\begin{equation}\label{lim:betaa1.u1}
   \lim\limits_{k\to\infty}
   \frac{\int_{\R^2}|\nabla u_{1k} |^2\dx}{\int_{\R^2}| u_{1k} |^4\dx}=\frac{a^*}{2}, \,\
\lim\limits_{k\to\infty}\int_{\R^2}| u_{1k} |^2\dx=1.
\end{equation}
\end{itemize}
\end{lemma}

\begin{proof}   (i).
We first note that $e(a_k,b,\beta_k)$   can be rewritten as
\begin{equation}\label{exp3:ea}
\begin{split}
  e(a_k,b,\beta_k) =
  & \int_{\R^2} \big(|\nabla u_{1k} |^2 + |\nabla u_{2k} |^2\big)\dx - \frac{a^*}{2}\int_{\R^2}(|u_{1k} |^2+|u_{2k} |^2)^2\dx \\
  & + \int_{\R^2} V_1(x) |u_{1k} |^2\dx + \int_{\R^2} V_2(x) |u_{2k} |^2\dx \\
  & + \frac{1}{2}\int_{\R^2}(\sqrt{a^*-a_k}|u_{1k} |^2-\sqrt{a^*-b}|u_{2k} |^2)^2\dx \\
  & + (\beta^*_k - \beta_k) \int_{\R^2} |u_{1k} |^2|u_{2k} |^2\dx.
\end{split}
\end{equation}
From \cite[Theorem 1.2]{GLWZ}, one can get that  $e(a_k,b,\beta_k)\to0$ as $k\to\infty$, and  \eqref{lim.betaa1.V} hence follows directly from \eqref{ineq:GN} and \eqref{exp3:ea}.
As for \eqref{lim:betaa1.u14}, we   prove it by  contradiction.
Suppose that $\int_{\R^2}| u_{1k} |^4\dx\leq C$ uniformly for all $k$.
By the following H\"{o}lder inequality
\begin{equation}\label{ineq:b1beta.mult}
\int_{\R^2}|u_{1k} |^2|u_{2k} |^2\dx
\leq\Big(\int_{\R^2}|u_{1k} |^4|\dx\Big)^\frac{1}{2}\Big(\int_{\R^2}|u_{2k} |^4\dx\Big)^\frac{1}{2},
\end{equation}
we then deduce from \eqref{exp3:ea} that
\begin{equation}\label{lim:betaa1.minus.2}
\begin{split}
&\lim\limits_{k\to\infty}\big(\sqrt{a^*-a_k}\|u_{1k}\|_{L^4(\R^2)}^2-\sqrt{a^*-b}\|u_{2k}\|_{L^4(\R^2)}^2\big)^2\\
\leq&\lim\limits_{k\to\infty}\int_{\R^2}(\sqrt{a^*-a_k}|u_{1k} |^2-\sqrt{a^*-b}|u_{2k} |^2)^2\dx=0,
\end{split}
\end{equation}
which implies that $\lim\limits_{k\to\infty}\int_{\R^2}|u_{2k} |^4\dx=0$, and thus $\lim\limits_{k\to\infty}\int_{\R^2}|u_{1k} |^2|u_{2k} |^2\dx=0$.
Following this, one can derive from \eqref{exp1:ea} that
$$\lim\limits_{k\to\infty}\Big(\int_{\R^2}\big(|\nabla u_{1k} |^2+|\nabla u_{2k} |^2\big)\dx-\frac{a_k}{2}\int_{\R^2}| u_{1k} |^4\dx\Big)=0,$$
which then implies that $\int_{\R^2}\big(|\nabla u_{1k} |^2+|\nabla u_{2k} |^2\big)\dx\leq C$ uniformly for all $k$.
On the other hand, similar to \cite[Lemma 3.1(1)]{GLWZ}, one can verify  that $\int_{\R^2}\big(|\nabla u_{1k} |^2+|\nabla u_{2k} |^2\big)\dx\to\infty$  as $k\to\infty$.
This is however a contradiction, and therefore \eqref{lim:betaa1.u14} holds true.

(ii).
It directly follows from \eqref{lim:betaa1.minus.2} that the first equality of \eqref{lim:betaa1.u2:u1.mult} holds,
and then the second one can be obtained by using the  H\"{o}lder inequality (\ref{ineq:b1beta.mult}).
As for \eqref{lim:betaa1.u2} and \eqref{lim:betaa1.u1}, we note from (\ref{lim:betaa1.u14}) that
\[
\frac{e(a_k,b,\beta_k)}{\int_{\R^2}| u_{1k} |^4\dx}\to 0\ \, \text{as}\, \ k\to\infty.
\]
Applying (\ref{lim.betaa1.V})-(\ref{lim:betaa1.u2:u1.mult}), it then follows from \eqref{exp1:ea} and above  that
\begin{equation}\label{lim:b1beta.nabla:u1}
  \lim\limits_{k\to\infty}
  \Big( \frac{\int_{\R^2}|\nabla u_{1k} |^2\dx}{\int_{\R^2}| u_{1k} |^4\dx}
  +\frac{\int_{\R^2}|\nabla u_{2k} |^2\dx}{\int_{\R^2}| u_{1k} |^4\dx}\Big)=\frac{a^*}{2}.
\end{equation}
On the other hand, one can obtain from \eqref{ineq:GNQ} that
\[\frac{\int_{\R^2}|\nabla u_{1k} |^2\dx}{\int_{\R^2}| u_{1k} |^4\dx}
\ge \frac{\int_{\R^2}|\nabla u_{1k} |^2\dx\int_{\R^2}| u_{1k} |^2\dx}{\int_{\R^2}| u_{1k} |^4\dx}\geq\frac{a^*}{2},\]
since $\|u_{1k}\|_2^2+\|u_{2k}\|_2^2=1$. Thus, \eqref{lim:betaa1.u2} and \eqref{lim:betaa1.u1} follow from \eqref{lim:b1beta.nabla:u1} and the above inequality, and the lemma is proved.
\end{proof}

For any sequence $\{a_k\}$ satisfying $a_k\nearrow a^*$ as $k\to\infty$,  define
\begin{equation}\label{def:betaa1.eps}
\bar{\varepsilon}_{k} :=\Big(\int_{\R^2}|u_{1k}(x)|^4\dx\Big)^{-\frac{1}{2}}>0,
\end{equation}
and by (\ref{lim:betaa1.u14}) we then have $\bar{\varepsilon}_{k}\to0$ as $k\to\infty$.
From \eqref{lim:betaa1.u2}, we know that $\bar{\varepsilon}_{k} u_{2k}(\bar{\varepsilon}_{k}x) \to 0$ strongly in $H^1(\R^2)$ as $k\to\infty$.
Similar to \cite[Theorem 1.2]{GWZZ} (see also \cite[Theorem 1.3]{GLWZ}), one can obtain from Lemma \ref{lem:betaa1.u} that, passing to a subsequence if necessary, $\bar{w}_{1k}$ satisfies
\begin{equation}\label{lim:betaa1.w1}
\bar{w}_{1k}(x):=\bar{\varepsilon}_{k} u_{1k}(\bar{\varepsilon}_{k}x+x_{1k})
   \xrightarrow{k}\sqrt{\frac{1}{2}}w\Big(\sqrt{\frac{a^*}{2}}x\Big)
\,\,\, \text{strongly in} \,\,\,H^1(\R^2),
\end{equation}
where $x_{1k}$ is the unique maximum point of $u_{1k}$. 
Under some further assumptions on the trapping potentials, the following proposition gives the explicit limit behavior of $u_{1k}$ as $k\to\infty$.

\begin{proposition}\label{Prop:a1beta.3}
{\em Suppose $V_i(x)\in C^2(\R^2)$ is homogeneous of degree $p_i$ and satisfies  \eqref{cond:V.1} and \eqref{cond:Vcritical.u1}, where $i=1,2$ and $2\le p_1\leq p_2$.
Let $ (u_{1k},u_{2k})$ be a nonnegative minimizer of $e(a_k, b, \beta_k)$ satisfying (\ref{cond:-Vcritical}).
Then there exists a subsequence, still denoted by $\{a_k\}$, of $\{a_k\}$   such that}
\begin{equation}\label{lim:betaa1.u.rate}
 \sqrt{a^*}\,\varepsilon_{k}  u_{1k}(\varepsilon_{k}x+x_{1k})\to  w(x) \ \ and\ \ \varepsilon_{k} u_{2{k}}(\varepsilon_{k}x)\to 0  \  \ as \ \ k\to\infty
\end{equation}
{\em uniformly in $\R^2$,   where  $x_{1k}$ is the unique maximum point of $u_{1k}$ satisfying}
\begin{equation}\label{lim:betaa1.z.rate}
  \lim\limits_{k\to\infty}\frac{x_{1k}}{\varepsilon_k}=y_0,
\end{equation}
{\em and}
\begin{equation}\label{def:betaa1.eps.rate}
  \varepsilon_k:=\frac{1}{\lambda}\Big[(a^*-a_k)(a^*-b)-(\beta_k-a^*)^2\Big]^\frac{1}{2+p_1}>0,\,\  \lambda =\Big[\frac{p_1}{2}(a^*-b)H_1(y_0)\Big]^\frac{1}{2+p_1}
\end{equation}
{\em for $y_0\in\R^2$ given by \eqref{cond:Vcritical.u1}.
Moreover, $\bar{u}_{1k}$ decays exponentially in the sense that}
 \begin{equation}\label{4:conexp1}
\bar{ u}_{1k}\le Ce^{-\frac{1}{2}|x|} \,\ in \,\ \R^2,
\end{equation}
{\em and}
\begin{equation}\label{4:conexp2}
|\nabla \bar{u}_{1k}|\le Ce^{-\frac{|x|}{4}} \,\ in\,\ \R^2,
\end{equation}
{\em where the constant $C>0$ is independent of $k$.}
\end{proposition}

\begin{proof}
We first prove that the energy $e(a_k,b,\beta_k)$ satisfies
\begin{equation}\label{lim:a1beta.e.rate}
\begin{split}
\lim\limits_{k\to\infty}\frac{e(a_k,b,\beta_k)}{\big[(a^*-a_k)(a^*-b)-(\beta_k-a^*)^2\big]^\frac{p_1 }{p_1 +2}}
=\frac{\lambda^2}{a^*(a^*-b)}\frac{p_1 +2}{p_1},
\end{split}
\end{equation}
where $(a_k,b,\beta_k)$ satisfies (\ref{cond:-Vcritical}) and  $\lambda>0$ is given in \eqref{def:betaa1.eps.rate}.

To derive (\ref{lim:a1beta.e.rate}), we take a test function of the form
\begin{equation}\label{def:betaa1.trial}
\begin{cases}
  \phi_1(x)=A_{}\frac{\tau}{\| w\|_2} w(\tau x-y_0 ), \\[2mm]
  \phi_2(x)=A_{}\frac{\sqrt{\beta_k-a^*}}{\sqrt{a^*-b}}\frac{\tau}{\| w\|_2} w(\tau x-y_0 ),
\end{cases}
\end{equation}
where $y_0\in\R^2$ is given by \eqref{cond:Vcritical.u1}, $\tau>0$ and $A_{}>0$ is chosen so that $\int_{\R^2} (\phi_1^2+\phi_2^2)\dx=1$.
One can check that $A_{}=\Big(\frac{a^*-b}{\beta_k-b}\Big)^\frac{1}{2}\le 1$, since
$(a_k, b, \beta_k)$ satisfies (\ref{cond:-Vcritical}).
Using \eqref{ide:w} and \eqref{def:h.homo}, some calculations yield that as $\tau \to \infty$,
\begin{equation}\label{estsup.betaa1.e.nabla}
\begin{split}
&\int_{\R ^2} \big(|\nabla \phi_1|^2+|\nabla \phi_2|^2\big)\dx-\int_{\R ^2} \big(\frac{a_k}{2}|\phi_1|^4+\frac{b}{2}|\phi_2|^4 +\beta _k|\phi_1|^2|\phi_2|^2\big) \dx\\
=&\tau^2 -\frac{A^4_{}}{a^*}\tau^2\Big[a_k+b\Big(\frac{\beta_k-a^*}{a^*-b}\Big)^2+2\beta _k\frac{\beta_k-a^*}{a^*-b}\Big]\\
=&\frac{A^4_{}}{a^*}\tau^2\Big[(a^*-a_k)+(a^*-b)\Big(\frac{\beta_k-a^*}{a^*-b}\Big)^2-2(\beta_k-a^*) \frac{\beta_k-a^*}{a^*-b}\Big]\\
=&\frac{A^4_{}}{a^*}(a^*-a_k)\Big[1-\frac{(\beta_k-a^*)^2}{(a^*-a_k)(a^*-b)}\Big]\tau^2 \\
=&\frac{1}{a^*}\frac{a^*-b}{(\beta_k-b)^2}\big[(a^*-a_k)(a^*-b)-(\beta_k-a^*)^2\big]\tau^2,
\end{split}
\end{equation}
and
\begin{equation}\label{estsup:betaa1.e.V}
\begin{split}
&\int_{\R ^2}\big( V_1(x)|\phi_1|^2+V_2(x)|\phi_2|^2\big) \dx\\
=&\frac{A^2_{}}{a^*}\int_{\R^2}V_1\big(\frac{x+y_0 }{\tau}\big)w^2\dx+\frac{A^2_{}}{a^*}\frac{\beta_k-a^*}{a^*-b}\int_{\R^2}V_2\big(\frac{x+y_0 }{\tau}\big)w^2\dx\\
=&\Big(\frac{1}{a^*}\frac{2}{p_1 }\frac{1}{\beta_k-b}\lambda^{p_1 +2}+o(1)\Big)\tau^{-p_1 }, \\
\end{split}
\end{equation}
where $\lambda>0$ is as in \eqref{def:betaa1.eps.rate}.
Thus, by taking
$$\tau=\lambda\Big(\frac{\beta_k-b}{a^*-b}\Big)^\frac{1}{p_1 +2}\Big[\frac{1}{(a^*-a_k)(a^*-b)-(\beta_k-a^*)^2}\Big]^\frac{1}{p_1 +2},$$
we derive from \eqref{estsup.betaa1.e.nabla} and \eqref{estsup:betaa1.e.V} that
\begin{equation*}
\begin{split}
 &e(a_k,b,\beta_k)
\leq E_{a_k,b,\beta_k}(\phi_1,\phi_2)\\
\leq&\frac{1}{a^*}\frac{a^*-b}{(\beta_k-b)^2}\big[(a^*-a_k)(a^*-b)-(\beta_k-a^*)^2\big]\tau^2
+\Big(\frac{1}{a^*}\frac{2}{p_1 }\frac{1}{\beta_k-b}\lambda^{p_1 +2}+o(1)\Big)\tau^{-p_1 }\\
\leq&\Big[\frac{\lambda^2}{a^*(a^*-b)}\frac{p_1 +2}{p_1}+o(1)\Big]
\big[(a^*-a_k)(a^*-b)-(\beta_k-a^*)^2\big]^\frac{p_1 }{p_1 +2}\quad \text{as}\ \ k\to\infty.
\end{split}
\end{equation*}
Hence, this estimate implies that
\begin{equation}\label{estsup:betaa1.e}
\begin{split}
\limsup\limits_{k\to\infty}\frac{e(a_k,b,\beta_k)}{\big[(a^*-a_k)(a^*-b)-(\beta_k-a^*)^2\big]^\frac{p_1 }{p_1 +2}}
\leq\frac{\lambda^2}{a^*(a^*-b)}\frac{p_1 +2}{p_1}.
\end{split}
\end{equation}

Let $(u_{1k},u_{2k})$ be now a nonnegative minimizer of $e(a_k,b,\beta_k)$, where $(a_k,b,\beta_k)$ satisfies (\ref{cond:-Vcritical}). Since $a^*\le \beta_k\le  \beta_k^*$, we then have
\begin{equation}\label{L:estsub:betaa1.gn}
\begin{split}
&\frac{a^*-a_k}{2}\int_{\R^2}|u_{1k}|^4\dx +\frac{a^*-b}{2}\int_{\R^2}|u_{2k}|^4\dx\\
&  + (a^*-\beta_k) \int_{\R^2} |u_{1k}|^2|u_{2k}|^2\dx \\
\geq&\frac{a^*-a_k}{2}\int_{\R^2}|u_{1k}|^4\dx
  \Big[1+\frac{a^*-b}{a^*-a_k}\frac{\int_{\R^2}|u_{2k}|^4 \dx}{\int_{\R^2}|u_{1k}|^4\dx}
        \\
        &\qquad\qquad\quad -2\frac{\beta_k-a^*}{a^*-a_k}\frac{(\int_{\R^2}|u_{1k}|^4\dx)^\frac{1}{2}
        (\int_{\R^2}|u_{2k}|^4\dx)^\frac{1}{2}}{\int_{\R^2}|u_{1k}|^4\dx}
  \Big]\\
\geq&\frac{a^*-a_k}{2}\int_{\R^2}|u_{1k}|^4\dx
  \Big\{1-\frac{(\beta_k-a^*)^2}{(a^*-a_k)(a^*-b)}\\
  &\qquad\quad \qquad\quad +\frac{a^*-b}{a^*-a_k}\Big[\Big(\frac{\int_{\R^2}|u_{2k}|^4\dx}{\int_{\R^2}|u_{1k}|^4\dx}\Big)^\frac{1}{2}-\frac{\beta_k-a^*}{a^*-b}\Big]^2
\Big\} \\
  \geq &\Big[1-\frac{(\beta_k-a^*)^2}{(a^*-a_k)(a^*-b)}\Big]\frac{a^*-a_k}{2}\int_{\R^2}|u_{1k}|^4\dx\\
= &\frac{(a^*-a_k)(a^*-b)-(\beta_k-a^*)^2}{2(a^*-b)}\int_{\R^2}|u_{1k}|^4\dx \\
\end{split}
\end{equation}
as $k\to\infty$, where the first inequality follows from the H\"{o}lder inequality (\ref{ineq:b1beta.mult}).
On the other hand, similar to  proving (3.33) in \cite{GLWZ}, one can verify from (\ref{lim:betaa1.w1}) that
\begin{equation}\label{L:estsub:betaa1.e.V}
\begin{split}
\liminf_{k\to\infty}\bar{\varepsilon}_k^{-p_1}\int_{\R^2} V_1(x)|u_{1k}|^2\dx
=&\liminf_{k\to\infty}\int_{\R^2} V_1\big(x+\frac{x_{1k}}{\bar{\varepsilon}_k}\big)|\bar w_{1k}|^2\dx\\
\geq&\frac{1}{a^*}\int_{\R^2} V_1\Big(\sqrt{\frac{2}{a^*}}x+y^{10}\Big)|w|^2\dx\\
=&\frac{1}{a^*} \Big({\frac{2}{a^*}}\Big)^\frac{p_1}{2}\int_{\R^2} V_1\Big(x+\sqrt{\frac{a^*}{2}}y^{10}\Big)|w|^2\dx\\
\geq& \Big({\frac{2}{a^*}}\Big)^\frac{p_1+2}{2}\frac{\lambda^{2+p_1}}{(a^*-b)p_1},
\end{split}
\end{equation}
where $\bar{\varepsilon}_k>0$ is defined by \eqref{def:betaa1.eps},  $\lambda>0$ is given in \eqref{def:betaa1.eps.rate} and $y^{10}:=\lim\limits_{k\to\infty}\frac{x_{1k}}{\bar{\varepsilon}_k}$.
Note that  the last equality of (\ref{L:estsub:betaa1.e.V}) holds, if and only if
\begin{equation}\label{estsub:betaa1.e.y}
\lim\limits_{k\to\infty}\frac{x_{1k}}{\bar{\varepsilon}_k}=y^{10}=\sqrt{\frac{2}{a^*}}y_0,
\end{equation}
where $y_0\in\R^2$ is given in \eqref{cond:Vcritical.u1}.
Hence, together with \eqref{exp2:ea} and  \eqref{ineq:GN}, it follows from  \eqref{L:estsub:betaa1.e.V} and \eqref{L:estsub:betaa1.gn} that
\begin{equation}\label{estsub:betaa1.e.2}
\begin{split}
  e(a_k,b,\beta_k)\geq&\frac{(a^*-a_k)(a^*-b)-(\beta_k-a^*)^2}{2(a^*-b)}(\bar{\varepsilon}_k)^{-2}\\
  &+\Big[\Big({\frac{2}{a^*}}\Big)^\frac{p_1+2}{2}\frac{\lambda^{2+p_1}}{(a^*-b)p_1}+o(1)\Big]\bar{\varepsilon}_k^{p_1} \ \ \text{as}\ \ k\to\infty.
\end{split}
\end{equation}
Taking the infimum of (\ref{estsub:betaa1.e.2}) over $\bar{\varepsilon}_k>0$ yields that
\begin{equation}\label{estsub:betaa1.e}
\begin{split}
\liminf\limits_{k\to\infty}\frac{e(a_k,b,\beta_k)}{\big[(a^*-a_k)(a^*-b)-(\beta_k-a^*)^2\big]^\frac{p_1 }{p_1 +2}}
\geq\frac{\lambda^2}{a^*(a^*-b)}\frac{p_1 +2}{p_1},
\end{split}
\end{equation}
where the equality holds if and only if (\ref{estsub:betaa1.e.y}) holds and
\begin{equation}\label{estsub:betaa1.eps}
\lim_{k\to\infty}\bar{\varepsilon}_{k}/\varepsilon_k =\sqrt{\frac{a^*}{2}}\ \text{ with }\ \varepsilon_k>0 \  \text{ given by } \ \eqref{def:betaa1.eps.rate}.
\end{equation}
We thus conclude from \eqref{estsup:betaa1.e} and \eqref{estsub:betaa1.e} that \eqref{lim:a1beta.e.rate} holds, which implies that all equalities in \eqref{L:estsub:betaa1.e.V} and \eqref{estsub:betaa1.e} hold. Therefore, both \eqref{estsub:betaa1.e.y} and \eqref{estsub:betaa1.eps} hold true.
Thus, it follows from \eqref{lim:betaa1.w1}, \eqref{estsub:betaa1.e.y} and \eqref{estsub:betaa1.eps} that \eqref{lim:betaa1.z.rate} is true and \eqref{lim:betaa1.u.rate} holds strongly in $H^1(\R^2)$.
Furthermore,  similar to the proof of  \cite[Lemma 4.1]{GLWZ}, we have the exponential decay \eqref{4:conexp1} and \eqref{4:conexp2}. Finally, applying the standard elliptic regularity theory,  the argument similar to proving   Theorem \ref{Th:b1} (see also \cite[Proposition 2.1]{GLW}) implies that \eqref{lim:betaa1.u.rate} holds uniformly in $L^\infty(\R^2)$.
This therefore completes the proof of Proposition \ref{Prop:a1beta.3}.
\end{proof}

In the following we address some sufficient conditions ensuring that $u_{2k}\not\equiv0$ in $\R^2$ for sufficiently large $k>0$.

 \begin{lemma} \label{lem3:u2}
Suppose $V_i(x)\in C^2(\R^2)$ is homogeneous of degree $p_i$ with $2\le p_1\leq p_2$, where $V_i(x)$ satisfies  \eqref{cond:V.1} and
\begin{equation}\label{cond:Vcritical}
\text{$y_0$ is a unique  and non-degenerate critical point of $H_1(y)$.}
\end{equation}
Let $ (u_{1k},u_{2k})$ be a nonnegative minimizer of $e(a_k, b, \beta_k)$ satisfying (\ref{cond:-Vcritical}).
If $\beta _k$ also satisfies $a^*-a_k=o(\beta_k-a^*)$ as $k\to\infty$, then we have
\begin{equation}\label{u2.non0}
  u_{2k}\not\equiv0 \, \ \text{in}\, \ \R^2 \, \, \text{for sufficiently large} \, \ k>0.
\end{equation}
\end{lemma}

\begin{proof}
We shall prove \eqref{u2.non0} by  contradiction.
On the contrary, suppose $u_{2k}\equiv0$ in $\R^2$ for sufficiently large $k>0$, from which
we first derive a refined  lower  estimate of the energy $e(a_k,b,\beta_k)$ satisfying (\ref{cond:-Vcritical}).
Under the assumption (\ref{cond:Vcritical}), since $u_{2k}\equiv0$ in $\R^2$ for sufficiently large $k>0$, we then derive from \cite[Theorem 1.2]{GLW} that $u_{1k}$ solves a single elliptic equation and admits the following refined spike profile
\begin{equation}\arraycolsep=1.5pt\label{est:u1k}
u_{1k}= \displaystyle\frac {1}{\|w\|_2\tilde{\varepsilon}_k} \Big\{   w\Big(\frac{x-x_k}{\tilde{\varepsilon}_k}\Big)+(a^*-a_k)\psi\Big(\frac{x-x_k}{\tilde{\varepsilon}_k}\Big)
+o\big(a^*-a_k\big)  \Big\}  \mbox{ as }  k\to\infty
\end{equation}
uniformly in $\R^2$ for some $\psi\in C^2(\R^2)\cap L^\infty(\R^2)$, where $\tilde{\varepsilon}_k>0$ satisfies
\begin{equation}
\label{est:3k}
\tilde{\varepsilon}_k=\frac{1}{\lambda_0}(a^*-a_k)^\frac{1}{p_1+2},\ \,  \lambda_0:=\Big(\frac{p_1}{2}H_1(y_0)\Big)^\frac{1}{p_1+2}>0,
\end{equation}
and $ x_k$ is the unique maximum point of $u_{1k}$ satisfying
\begin{equation}\label{1a:conb}
 \big|\frac{x_k}{\tilde{\varepsilon}_k}-y_0\big|=(a^*-a_k)O(|y^0|) \, \ \text{as} \,  \  k\to\infty
\end{equation}
for some $y^0\in\R^2$.
We then derive from (\ref{est:u1k}) that
\begin{equation*}
\int_{\R^2} |\nabla u_{1k}|^2\dx=\frac{1}{\tilde{\varepsilon}_k^2}
    +\frac{2\lambda_0^{p_1+1}\int_{\R^2}\nabla w\nabla \psi}{a^*}\tilde{\varepsilon}_k^{p_1}+o(\tilde{\varepsilon}_k^{p_1}) \, \ \text{as} \,  \  k\to\infty,
\end{equation*}
and
\begin{equation*}
\begin{split}
  \int_{\R^2} V_1(x)u_{1k}^2\dx
  =&\int_{\R^2} V_1(\tilde{\varepsilon}_k x+x_k)\tilde{\varepsilon}_k^2u_{1k}^2(\tilde{\varepsilon}_k x+x_k)\dx\\
  =&\frac{1}{a^*}\tilde{\varepsilon}_k^{p_1}\int_{\R^2} V_1\Big( x+\frac{x_k}{\tilde{\varepsilon}_k}\Big)
  \Big[w^2+2\lambda_0^{p_1+1}\tilde{\varepsilon}_k^{p_1+2}\psi w+o(\tilde{\varepsilon}_k^{p_1+2})\Big]\\
   =&\frac{H_1(y_0)}{a^*}\tilde{\varepsilon}_k^{p_1}
   +\frac{2}{a^*}\lambda_0^{p_1+1}\tilde{\varepsilon}_k^{2p_1+2}\int_{\R^2} V_1(x+y_0)\psi w\\
   &
   +\frac{1}{a^*}\tilde{\varepsilon}_k^{p_1}\Big[H_1\Big(\frac{x_k}{\tilde{\varepsilon}_k}\Big)-H_1(y_0)\Big]
   +o(\tilde{\varepsilon}_k^{2p_1+2}) \, \ \text{as} \,  \  k\to\infty.\\
\end{split}
\end{equation*}
Note from  \eqref{cond:Vcritical} and \eqref{1a:conb}  that
$H_1\Big(\frac{x_k}{\tilde{\varepsilon}_k}\Big)-H_1(y_0)= o(a^*-a_k)$ as $k\to\infty$.
Hence, we reduce from above that
\begin{equation*}
\begin{split}
e(a_k,b,\beta_k)
\geq&\frac{1}{a^*}(a^*-a_k)\int_{\R^2} |\nabla u_{1k}|^2\dx+\int_{\R^2} V_1(x)u_{1k}^2\dx\\
=&\frac{\lambda_0^2}{a^*}\frac{p_1+2}{p_1}(a^*-a_k)^\frac{p_1}{p_1+2}+\frac{2\lambda_0\int_{\R^2}\nabla w\nabla \psi}{(a^*)^2}(a^*-a_k)^{\frac{p_1}{p_1+2}+1}\\
&+\frac{2}{a^*}\frac{\int_{\R^2} V_1(x+y_0)\psi w}{\lambda_0^{p_1+1}} (a^*-a_k)^\frac{2+2p_1}{2+p_1}
   +o\Big((a^*-a_k)^\frac{2+2p_1}{2+p_1}\Big) \, \ \text{as} \,  \  k\to\infty,
\end{split}
\end{equation*}
where (\ref{ineq:GNQ}) is used in the first inequality.
Therefore, under the assumption (\ref{cond:Vcritical}) we conclude from above that
\begin{equation}\label{est:low.e}
  e(a_k,b,\beta_k)\geq \frac{\lambda_0^2}{a^*}\frac{p_1+2}{p_1}(a^*-a_k)^\frac{p_1}{p_1+2}\Big[1+O\big(a^*-a_k\big)\Big] \, \ \text{as} \,  \  k\to\infty,
\end{equation}
where $\lam _0>0$ is defined by (\ref{est:3k}).

Under the additional assumption that $\beta _k$ also satisfies $a^*-a_k=o(\beta_k-a^*)$ as $k\to\infty$,
we next derive a refined upper estimate of $e(a_k,b,\beta_k)$ as $k\to\infty$. Take a test function of the form \eqref{def:betaa1.trial},
where $y_0\in\R^2$ is given by (\ref{cond:Vcritical}), $A_{}=\Big(\frac{a^*-b}{\beta_k-b}\Big)^\frac{1}{2}<1$,
and $\tau=
\lambda_0\Big[\frac{1}{1-\frac{(\beta_k-a^*)^2}{(a^*-a_k)(a^*-b)}}\Big]^\frac{1}{p_1 +2}(a^*-a_k)^\frac{-1}{p_1 +2}>0$ for $\lam _0>0$ defined by (\ref{est:3k}).
Similar to \eqref{estsup.betaa1.e.nabla} and \eqref{estsup:betaa1.e.V}, some calculations then yield that
\begin{equation}\label{3.3:1}
\begin{split}
&\int_{\R ^2} \big(|\nabla \phi_1|^2+|\nabla \phi_2|^2\big)\dx-\int_{\R ^2} \Big(\frac{a_k}{2}|\phi_1|^4+\frac{b}{2}|\phi_2|^4 +\beta _k|\phi_1|^2|\phi_2|^2\Big) \dx\\
=&\frac{1}{a^*}\Big(\frac{a^*-b}{\beta_k-b}\Big)^2\Big[1-\frac{(\beta_k-a^*)^2}{(a^*-a_k)(a^*-b)}\Big](a^*-a_k)\tau^2\\
\leq&\frac{\lambda_0^2}{a^*}(a^*-a_k)^\frac{p_1}{p_1 +2}
\Big[1-\frac{(\beta_k-a^*)^2}{(a^*-a_k)(a^*-b)}\Big]^\frac{p_1}{p_1+2} \, \ \text{as} \,  \  k\to\infty,\\
\end{split}
\end{equation}
and
\begin{equation}\label{3.3:2}
\begin{split}
&\int_{\R ^2}\big( V_1(x)|\phi_1|^2+V_2(x)|\phi_2|^2\big) \dx\\
=&\frac{1}{a^*}\frac{2}{p_1 }\frac{a^*-b}{\beta_k-b}\lambda_0^{p_1 +2}\tau^{-p_1} +\frac{1}{a^*}\frac{a^*-b}{\beta_k-b}\frac{\beta_k-a^*}{a^*-b}H_2(y_0)\tau^{-p_2}\\
\leq&\frac{\lambda_0^2}{a^*}\frac{2}{p_1 }(a^*-a_k)^\frac{p_1}{p_1 +2}
\Big[1-\frac{(\beta_k-a^*)^2}{(a^*-a_k)(a^*-b)}\Big]^\frac{p_1}{p_1+2}\\
&+\frac{\beta_k-a^*}{a^*-b}\frac{1}{a^*}
\frac{H_2(y_0)}{\lambda_0^{p_2}}\Big[1-\frac{(\beta_k-a^*)^2}{(a^*-a_k)(a^*-b)}\Big]^\frac{p_2}{p_1+2}(a^*-a_k)^\frac{p_2}{p_1+2}
\end{split}
\end{equation}
as $k\to\infty$.
Thus,  we derive from \eqref{3.3:1} and \eqref{3.3:2} that
\begin{equation}\label{3:upper-1}
e(a_k,b,\beta_k)\leq E_{a_k,b,\beta_k}(\phi_1,\phi_2)\le\frac{\lambda_0^2}{a^*}\frac{p_1+2}{p_1}(a^*-a_k)^\frac{p_1}{p_1+2}I_k \, \ \text{as} \,  \  k\to\infty,
\end{equation}
where $\lam _0>0$ is defined by (\ref{est:3k}) and $I_k>0$ satisfies
\[\begin{split}
I_k:=&\Big[1- \frac{(\beta_k-a^*)^2}{(a^*-a_k)(a^*-b)}\Big]^\frac{p_1}{p_1+2}\Big\{1+\frac{2}{p_1+2}\frac{H_2(y_0)}{H_1(y_0)}\\
&\qquad \cdot\Big[\frac{1}{\lambda_0^{p_1+2}}
\Big(1-\frac{(\beta_k-a^*)^2}{(a^*-a_k)(a^*-b)}\Big)(a^*-a_k)\Big]^\frac{p_2-p_1}{p_1+2}\frac{\beta_k-a^*}{a^*-b}\Big\}\\
\le &\Big[1- \frac{(\beta_k-a^*)^2}{(a^*-a_k)(a^*-b)}\Big]^\frac{p_1}{p_1+2}\Big\{1+\frac{2}{p_1+2}
\frac{H_2(y_0)}{H_1(y_0)}\frac{\beta_k-a^*}{a^*-b}\Big\},
\end{split}\]
since $2\le p_1\le p_2$. Under the assumption that $a^*-a_k=o(\beta_k-a^*)$ as $k\to\infty$, we next derive a contradiction by two cases.

We first consider the case where $\liminf\limits_{k\to\infty}\frac{(\beta_k-a^*)^2}{(a^*-a_k)(a^*-b)}:=\gamma >0$, which then implies that $0<\gamma \le 1$ in view of (\ref{cond:-Vcritical}). We further reduce from above that
\[
0<I_k< I_0:=\Big(1-\frac{\gamma}{2}\Big)^\frac{p_1}{p_1+2}<1\, \ \text{as} \,  \  k\to\infty.
\]
This estimate and (\ref{3:upper-1}) then give that
\begin{equation}\label{3.3:C}
e(a_k,b,\beta_k)<\frac{\lambda_0^2}{a^*}\frac{p_1+2}{p_1}(a^*-a_k)^\frac{p_1}{p_1+2}I_0\, \ \text{as} \,  \  k\to\infty,
\end{equation}
which however contradicts to \eqref{est:low.e}, and the lemma is therefore proved in the first case.
We next consider the case where $\liminf\limits_{k\to\infty}\frac{(\beta_k-a^*)^2}{(a^*-a_k)(a^*-b)}=0$. In this case, since $a^*-a_k=o(\beta_k-a^*)$ as $k\to\infty$, we then have
\[
\beta_k-a^*=o\Big(\frac{(\beta_k-a^*)^2}{(a^*-a_k)(a^*-b)}\Big)\, \ \text{as} \,  \  k\to\infty,
\]
from which we have
\begin{equation*}
\begin{split}
0<I_k\le&\Big[1-\frac{p_1}{p_1+2}\frac{(\beta_k-a^*)^2}{(a^*-a_k)(a^*-b)}+o\Big(\frac{(\beta_k-a^*)^2}{a^*-a_k}\Big)\Big]\\
&\cdot\Big\{1+\frac{2}{p_1+2}\frac{H_2(y_0)}{H_1(y_0)}\frac{\beta_k-a^*}{a^*-b}\Big\}\\
< &1-\frac{p_1}{2(p_1+2)(a^*-b)} \Big(\frac{\beta_k-a^*}{a^*-a_k}\Big)^2(a^*-a_k)\, \ \text{as} \,  \  k\to\infty.\\
\end{split}
\end{equation*}
This estimate and (\ref{3:upper-1}) then give that
\[e(a_k,b,\beta_k)<\frac{\lambda_0^2}{a^*}\frac{p_1+2}{p_1}(a^*-a_k)^\frac{p_1}{p_1+2}\Big[1-\frac{p_1}{2(p_1+2)(a^*-b)} \Big(\frac{\beta_k-a^*}{a^*-a_k}\Big)^2(a^*-a_k)\Big]\]
as $k\to\infty$, which also contradicts to \eqref{est:low.e} in view of the assumption that $a^*-a_k=o(\beta_k-a^*)$ as $k\to\infty$. This therefore finishes the proof of the lemma.
\end{proof}

\begin{remark}\label{rem3.1} Under the assumptions that $V_i(x)\in C^2(\R^2)$ is homogeneous of degree $p_i$ with $2\le p_1\leq p_2$ and satisfies \eqref{cond:V.1} and \eqref{cond:Vcritical.u1}, instead of the non-degeneracy \eqref{cond:Vcritical}, the proof of Lemma \ref{lem3:u2} implies that if $\beta _k$ also satisfies $\lim\limits_{k\to\infty}\inf\frac{(\beta_k-a^*)^2}{(a^*-a_k)(a^*-b)}>0$, we also have $u_{2k}\not\equiv 0$ in $\R^2$ for sufficiently large $k>0$.
\end{remark}

\subsection{Refined spike profiles of $u_{2k}$}

Based on Proposition \ref{Prop:a1beta.3}, the first purpose of this subsection is to derive the refined spike profiles of $u_{2k}$ as $k\to\infty$ for the case where $u_{2k}\not \equiv 0$ in $\R^2$, by which we then complete the proof of Theorems \ref{thm:1.1} and \ref{thm:1.2}. Recall  that $(u_{1k}, u_{2k})$ solves the following PDE system
\begin{equation}\label{4.1:1a}
\begin{cases}
-\Delta u_{1k}+ V_1(x)u_{1k}
= \mu_{k} u_{1k}+a_ku_{1k}^3+\beta _ku_{2k}^2u_{1k}\ \ \text{in}\,\ \R^2,\\
-\Delta u_{2k}+ V_2(x)u_{2k}
= \mu_{k} u_{2k}+bu_{2k}^3+\beta _ku_{1k}^2u_{2k}\quad \ \text{in}\,\ \R^2,
\end{cases}
\end{equation}
where $\mu _{k}\in \R$ is a suitable Lagrange multiplier.

If $u_{2k}\not \equiv 0$ in $\R^2,$  define
\begin{equation}\label{4.1:1L}
\begin{split}
\bar u_{1k}(x)=\sqrt{a^*} \varepsilon_{k}  u_{1k}(\varepsilon_{k}x+x_{1k})\,\  \text{and}\,\ u_{2k}(\varepsilon_{k}x+x_{1k})=C_\infty \sigma _k\bar u_{2k}(x), \\
\text{where}\,\  \sigma _k=\|u_{2k}\|_\infty>0 \,\  \text{and}\,\ C_\infty=\frac{1}{\|w\|_\infty}> 0,\qquad\qquad\qquad\quad\quad
\end{split}
\end{equation}
and $ x_{1k}$ is the unique maximum point of $u_{1k}$, so that
\begin{equation}\label{4.1:2}
\bar u_{1k}(x)\to w(x) \,\  \text{and}\,\ \sigma _k\eps_k\to 0\,\  \text{as}\,\ k\to\infty,
\end{equation}
where (\ref{lim:betaa1.u.rate}) is used. Then $(\bar u_{1k}, \bar u_{2k})$ solves the following PDE system
\begin{equation}\label{4.1:2a}\arraycolsep=1.5pt
\begin{cases}
 &-\Delta \bar u_{1k}+\eps _k^2V_1(\eps _k x+x_{1k})\bar u_{1k}\\[1mm]
=& \mu_{k}\eps _k^2 \bar u_{1k}+\displaystyle\frac{a_k}{a^*}\bar u_{1k}^3+\beta _kC_\infty^2\sigma _k^2\eps_k^2\bar u_{2k}^2\bar u_{1k}\ \ \text{in}\,\ \R^2,\\[2mm]
 &-\Delta \bar u_{2k}+\eps _k^2V_2(\eps _k x+x_{1k})\bar u_{2k}\\[1mm]
=&\mu_{k}\eps _k^2 \bar u_{2k}+bC_\infty^2\sigma _k^2\eps_k^2\bar u_{2k}^3+\displaystyle\frac{\beta_k}{a^*} \bar u_{1k}^2 \bar u_{2k}\quad \ \text{in}\,\ \R^2.
\end{cases}
\end{equation}
The following lemma gives the fundamental limit behavior of $u_{2k}$ as $k\to\infty$.

\begin{lemma}\label{lem3.4A} Suppose $V_i(x)\in C^2(\R^2)$ is homogeneous of degree $p_i$ and satisfies  \eqref{cond:V.1} and \eqref{cond:Vcritical.u1}, where $i=1,2$ and $2\le p_1\leq p_2$.  Let $ (u_{1k},u_{2k})$ be a nonnegative minimizer of $e(a_k, b, \beta_k)$ satisfying \eqref{cond:-Vcritical}. Suppose that $u_{2k}\not \equiv 0$ in $\R^2$ and define
\begin{equation}\label{3.4:1}
u_{2k}(\varepsilon_{k}x+x_{2k})=C_\infty \sigma _k \tilde  u_{2k}(x),
\end{equation}
where $\sigma _k=\|u_{2k}\|_\infty>0$, $C_\infty=\frac{1}{\|w\|_\infty}> 0$,
and $x_{ik}$ is the unique maximum point of $u_{ik}$ for $i=1, 2$. Then there exists a subsequence of $\{\tilde u_{2k}\}$ (still denoted by $\{\tilde u_{2k}\}$) such that
\begin{equation}\label{3.4:2}
  \tilde u_{2k}(x)\to w(x)\,\  \text{uniformly in $\R^2$ as} \,\ k\to\infty,
\end{equation}
and
\begin{equation}\label{4.1:3D}
  \lim\limits_{k\to\infty}\frac{x_{2k}-x_{1k}}{\varepsilon_k}=0.
\end{equation}
\end{lemma}

\begin{proof}
 Consider \eqref{3.4:1}, where $x_{2k}\in\R^2$ is a global maximum point of $u_{2k}$. We then obtain that
\begin{equation}\label{3.1:3AD}\tilde u_{2k}(0)=\|\tilde u_{2k}(x)\|_\infty=\|w\|_{\infty}>0,
\end{equation}
and $(\bar u_{1k}, \tilde u_{2k})$ solves the elliptic PDE system
\begin{equation}\label{3.4:3D}
\begin{cases}
\begin{split}
 &-\Delta \bar u_{1k}+\eps _k^2V_1(\eps _k x+x_{1k})\bar u_{1k}\\
=& \mu_{k}\eps _k^2 \bar u_{1k}+\displaystyle\frac{a_k}{a^*}\bar u_{1k}^3+\beta _kC_\infty^2\sigma _k^2\eps_k^2\tilde  u_{2k}^2\big(x+\frac{x_{1k}-x_{2k}}{\varepsilon_k}\big)\bar u_{1k}\ \ \text{in}\,\ \R^2,\\
 &-\Delta \tilde  u_{2k}+\eps _k^2V_2(\eps _k x+x_{2k})\tilde  u_{2k}\\
=&\mu_{k}\eps _k^2 \tilde  u_{2k}+bC_\infty^2\sigma _k^2\eps_k^2\tilde  u_{2k}^3+\displaystyle\frac{\beta_k}{a^*} \bar u_{1k}^2\big(x+\frac{x_{2k}-x_{1k}}{\varepsilon_k}\big) \tilde  u_{2k}\quad \ \text{in}\,\ \R^2,
\end{split}
\end{cases}
\end{equation}
where the Lagrange multiplier $\mu _{k}\in \R$ satisfies $\mu _{k}\eps_k^2\to -1$ as $k\to\infty$.
Using the elliptic regularity theory, we thus deduce from \eqref{3.4:3D} that  there exists $0\leq  u_0(x)\in H^1(\R^2)$ such that
\begin{equation}\label{eq3.655}
\tilde u_{2k}\to u_0(x) \text{ in } C^{2,\alpha}_{\rm loc}(\R^2) \text{ as } k\to\infty. \end{equation}
Also, we have
\begin{equation}\label{eq3.65}
u_{0}(0)= \|u_{0}(x)\|_\infty=\|w\|_{\infty}>0.
\end{equation}
Similar to those in \cite{GLWZ} and references therein, one can further derive from (\ref{3.4:3D}) that $(  u_{1k},   u_{2k})$ admits a unique maximum global point $(x_{1k}, x_{2k})$, and satisfies the exponential decay (\ref{4:conexp1}) and (\ref{4:conexp2}).

We now show that
\begin{equation}\label{eq3.66}
\Big\{\frac{x_{2k}-x_{1k}}{\eps_k}\Big\}\  \text{ is bounded uniformly in } \R^2.
\end{equation}
Indeed, since $\beta_k\searrow a^*$ as $k\to\infty$, if (\ref{eq3.66}) is false, we then obtain from \eqref{4.1:2} and  (\ref{3.4:3D}) that $u_0$ satisfies $-\Delta u_0(x)+ u_0=0$ in $\R^2$. This implies that $ u_0(x)\equiv 0$ in $\R^2$, which however contradicts to (\ref{eq3.65}). Therefore, the estimate (\ref{eq3.66}) holds true.
Up to a subsequence if necessary, we then deduce from  (\ref{eq3.66}) that there exists an $x_0\in\R^2$ such that
\begin{equation}\label{3.4:5}
\lim_{k\to\infty}\frac{x_{2k}-x_{1k}}{\eps_k}=x_0\in\R^2.
\end{equation}
Moreover, it follows from \eqref{4.1:2}, (\ref{3.4:3D}) and (\ref{3.4:5}) that $ u_0(x)$ satisfies
 \begin{equation}
-\Delta u_0(x) + u_0(x)=w^2(x+x_0) u_0 (x)\  \text{  in } \, \R^2.
\end{equation}
We thus obtain from (\ref{3.1:3AD}) and \cite[Lemma 4.1]{Wei96} that
\begin{equation*}
u_0(x)\equiv  w(x+x_0) \, \text{  in  }\, \R^2.
\end{equation*}
Since $x=0$ is a maximum point of  $\tilde u_{2k}(x)$ for each $k\in  \mathbb{N}$, it  is also a maximum point of $w(x+x_0)$.
However, $w(x)$ admits  a unique maximum point $x=0$, from which we conclude that (\ref{3.4:5}) holds for
$x_0=0$. Therefore, this implies that (\ref{4.1:3D}) holds.

Finally, since $(u_{1k}, u_{2k})$ satisfies the exponential decay (\ref{4:conexp1}) and (\ref{4:conexp2}), by (\ref{eq3.655}) we can follow the proof of Theorem \ref{Th:b1} to conclude that (\ref{3.4:2}) holds uniformly in $\R^2$ as $k\to\infty$. This completes the proof of the lemma.
\end{proof}

Under the assumptions of Lemma \ref{lem3.4A}, if $u_{2k}\not \equiv 0$ in $\R^2,$  we next define
\begin{equation}\label{4.1:3a}
w_{ik}(x)=\bar u_{ik}(x)-w(x),\,\ i=1, 2,
\end{equation}
so that $ w_{ik}(x)\to 0$ uniformly in $\R^2$ as $k\to\infty$ in view of (\ref{4.1:2}), (\ref{3.4:2}) and (\ref{4.1:3D}).
We also denote the linearized operator
\begin{equation}\label{4.1:4} \arraycolsep=1.5pt
\left\{ \arraycolsep=1.5pt\begin{array}{lll}
\  \qquad \mathcal{L}_{1k} w_{1k}&=& -\Delta w_{1k}+\big[1-\big(\bar u_{1k}^2+\bar u_{1k}w+w^2\big)\big]w_{1k}\  \  \text{in}\,\ \R^2,\\[2mm]
 \mathcal{L}_{2k}(w_{1k})w_{2k}
&=& -\Delta w_{2k}+\big[1- \bar u_{1k}^2\big]w_{2k} \\[2mm]
 &&-w\big(\bar u_{1k}+w\big) w_{1k}\, \  \text{in}\,\ \R^2,
\end{array}\right.
\end{equation}
and the associated  limit operator
\begin{equation}\label{4.1:4C}
\begin{cases}
\mathcal{L}_{1}  \phi _1 =-\Delta \phi _1+\big[1-3w^2 \big]\phi _1\ \ &\text{in}\,\ \R^2,\\
\mathcal{L}_{2}(\phi _1)\phi _2
=-\Delta \phi _2+(1- w^2)\phi _2-2w^2\phi _1\  \ &\text{in}\,\ \R^2.
\end{cases}
\end{equation}
Then $(w_{1k}, w_{2k})$ satisfies $\nabla w_{1k}(0)=0$  and
\begin{equation}\label{4.1:5}\arraycolsep=1.5pt\begin{array}{lll}
\qquad \mathcal{L}_{1k} w_{1k} &=&-(a^*-a_k)\displaystyle \frac{1}{a^*}\bar u_{1k}^3(x)- \displaystyle  \eps_k^{2+p_1}  V_1\big(x+\frac{  x_{1k}}{\eps _k}\big)\bar u_{1k}(x) \\[4mm]
&&+\delta_k\bar u_{1k}(x)+\displaystyle\beta _kC_\infty^2\sigma _k^2\eps_k^2\bar u_{2k}^2 \bar u_{1k}:=h_{1k}(x)\,\ \mbox{in}\,\ \R^2, \\[2mm]
\mathcal{L}_{2k}(w_{1k})w_{2k}&=& \displaystyle  (\beta_k -a^*)\frac{1}{a^*}\bar u_{1k}^2 \bar u_{2k}-\displaystyle \eps_k^{2+p_2}V_2\big(x+\frac{  x_{1k}}{\eps _k}\big)\bar u_{2k}\\[3mm]
&&+\delta_k\bar u_{2k}(x)+bC^2_\infty\sigma _k^2\eps_k^2\bar u_{2k}^3:=h_{2k}(x)  \,\ \mbox{in}\,\ \R^2,
\end{array}\end{equation}
where we denote
\begin{equation}\label{4.1:6}
\delta_k=1+\mu_{k}\eps _k^2\to 0\,\  \text{as} \,\ k\to\infty.
\end{equation}
Taking the limit of (\ref{4.1:5}) and using (\ref{lim:betaa1.z.rate}), let $(w_1, w_2)$ solve the following system
\begin{equation}\label{4.1:7}\arraycolsep=1.5pt
\left\{\begin{array}{lll}
\qquad   \mathcal{L}_{1}  w_1
&=&\delta_kw+\displaystyle a^* C_\infty^2\sigma _k^2\eps_k^2w^3- \displaystyle \frac{a^*-a_k}{a^*}w^3- \displaystyle  \eps_k^{2+p_1}  V_1 (x+y_0)w \\[2mm]& :=&h_{1}(x)
\  \  \text{in}\,\ \R^2, \,\ \nabla w_{1}(0)= 0,\\[2mm]
\mathcal{L}_{2}(w_1)w _2
&=&\delta_kw+bC^2_\infty\sigma _k^2\eps_k^2w^3+\displaystyle   \frac{\beta _k-a^*}{a^*}w^3-\displaystyle \eps_k^{2+p_2}V_2(x+y_0)w
\\[2mm]
& :=&h_{2}(x)\  \  \text{in}\,\ \R^2, \,\ \nabla w_{2}(0)= 0,
\end{array}\right.
\end{equation}
where $w_i \in C^2(\R^2)\cap L^\infty(\R^2)$ for $i=1, 2$. Here $\nabla w_{2}(0)= 0$ is due to (\ref{4.1:3D}), (\ref{4.1:3a}) and the fact that $\nabla \bar u_{2k}(\frac{x_{2k}-x_{1k}}{\varepsilon_k})\equiv 0$. We then obtain that $w_i$ exists and satisfies $w_i \to 0$ uniformly in $\R^2$ as $k\to\infty$ for $i=1, 2$. Following above results, we are now ready to complete the proof of Theorem \ref{thm:1.1}.

\vskip 0.1truein

\begin{proof}[\bf Proof of Theorem \ref{thm:1.1}.] Under the assumptions of Theorem \ref{thm:1.1}, by Proposition \ref{Prop:a1beta.3} the rest is to further prove that $u_{2k}$ satisfies
\begin{equation}\label{5:1}
\sqrt{\frac{a^*(a^*-b)}{\beta _k -a^*}}    \varepsilon_k   u_{2k}(\varepsilon_k  x+x_{2k})\to w(x)
\end{equation}
uniformly in $\R^2$ as $ k\to\infty$, where $\eps _k>0$ is given by (\ref{def:betaa1.eps.rate}).

Actually, under the assumptions of Theorem \ref{thm:1.1}, we note from Lemma \ref{lem3:u2} that $u_{2k}\not \equiv 0$ for sufficiently large $k>0$. By considering $\bar u_{ik}$ defined in (\ref{4.1:1L}) for $i=1, 2$, we then get that Lemma  \ref{lem3.4A} holds true. Following these, we thus deduce from (\ref{4.1:7}) that $w_1$ exists and
\begin{equation}\label{4.1:8}\arraycolsep=1.5pt\begin{array}{lll}
&- \displaystyle 2\inte w^3w_1\\
=& \displaystyle\inte w\big[-\Delta +(1-w^2)\big]w_1-2 \displaystyle\inte w^3w_1\\[3mm]
=& \displaystyle\inte w\mathcal{L}_{1}w _1=\inte wh_1(x)\\[3mm]
=&a^*\delta_k+2(a^*)^2C^2_\infty\sigma _k^2\eps_k^2-2(a^*-a_k)-H_1(y_0)\eps_k^{2+p_1} \  \  \text{as}\,\ k\to\infty.
\end{array}\end{equation}
On the other hand, we also get from (\ref{4.1:7}) that $w_2$ exists and
\begin{equation}\label{4.1:9}\arraycolsep=1.5pt\begin{array}{lll}
&- \displaystyle 2\inte w^3w_1\\
=& \displaystyle\inte w\Big\{\big[-\Delta +(1-w^2)\big]w_2-2 \displaystyle  w^2w_1\Big\}\\[3mm]
=& \displaystyle\inte w\mathcal{L}_{2}(w _1)w_2=\inte wh_2(x)\\[3mm]
=&2(\beta_k -a^*)-H_2(y_0)\eps_k^{2+p_2}+a^*\delta_k+2a^*bC^2_\infty\sigma _k^2\eps_k^2 \  \  \text{as}\,\ k\to\infty,
\end{array}\end{equation}
where $H_2(y)$ is defined by (\ref{H}).
Therefore, above two identities give that
\begin{equation}\label{4.1:sigma-1}\arraycolsep=1.5pt\begin{array}{lll}
&&2a^*(a^*-b)C^2_\infty\sigma _k^2\eps_k^2\\[3mm]
&=&2(\beta _k-a^*)+2(a^*-a_k)+ H_1(y_0)\eps_k^{2+p_1}-H_2(y_0)\eps_k^{2+p_2} \  \  \text{as}\,\ k\to\infty.
\end{array}\end{equation}
Since the assumption  (\ref{con:a1beta.ab}) implies that $a^*-a_k=o(\beta_k -a^*)$ as $k\to\infty$, we then derive from above that
\begin{equation}\label{4.1:sigma}
a^*(a^*-b)C^2_\infty\sigma _k^2\eps_k^2\  \sim \  (\beta_k -a^*)  \  \  \text{as}\,\ k\to\infty.
\end{equation}
Applying Lemma \ref{lem3.4A}, we therefore conclude (\ref{5:1}) from  (\ref{4.1:1L}) and (\ref{4.1:sigma}), and we are done.
\end{proof}

As a byproduct, the argument of proving Theorem \ref{thm:1.1} leads us to complete the proof of Theorem \ref{thm:1.2}.


\begin{proof}[\bf Proof of Theorem \ref{thm:1.2}.] If $0<b<a^*$, $a_k\nearrow a^*$ and $\beta _k\to \beta _*\in (0, a^*)$ as $k\to\infty$, Theorem \ref{thm:1.2} then follows directly from Theorem \ref{Th:b1}.

We now address Theorem \ref{thm:1.2} for the case where $0<b<a^*$, $a_k\nearrow a^*$ and  $\beta _k\nearrow a^*$ satisfy $a^*-a_k=o(a^*-\beta_k)$ as $k\to\infty$. In this case, we first note that Proposition \ref{Prop:a1beta.3} still holds for $\eps_k>0$ satisfying (\ref{def:a1.eps}), and hence the rest is to prove that $u_{2k}\equiv 0$ for sufficiently large $k>0$. On the contrary, assume that $u_{2k}\not \equiv 0$ for sufficiently large $k>0$. We then consider
\begin{equation}\label{4.1:1Q}
\begin{split}
\bar u_{1k}(x)=\sqrt{a^*} \varepsilon_{k}  u_{1k}(\varepsilon_{k}x+x_{1k})\,\  \text{and}\,\ u_{2k}(\varepsilon_{k}x+x_{1k})=C_\infty \sigma _k\bar u_{2k}(x), \\
\text{where}\,\  \sigma _k=\|u_{2k}\|_\infty>0 \,\  \text{and}\,\ C_\infty=\frac{1}{\|w\|_\infty}> 0,\qquad\qquad\qquad\quad\quad
\end{split}
\end{equation}
where $\eps_k>0$ satisfies (\ref{def:a1.eps}), and $x_{1k}\in \R^2$ is the unique maximum point of $u_{1k}$. In this case, one can check that $\bar u_{ik}$ ($i=1, 2$) still satisfies Lemma \ref{lem3.4A}. Following these, the argument of (\ref{4.1:8})-(\ref{4.1:sigma-1}) further gives that there exists a constant $M>0$, independent of $k$, such that
\begin{equation}\label{4.1Q:sigma-1}\arraycolsep=1.5pt\begin{array}{lll}
0&<&2a^*(a^*-b)C^2_\infty\sigma _k^2\eps_k^2\\[3mm]
&=&2(a^*-a_k)-2(a^*-\beta _k)+ H_1(y_0)\eps_k^{2+p_1}-H_2(y_0)\eps_k^{2+p_2} \\[3mm]
&\le &M(a^*-a_k)-2(a^*-\beta _k)\le -(a^*-\beta _k)\  \  \text{as}\,\ k\to\infty,
\end{array}\end{equation}
a contradiction, where the last inequality follows from the assumption that $a^*-a_k=o(a^*-\beta_k)$ as $k\to\infty$. Therefore, we also have $u_{2k}\equiv 0$ for sufficiently large $k>0$ in this case, and the proof is then complete.
\end{proof}

The rest part of this subsection is to derive the following theorem by a different approach, which shows that if $\beta _k$ is close enough to $\beta ^*_k$, the refined spike behavior of $u_{2k}$ stated in Theorem \ref{thm:1.1} still holds without the non-degeneracy assumption of \eqref{1:H}.

\begin{theorem}\label{thm3.7}Suppose $V_i(x)\in C^2(\R^2)$ is homogeneous of degree $p_i$ and satisfies  \eqref{cond:V.1} and \eqref{cond:Vcritical.u1}, where $i=1,2$ and $2\le p_1\leq p_2$.
Let $ (u_{1k},u_{2k})$ be a nonnegative minimizer of $e(a_k, b, \beta_k)$ satisfying (\ref{cond:-Vcritical}).
If, additionally, $(a_k, b, \beta_k)$ also satisfies
\begin{equation}\label{4.1:3M}
\frac{\beta_k-a^*}{\sqrt{(a^*-a_k)(a^*-b)}}\to \varsigma_0\in (0,1)\,\  \text{as} \,\ k\to\infty,
\end{equation}
then (\ref{thm1.2:M}) still holds.
\end{theorem}


We first remark that Proposition \ref{Prop:a1beta.3} holds under the assumptions of Theorem \ref{thm3.7}. Further, if $(a_k, b, \beta_k)$ also satisfies (\ref{4.1:3M}), it follows from Remark \ref{rem3.1} that $u_{2k}\not \equiv 0$ in $\R^2$ for sufficiently large $k>0$. Further, the following refined estimates are needed for the proof of Theorem \ref{thm3.7}.

\begin{lemma}\label{lem3.6}
Under the assumptions of  Theorem \ref{thm3.7}, let $ (u_{1k},u_{2k})$ be a nonnegative minimizer of $e(a_k, b, \beta_k)$ as $ k\to\infty$.
Then we have
\begin{equation}\label{eq3.16}
\lim_{k\to\infty}\frac{\|u_{2k}\|_4^4}{\|u_{1k}\|_4^4}\Big(\frac{a^*-b}{\beta_k-a^*}\Big)^2=1,
\end{equation}
\begin{equation}\label{eq3.17}
\lim_{k\to\infty}\|u_{2k}\|_2^2\frac{a^*-b}{\beta_k-a^*}=1,
\end{equation}
and
\begin{equation}\label{eq3.18}
\inte|\nabla u_{2k}|^2dx\leq C\frac{\beta_k-a^*}{a^*-b}\inte|u_{1k}|^4.
\end{equation}
\end{lemma}

\begin{proof}
Under the assumptions of Theorem \ref{thm3.7}, we first note that Proposition \ref{Prop:a1beta.3} holds true, and its proof gives that the energy $e(a_k,b,\beta_k)$ satisfies
\begin{equation}\label{eq3.73}
\lim\limits_{k\to\infty}\frac{e(a_k,b,\beta_k)}{\varepsilon_k^{p_1}}
=\frac{\lambda^{2+p_1}}{a^*(a^*-b)}\frac{p_1 +2}{p_1},
\end{equation}
where $\varepsilon_k>0$ and $\lambda>0$ are given by \eqref{def:betaa1.eps.rate}.
By Proposition \ref{Prop:a1beta.3}, one can check from (\ref{L:estsub:betaa1.e.V}) that
\begin{equation}\label{eq3.74}
\liminf_{k\to\infty} \varepsilon _k^{-p_1}\int_{\R^2} V_1(x)|u_{1k}|^2\dx
\ge \frac{\lambda^{2+p_1}}{a^*(a^*-b)}\frac{2}{p_1}.
\end{equation}
Since $a^*\le \beta_k\le  \beta_k^*$, we also derive from (\ref{L:estsub:betaa1.gn}) and (\ref{4.1:3M}) that
\vspace{-3em}
\begin{equation}\label{eq3.75}
\begin{split}
I:=&\frac{a^*-a_k}{2}\int_{\R^2}|u_{1k}|^4\dx +\frac{a^*-b}{2}\int_{\R^2}|u_{2k}|^4\dx\\
&  + (a^*-\beta_k) \int_{\R^2} |u_{1k}|^2|u_{2k}|^2\dx \\
\geq&\frac{a^*-a_k}{2}\int_{\R^2}|u_{1k}|^4\dx
  \Big\{1-\frac{(\beta_k-a^*)^2}{(a^*-a_k)(a^*-b)}\\
  & +\frac{a^*-b}{a^*-a_k}
  \Big[\Big(\frac{\int_{\R^2}|u_{2k}|^4\dx}{\int_{\R^2}|u_{1k}|^4\dx}\Big)^\frac{1}{2}-\frac{\beta_k-a^*}{a^*-b}\Big]^2
\Big\} \\
= &\frac{(a^*-a_k)(a^*-b)-(\beta_k-a^*)^2}{2(a^*-b)}\int_{\R^2}|u_{1k}|^4\dx \\
&\cdot \Big[1+\frac{\varsigma_0^2(1+o(1))}
{1-\varsigma_0^2}\big(\frac{a^*-b}{\beta_k-a^*}\frac{\|u_{2k}\|_4^2}{\|u_{1k}\|_4^2}-1\big)^2\Big]\\
\geq &\frac{\lambda^{2+p_1}\varepsilon _k^{p_1}(1+o(1))}{a^*(a^*-b)} \ \text{as}\ \ k\to\infty,\vspace{-0.4em}
\end{split}
\end{equation}
where the identity in the last inequality holds if and only if (\ref{eq3.16}) holds. Using (\ref{eq3.74}) and (\ref{eq3.75}) we then deduce that
\begin{equation*}\vspace{-0.4em}
\lim\limits_{k\to\infty}\frac{e(a_k,b,\beta_k)}{\varepsilon_k^{p_1}}
\geq\frac{\lambda^{2+p_1}}{a^*(a^*-b)}\frac{p_1 +2}{p_1}.
\end{equation*}
Together with (\ref{eq3.73}), this indicates that the identity  in the last inequality of (\ref{eq3.75}) holds, and (\ref{eq3.16}) is thus proved.

We next prove (\ref{eq3.17}) as follows. Since $u_{2k}\not\equiv0$ in $\R^2$ for sufficiently large $k>0$, we then deduce from \eqref{ineq:GNQ} and  (\ref{ineq:b1beta.mult})  that
\begin{equation}\label{eq3.33}
\begin{split}
 II:= &\inte(|\nabla u_{1k}|^2+|\nabla u_{2k}|^2)dx\\
 &-\inte \Big(\frac{a_k}{2}|u_{1k}|^4 +\frac{b}{2}|u_{2k}|^4 +\beta_k|u_{1k}|^2|u_{2k}|^2\Big)dx\\
  \geq & \Big(\frac{a^*}{2\|u_{1k}\|_2^2}-\frac{a_k}{2}\Big)\inte|u_{1k}|^4dx+\Big(\frac{a^*}{2\|u_{2k}\|_2^2}-\frac{b}{2}\Big)\inte|u_{2k}|^4dx\\
  &-\beta_k\Big(\inte|u_{1k}|^4dx\inte|u_{2k}|^4dx\Big)^\frac12\\
   =&\frac{a^*-a_k}{2}\Big[\frac{1}{\|u_{1k}\|_2^2}+\frac{a_k}{a^*-a_k}\frac{\|u_{2k}\|_2^2}{1-\|u_{2k}\|_2^2}+\frac{a^*(1-\|u_{2k}\|_2^2)}{(a^*-a_k)\|u_{2k}\|_2^2}\frac{\|u_{2k}\|_4^4}{\|u_{1k}\|_4^4}\\
  &+\frac{a^*-b}{a^*-a_k}\frac{\|u_{2k}\|_4^4}{\|u_{1k}\|_4^4}-
  \frac{2\beta_k}{a^*-a_k}\frac{\|u_{2k}\|_4^2}{\|u_{1k}\|_4^2}\Big]\inte|u_{1k}|^4dx\ \ \text{as}\ \ k\to\infty,
\end{split}
\end{equation}
where we have used that $\|u_{1k}\|_2^2+\|u_{2k}\|_2^2 =1$.
For simplicity, we now set  $t_k:=\frac{\|u_{2k}\|_2^2}{1-\|u_{2k}\|_2^2}$£¬ and note that
\begin{equation}\label{eq3.34}
\begin{split}
&\frac{a_k}{a^*-a_k}\frac{\|u_{2k}\|_2^2}{1-\|u_{2k}\|_2^2}
+\frac{a^*(1-\|u_{2k}\|_2^2)}{(a^*-a_k)\|u_{2k}\|_2^2}\frac{\|u_{2k}\|_4^4}{\|u_{1k}\|_4^4}\\
 =&\Big(\sqrt{\frac{a_kt_k}{a^*-a_k}}
 -\sqrt{\frac{a^*}{t_k(a^*-a_k)}}\frac{\|u_{2k}\|_4^2}{\|u_{1k}\|_4^2}\Big)^2
 +\frac{2\sqrt{a_ka^*}}{a^*-a_k}\frac{\|u_{2k}\|_4^2}{\|u_{1k}\|_4^2}\\
 \geq &\frac{2\sqrt{a_ka^*}}{a^*-a_k}\frac{\|u_{2k}\|_4^2}{\|u_{1k}\|_4^2}\ \ \text{as}\ \ k\to\infty.
\end{split}
\end{equation}
In view of \eqref{lim:betaa1.u2:u1.mult}, \eqref{4.1:3M}  and (\ref{eq3.16}), we have $\sqrt{\frac{a^*}{t_k(a^*-a_k)}}\frac{\|u_{2k}\|_4^2}{\|u_{1k}\|_4^2}\overset{k}\to+\infty$.  Therefore,  the ``=" in the last inequality of (\ref{eq3.34}) holds true if and only if
\begin{equation}\label{eq3.35}
\sqrt{\frac{a_kt_k}{a^*-a_k}}=\sqrt{\frac{a^*}{t_k(a^*-a_k)}}\frac{\|u_{2k}\|_4^2}{\|u_{1k}\|_4^2}+o(1)\ \ \text{as}\ \ k\to\infty.
\end{equation}
By (\ref{eq3.33}) and (\ref{eq3.34}), we obtain that
\begin{equation}\label{eq3.79}
\begin{split}
  II&\geq\frac{a^*-a_k}{2}\Big[1+\frac{2\sqrt{a_ka^*}-2\beta_k}{a^*-a_k}\frac{\|u_{2k}\|_4^2}{\|u_{1k}\|_4^2}
  +\frac{a^*-b}{a^*-a_k}\frac{\|u_{2k}\|_4^4}{\|u_{1k}\|_4^4}\Big]\inte|u_{1k}|^4dx\\
  &=\frac{a^*-a_k}{2}\Big[1-\Big(\frac{2(\beta_k-a^*)}{a^*-a_k}+\frac{2\sqrt{a^*}}{\sqrt{a^*}+\sqrt{a_k}}\Big)
  \frac{\|u_{2k}\|_4^2}{\|u_{1k}\|_4^2}\\
  &\qquad\quad +\frac{a^*-b}{a^*-a_k}\frac{\|u_{2k}\|_4^4}{\|u_{1k}\|_4^4}\Big]\inte|u_{1k}|^4dx\\
  &=\frac{a^*-a_k}{2}\Big[1-\frac{(\beta_k-a^*)^2}{(a^*-a_k)(a^*-b)}+o(1)\Big]\inte|u_{1k}|^4dx\ \ \text{as}\ \ k\to\infty,
\end{split}
\end{equation}
where (\ref{eq3.16}) is used in the last equality. Using \eqref{eq3.74} and (\ref{eq3.79}) we can obtain that
\begin{equation*}
\lim\limits_{k\to\infty}\frac{e(a_k,b,\beta_k)}{\varepsilon_k^{p_1}}
\geq\frac{\lambda^{2+p_1}}{a^*(a^*-b)}\frac{p_1 +2}{p_1}.
\end{equation*}
By (\ref{eq3.73}),  this  yields  that all equalities in (\ref{eq3.34}) and (\ref{eq3.79}) hold true. Therefore, (\ref{eq3.35}) is true, and then (\ref{eq3.17}) follows by applying (\ref{eq3.16}).

As for (\ref{eq3.18}), we note that $(u_{1k}, u_{2k})$ solves the system \eqref{4.1:1a} with the Lagrange multiplier $\mu _{k}$ satisfying $\mu_k\eps_k^2\to-1$ as $k\to\infty$.
By (\ref{eq3.16}) and (\ref{eq3.17}), we thus derive from above that
\begin{equation*}
\begin{split}
&\inte \big(|\nabla u_{2k}|^2+V_2(x)|u_{2k}|^2\big)dx\\
 \leq &\,\mu_k\inte|u_{2k}|^2dx+b\inte|u_{2k}|^4dx+\beta_k\|u_{1k}\|_4^2\|u_{2k}\|_4^2\\
\leq &\,C\|u_{1k}\|_4^4\frac{\beta_k-a^*}{a^*-b}+b\big(\frac{\beta_k-a^*}{a^*-b}\big)^2\|u_{1k}\|_4^4.
\end{split}
\end{equation*}
This estimate then completes the proof of (\ref{eq3.18}), and the proof of the lemma is therefore proved.
\end{proof}

\begin{proof}[\bf Proof of Theorem \ref{thm3.7}.]  Define
\begin{equation}\label{4.1:1}
\begin{split}
 \bar u_{2k}(x):= \sqrt{\frac{a^*(a ^*-b)}{\beta_k-a^*}}\varepsilon_{k}u_{2k}(\varepsilon_{k}x+x_{2k}),
\end{split}
\end{equation}
where $x_{2k}\in\R^2$ is a global maximum point of $u_{2k}$, and $\varepsilon_{k}>0$ is defined by (\ref{def:betaa1.eps.rate}).
It then follows from (\ref{eq3.16})-(\ref{eq3.18}) that
\begin{equation}\label{eq3.44}
\lim_{k\to\infty}\inte\bar u_{2k}^2dx=a^*,\,\, \lim_{k\to\infty}\inte\bar u_{2k}^4dx=2a^*
\text{ and }\inte|\nabla \bar u_{2k}|^2dx\leq C<\infty.
\end{equation}
Following these, one can derive (see \cite{GWZZ,GZZ,LL}) that there exists $C>0$, independent of $k$, such that
$$\bar u_{2k}(0)=\|u_{2k}(x)\|_\infty>C>0.$$
We also recall from \eqref{4.1:1a}   that $ \bar u_{2k}$ satisfies
\begin{equation}\label{eq3.46}
\begin{split}
&-\Delta \bar u_{2k}+\eps_k^2V_2(\eps_kx+x_{2k})\bar u_{2k}\\
=&\mu_k \eps_k^2 \bar u_{2k}+\frac{\beta_k-a^*}{a^*(a^*-b)}\bar u_{2k}^3+\frac{\beta_k}{a^*} \bar u_{1k}^2\Big(x+\frac{x_{2k}-x_{1k}}{\eps_k}\Big)\bar u_{2k}\ \ \text{in}\,\ \R^2,
\end{split}\end{equation}
where $\bar u_{1k}$ is given by \eqref{lim:betaa1.u.rate}.
It then follows from  the  De Giorgi-Nash-Moser theory (c.f. \cite[Theorem 4.1]{HL}) that
\begin{equation}\label{eq3.4777}
\int_{|x|<1}|\bar u_{2k}|^2dx\geq C\bar u_{2k}(0)\geq \tilde C>0.
\end{equation}
Further, since it yields from (\ref{eq3.44}) that $\{\bar u_{2k}\}$ is bounded uniformly in $H^1(\R^2)$, then there exists $0\leq \bar u_0(x)\in H^1(\R^2)$ such that $\bar u_{2k}\overset{k}\rightharpoonup \bar u_0$ in $H^1(\R^2)$, and we derive from (\ref{eq3.4777}) that $\bar u_0(x)\not \equiv 0$ in $\R^2$.
On the other hand, following the proof of \eqref{3.4:5}, one can obtain that, up to a subsequence if necessary,
$\frac{x_{2k}-x_{1k}}{\eps_k}\to x_0$ for some $x_0\in\R^2$.
Hence, it follows from (\ref{eq3.44}) that $\bar u_0$ solves the elliptic PDE
 \begin{equation}
-\Delta \bar u_0(x) +\bar u_0(x)=w^2(x+x_0)\bar u_0 (x)\ \ \text{in}\, \ \R^2.
\end{equation}
By \cite[Lemma 4.1]{Wei96}, we thus conclude from above and (\ref{eq3.44}) that
\begin{equation}\label{3.1:gamma}
\bar u_0(x)= \gamma _0w(x+x_0)\ \ \text{in}\, \ \R^2,\, \   0<\gamma \le 1.
\end{equation}

\textbf{We claim  that $\gamma  _0=1$ in (\ref{3.1:gamma}).} Actually, for any $\delta>0$ one can choose $R>0$ large that
\begin{equation}\label{eq3.51}
\int_{R\le |x|<R+1}|\bar u_0|^2dx\le \delta^2\ \text{ and }\ 2\int_{|x|\ge R}|w(x+x_0)|^4dx\le\delta^2.
\end{equation}
For above fixed $R>0$, we then choose a cut-off function $\varphi_R(x)\in C^2(\R^2)$ such that $\varphi_R(x)\equiv0$ for $|x|\le R$, $\varphi_R(x)\in (0,1]$ for $R< |x|< R+1$ and  $\varphi_R(x)\equiv1$ for $|x|\ge R+1$, where $|\nabla\varphi_R(x)|\leq C$ holds for $C>0$ independent of $R$.
Multiplying both sides of (\ref{eq3.46}) by $\varphi_R \bar u_{2k}$ and integrating   over $\{x\in\R^2:|x|\ge R\}$,  it then follows that
\begin{equation}\label{eq3.52}
\begin{split}
&\int_{|x|\ge R}\varphi_R|\nabla\bar u_{2k}|^2
+\int_{R\le |x|<R+1}\bar u_{2k}\nabla \bar u_{2k}\nabla \varphi_R
-\mu_k\eps_k^2\int_{|x|\ge R}\varphi_R\bar u_{2k}^2\\
\leq& \frac{\beta_k-a^*}{a^*(a^*-b)}\inte\bar u_{2k}^4
+\frac{\beta_k}{a^*} \Big[\int_{|x|\ge R}\bar u_{1k}^4\big(x+\frac{x_{2k}-x_{1k}}{\eps_k}\big)\Big]^\frac12\Big(\inte \bar u_{2k}^4\Big)^\frac12,
\end{split}
\end{equation}
where the H\"{o}lder inequality (\ref{ineq:b1beta.mult}) is used.
Since
\begin{equation}\label{eq3.53}
\bar u_{2k}\to \bar u_0 \text{ strongly in $L^p_{\rm loc}(\R^2)$ for any $2\leq p<\infty$}
,\end{equation}
it then follows from (\ref{eq3.44}) and (\ref{eq3.51}) that
\begin{equation*}
\Big|\int_{R\le |x|<R+1}\bar u_{2k}\nabla \bar u_{2k}\nabla \varphi_Rdx\Big|\leq C\|\nabla \bar u_{2k}\|_2\Big(\int_{R\le |x|<R+1}\bar u_{2k}^2dx\Big)^\frac12\leq C\delta.
\end{equation*}
We also derive from \eqref{lim:betaa1.u.rate} and (\ref{eq3.51}) that
$$ \int_{|x|\ge R}\bar u_{1k}^4\Big(x+\frac{x_{2k}-x_{1k}}{\eps_k}\Big)dx\leq 2\int_{|x|\ge R}|w(x+x_0)|^4dx\le\delta^2.$$
Since $\beta_k\searrow a^*$ and $\mu_k\eps_k^2\to -1$ as $k\to\infty$,  it then yields from  (\ref{eq3.52}) and above that
$$\int_{|x|\ge R+1} \big(|\nabla\bar u_{2k}|^2+|\bar u_{2k}|^2\big )\dx\leq C\delta  \, \text{  if  $k$  large enough},$$
where $\delta>0$ is arbitrary.
Using (\ref{eq3.53}), this yields that
\begin{equation}\label{eq3.55}
\bar u_{2k}(x)\overset{k}\to \bar u_0(x)=\gamma _0 w(x+x_0)  \text{ strongly in } L^2(\R^2).
\end{equation}
We then deduce from (\ref{eq3.44}) that $a^*=\lim_{k\to\infty}\|\bar u_{2k}\|_2^2=\|\gamma  _0w\|_2^2=\gamma^2 _0 a^*$. Therefore, we have $\gamma _0=1$ in (\ref{3.1:gamma}), and the claim is thus proved.

The same argument of proving Lemma \ref{lem3.4A} further gives that $x_0=0$, and $x_{2k}$ is the unique maximum point of  $u_{2k}$. Therefore, we now conclude that $\bar u_{2k}\overset{k}\rightharpoonup w$ in $H^1(\R^2)$. Similar to those in \cite{GLWZ}, one can also derive from (\ref{4.1:1a}) that $( u_{1k}, u_{2k})$ satisfies the exponential decay (\ref{4:conexp1}) and (\ref{4:conexp2}). Following these, the standard elliptic regularity theory further yields that $\bar u_{2k}\to w$  uniformly in $\R^2$ as $k\to\infty$. Therefore, (\ref{thm1.2:M}) holds true in view of \eqref{lim:betaa1.u.rate}, which then completes the proof of Theorem \ref{thm3.7}.
\end{proof}

\section{Uniqueness of Nonnegative Minimizers}

This section is devoted to the proof of Theorem \ref{thm1.3} on the uniqueness of nonnegative minimizers  for $e(a, b, \beta)$, where    $a\nearrow a^*$ and $0<b<a^*$ and $\beta \ge a^*$ satisfies (\ref{5:beta}).  We first note that $ (u_0, v_0)= (w, w)$ is a positive solution of the following system
\begin{equation}\label{uniq:limit-1}
\arraycolsep=1.5pt
\left\{\begin{array}{lll}
\ \Delta u -  u +  u ^3   &=0   \,\ \mbox{in}\,\  \R^2,\\ [2mm]
 \Delta v -  v +    u ^2 v &=0   \,\ \mbox{in}\,\  \R^2,\,\
\end{array}\right.
\end{equation}
where $w>0$ is a unique positive solution of (\ref{equ:w}).
We claim  that the positive solution $(u_0,v_0)= (w, w)$ is {\em non-degenerate}, in the sense that the solution set of the linearized system for (\ref{uniq:limit-1}) about $(u_0,v_0)$ satisfying
\begin{equation} \label{uniq:limit-2} \arraycolsep=1.5pt
\left\{\begin{array}{lll}
\quad \ \mathcal{L}_1(\phi_1)&:=&\Delta \phi_1-  \phi_1 +3u_0^2 \phi_1 =0   \,\ \mbox{in}\,\  \R^2, \\ [2.5mm]
\mathcal{L}_2(\phi_1)\phi_2&:=&\Delta \phi_2-  \phi_2+u_0^2 \phi_2 +2u_0 v_0\phi_1 =0   \,\ \mbox{in}\,\  \R^2,\,\
 \end{array}\right.\end{equation}
satisfies
\begin{equation}\label{uniq:limit-3}
\left(\begin{array}{cc} \phi_1\\[2mm]
 \phi_2\end{array}\right)=b_0\left(\begin{array}{cc} 0\\[2mm]
w\end{array}\right)+\sum _{j=1}^{2}b_j\left(\begin{array}{cc} \frac{\partial w}{\partial x_j}\\[2mm] \frac{\partial w}{\partial x_j}
\end{array}\right)
\end{equation}
for some constants $b_j$, where $j=0, 1, 2$. Actually, one can note from  \cite[Lemma 4.1]{Wei96} that the solution set of $\mathcal{L}_1\phi_1=-\Delta \phi_1+\big[1-3w^2\big] \phi_1 =0$ in $\R^2$ satisfies $\phi_1=\sum _{j=1}^{2}b_j\frac{\partial w}{\partial x_j}$ for some constants $b_j$, where $j= 1, 2$. Since it also follows from \cite[Lemma 4.1]{Wei96} that the solution set of $-\Delta \phi +(1- w^2 )\phi =0$ in $\R^2$ satisfies $ \phi=b_0w$ for some constant $b_0$, the claim (\ref{uniq:limit-3}) therefore follows.


For convenience, in the following we always suppose that $(u_k,v_k)$ is a nonnegative minimizer of $e(a_k,b,\beta_k)$ satisfying  (\ref{5:beta}).
Then $(u_k,v_k)$ satisfies
\begin{equation}\label{uniq:a-1H}
\arraycolsep=1.5pt
\left\{\begin{array}{lll}
-\Delta u_k +V_1(x)u_k =\mu_{ k} u_k +a_ku_k^3 +\beta _kv_k^2 u_k   \,\ &\mbox{in}\,\  \R^2,\\[2mm]
-\Delta v_k +V_2(x)v_k \, =\mu_{k} v_k +bv_k^3 +\beta _ku_k^2 v_k   \ \ &\mbox{in}\,\  \R^2,\,\
\end{array}\right.
\end{equation}
where $\mu _{k}\in \R$ is a suitable Lagrange multiplier and satisfies
\begin{equation}\label{uniq:a-6H}
 \mu_k=e(a_k,b,\beta _k)- \displaystyle \frac{a_k}{2} \inte   u_k^4 - \displaystyle \frac{b}{2 }\inte    v_k^4-\displaystyle \beta _k \inte u_k^2v_k^2.
\end{equation}
Under the assumptions of Theorem \ref{thm1.3}, one can further check from (\ref{uniq:a-6H}) and previous sections that $\mu _{k}$ satisfies
\begin{equation}\label{uniq:a-6HH}
\lim_{k\to\infty} \mu_k\varepsilon^2_{k}=- 1,
\end{equation}
where $\varepsilon _k>0$ is defined by
\begin{equation}\label{4:eps}
  \varepsilon_k:=\frac{1}{\lambda}\Big[(a^*-a_k)(a^*-b)-(\beta _k-a^*)^2\Big]^\frac{1}{2+p_1}>0,\,\  \lambda =\Big[\frac{p_1(a^*-b)}{2}H_1(y_0)\Big]^\frac{1}{2+p_1},
\end{equation}
and $H_1(y_0)>0$ is defined in \eqref{H}.

\subsection{Proof of Theorem \ref{thm1.3}}

In this subsection we shall complete the proof of Theorem \ref{thm1.3}. Since $(a_k,b,\beta_k)$ satisfies (\ref{5:beta}), we then obtain from  Proposition \ref{Prop:a1beta.3} that the nonnegative solution $(u_k,v_k)$ of $e(a_k,b,\beta _k)$  satisfies the limit behavior (\ref{lim:betaa1.u.rate}) as $k\to\infty$. It further follows from Theorem \ref{thm:1.1} that $v_k\not\equiv 0$ for sufficiently large $k>0$.
By Theorem \ref{thm:1.1} and Lemma \ref{lem3.4A}, if $v_k(\varepsilon_{k}x+y_k)=C_\infty \sigma _k\bar v_k(x)$, where $\sigma _k=\|v_k\|_\infty>0$, $C_\infty=\frac{1}{\|w\|_\infty}> 0$ and $y_k\in\R^2$ is a unique maximal point of $v_k$, then we have
\begin{equation}\label{Uniq:3}
  \bar v_k(x)\to w(x)\,\  \text{uniformly in $\R^2$ as} \,\ k\to\infty.
\end{equation}
We now suppose that there exist two different nonnegative minimizers $(u_{1,k},v_{1,k})$ and $(u_{2,k},v_{2,k})$ of $e(a_k,b,\beta _k)$ satisfying (\ref{5:beta}). Let $(x_{1,k},y_{1,k})$ and $(x_{2,k},y_{2,k})$ be the unique maximum point of $(u_{1,k},v_{1,k})$ and $(u_{2,k},v_{2,k})$, respectively. Note from (\ref{uniq:a-1H}) that the nonnegative minimizer $(u_{i,k},v_{i,k})$ solves the system
\begin{equation}\label{uniq:a-1}
 \begin{cases}
-\Delta u_{i,k} +V_1(x)u_{i,k} =\mu_{i,k} u_{i,k} +a_ku_{i,k}^3 +\beta _k   v_{i,k}^2 u_{i,k}   \,\ \mbox{in}\,\  \R^2,\\
-\Delta v_{i,k} +V_2(x)v_{i,k}\, =\mu_{i,k} v_{i,k} +bv_{i,k}^3 +\beta _k   u_{i,k}^2 v_{i,k}  \ \ \ \ \mbox{in}\,\  \R^2,
\end{cases}
\end{equation}
where $\mu _{i,k}\in \R$ is a suitable Lagrange multiplier satisfying (\ref{uniq:a-6H}) and (\ref{uniq:a-6HH}) with $\mu _k=\mu _{i,k}$ for $i=1, 2$.
Motivated by (\ref{Uniq:3}), we define
\begin{equation}\label{uniq:a-2}
\bar u_{i,k}(x)=\sqrt{a^*} \varepsilon_{k}  u_{i,k}(\varepsilon_{k}x+x_{2,k})\,\  \text{and}\,\ v_{i,k}(\varepsilon_{k}x+x_{2,k})=C_\infty \sigma _k\bar v_{i,k}(x),
\end{equation}
where $i=1, 2$, $\sigma _k=\|v_{2,k}\|_\infty>0$ and $C_\infty=\frac{1}{\|w\|_\infty}> 0$.
By Theorem \ref{thm:1.1}, we then obtain from (\ref{lim:betaa1.u.rate}), (\ref{4.1:3D}) and (\ref{Uniq:3}) that
\begin{equation}\label{uniq:a-3}
\big(\bar u_{i,k}(x), \bar v_{i,k}(x)\big)\to \big(u_0,v_0\big)\equiv (w, w)
\end{equation}
uniformly in $\R^2$ as $k\to\infty$, where $(u_0 ,v_0 )= (w, w)$ is a positive solution of the system (\ref{uniq:limit-1}), and  $\big(\bar u_{i,k}(x), \bar v_{i,k}(x)\big)$ satisfies the system
\begin{equation}\label{uniq:a-4}\arraycolsep=1.5pt
 \begin{cases}
\quad-\Delta \bar u_{i,k}+\eps _k^2V_1(\eps _k x+x_{2k})\bar u_{i,k}\\
=  \mu_{i,k}\eps _k^2 \bar u_{i,k}+\displaystyle\frac{a_k}{a^*}\bar u_{i,k}^3+\beta _k C_\infty^2\sigma _k^2\eps_k^2\bar v_{i,k}^2\bar u_{i,k}\ \ \text{in}\,\ \R^2,\\
\quad -\Delta \bar v_{i,k}+\eps _k^2V_2(\eps _k x+x_{2k})\bar v_{i,k}\\
=\mu_{i,k}\eps _k^2 \bar v_{i,k}+bC_\infty^2\sigma _k^2\eps_k^2\bar v_{i,k}^3+\displaystyle\frac{\beta _k}{a^*} \bar u_{i,k}^2 \bar v_{i,k}\ \ \text{in}\,\ \R^2.
\end{cases}
\end{equation}
One can check that $\varepsilon_{k}^2(\mu_{2,k}-\mu_{1,k})$ satisfies
\begin{equation}\label{uniq:a-6}\arraycolsep=1.5pt
 \begin{array}{lll}
\varepsilon_{k}^2(\mu_{2,k}-\mu_{1,k})&=- \displaystyle \frac{a_k}{2}\varepsilon_{k}^2\inte (  u_{2,k}^4- u_{1,k}^4)- \displaystyle \frac{b}{2 }\varepsilon_{k}^2\inte (  v_{2,k}^4- v_{1,k}^4)\\[3mm]
&\quad-\displaystyle \beta _k \varepsilon_{k}^2\inte \Big[v_{2,k}^2(  u_{2,k}^2- u_{1,k}^2)+u_{1,k}^2(  v_{2,k}^2- v_{1,k}^2)\Big]\\[3mm]
&=- \displaystyle \frac{a_k}{2(a^*)^2}\inte ( \bar u_{2,k}^4- \bar u_{1,k}^4)- \displaystyle \frac{b}{2} C_\infty^4\sigma _k^4\eps_k^4\inte (  \bar v_{2,k}^4- \bar v_{1,k}^4)\\[3mm]
\quad &\quad-\displaystyle\frac{\beta _k }{a^*}  C_\infty^2\sigma _k^2\eps_k^2\inte \Big[\bar v_{2,k}^2( \bar  u_{2,k}^2- \bar u_{1,k}^2)+\bar u_{1,k}^2( \bar  v_{2,k}^2- \bar v_{1,k}^2)\Big],
\end{array}
\end{equation}
where we have used (\ref{uniq:a-6H}). Recall also from Proposition \ref{Prop:a1beta.3} that both $\bar u_{i,k}(x)$ and $ \nabla \bar u_{i,k}(x)$ decay exponentially as $|x|\to\infty$ for $i=1$ and $2$. Further, one can derive from (\ref{uniq:a-4}) that both $\bar v_{i,k}(x)$ and $\nabla \bar v_{i,k}(x)$ also admit the similar exponential decay as $|x|\to\infty$ for $i=1$ and $2$.

%

We also define
\begin{equation}\label{uniq:B-1}
\big(\hat u_{i,k}(\varepsilon_{k} x+x_{2,k}), \hat v_{i,k}(\varepsilon_{k} x+x_{2,k})\big):=   \big( \bar u_{i,k}( x), \bar  v_{i,k}(x)\big), \ \ \mbox{where}\ \ i=1,2,
\end{equation}
so that
\begin{equation}\label{uniq:B-2}
\big(\hat u_{i,k}(\varepsilon_{k} x+x_{2,k}), \hat v_{i,k}(\varepsilon_{k} x+x_{2,k})\big)\to \big(u_0,v_0\big)\equiv (w, w)
\end{equation}
uniformly in $\R^2$ as $k\to\infty$ by (\ref{uniq:a-3}).
Note that $\big(\hat u_{i,k}(x), \hat v_{i,k}(x)\big)$ satisfies the system
\begin{equation}\label{5.2:0}\arraycolsep=1.5pt
\left\{\begin{array}{lll}
&- \varepsilon_{k}^2\Delta \hat u_{i,k} +\varepsilon_{k}^2V_1(x)\hat u_{i,k} \\[2mm]
=&\mu_{i,k} \varepsilon_{k}^2\hat u_{i,k} +\displaystyle\frac{a_k}{a^*}\hat u_{i,k}^3 +\displaystyle \beta _k  C_\infty^2\sigma _k^2\eps_k^2 \hat v_{i,k}^2 \hat u_{i,k}   \,\, \mbox{in}\,\  \R^2,\\[2.5mm]
&- \varepsilon_{k}^2\Delta \hat v_{i,k} +\varepsilon_{k}^2V_2(x)\hat v_{i,k}  \\[2mm]
=&\mu_{i,k} \varepsilon_{k}^2\hat v_{i,k} +b C_\infty^2\sigma _k^2\eps_k^2\hat v_{i,k}^3 +\displaystyle\frac{\beta _k  }{a^*} \hat u_{i,k}^2 \hat v_{i,k} \ \  \,\ \mbox{in}\,\  \R^2.\,\
\end{array}\right.
\end{equation}
Since $(u_{1,k},v_{1,k})\not \equiv (u_{2,k},v_{2,k})$,  we define
\begin{equation}\label{uniq:B-7}\arraycolsep=1.5pt
 \begin{array}{lll}
  \hat  \xi_{1,k}(x)&=\displaystyle\frac{ \hat   u_{2,k}(x)-   \hat  u_{1,k}(x)}{\|  \hat  u_{2,k}-  \hat   u_{1,k}\|_{L^\infty(\R^2)}+\frac{1}{\eps _k}\| \hat   v_{2,k}-  \hat  v_{1,k}\|_{L^2(\R^2)}},\\[4mm]
 \hat \xi_{2,k}(x)&=\displaystyle\frac{ \hat   v_{2,k}(x)-  \hat   v_{1,k}(x)}{\| \hat   u_{2,k}-   \hat  u_{1,k}\|_{L^\infty(\R^2)}+\frac{1}{\eps _k}\| \hat   v_{2,k}- \hat   v_{1,k}\|_{L^2(\R^2)}},
\end{array}
\end{equation}
which is different from those used in \cite{GLWZ}. We then have the following local estimates of $(\hat\xi_{1,k}, \hat\xi_{2,k})$.

\begin{lemma}\label{lem4.3} Assume that $(a_k, b, \beta _k)$ satisfies (\ref{5:beta}). Then for any $x_0\in\R^2$, there exists a small constant $\delta >0$  such that
\begin{equation}
    \int_{\partial B_\delta (x_0)} \Big( \eps ^2_k |\nabla \hat \xi_{i,k}|^2+ \frac{1}{2} \hat \xi_{i,k}^2+ \eps ^2_k  V_i(x)\hat \xi_{i,k}^2\Big)dS=O( \eps ^2_k)\,\ \text{as}\,\ k\to\infty,\,\   i=1, 2.
\label{5.2:6}
\end{equation}
\end{lemma}

The proof of Lemma \ref{lem4.3} is given in Appendix A. Associated to $(\hat \xi_{1,k},\hat \xi_{2,k})$, it is also convenient to define
\begin{equation}\label{uniq:a-7}\arraycolsep=1.5pt
\begin{array}{lll}
\xi_{1,k}(x)&=\displaystyle\frac{ \bar  u_{2,k}(x)-   \bar u_{1,k}(x)}{\|  \bar u_{2,k}-  \bar  u_{1,k}\|_{L^\infty(\R^2)}+\| \bar  v_{2,k}-  \bar v_{1,k}\|_{L^2(\R^2)}},\\[4mm]
\xi_{2,k}(x)&=\displaystyle\frac{ \bar  v_{2,k}(x)-  \bar  v_{1,k}(x)}{\| \bar  u_{2,k}-   \bar u_{1,k}\|_{L^\infty(\R^2)}+\| \bar  v_{2,k}- \bar  v_{1,k}\|_{L^2(\R^2)}},
\end{array}
\end{equation}
so that
\begin{equation}\label{uniq:B-3}
\xi_{i,k}(x)=\hat\xi_{i,k}(\eps_kx+x_{2,k}),\, \ \text{where} \, \ i=1, 2.
\end{equation}
In the following we shall complete the proof of Theorem \ref{thm1.3} by considering separately three different cases.

(1).\, We now consider the first case where $u_{2,k}\not \equiv u_{1,k}$ and $v_{2,k}\not \equiv v_{1,k}$ in $\R^2$, for which we shall continue the proof of Theorem \ref{thm1.3} by the following six steps:

\vskip 0.1truein

\noindent{\em  Step 1.}
There exists a subsequence (still denoted by $\{a_k\}$) of $\{a_k\}$ such that
\[(\xi_{1,k}, \xi_{2,k})\to (\xi_{10}, \xi_{20}) \,\ \text{in}\,\  C_{loc}(\R^2)  \,\ \text{in}\,\   k\to\infty,\]
where $(\xi_{10}, \xi_{20})$ satisfies
\begin{equation}\label{uniq:limit-A3}
\left(\begin{array}{cc}   \xi_{10}\\[2mm]
\xi_{20}\end{array}\right)=\displaystyle b_0 \left(\begin{array}{cc}   0 \\[2mm]
w\end{array}\right)+\sum _{j=1}^{2}b_j\left(\begin{array}{cc} \frac{\partial w}{\partial x_j}\\[2mm] \frac{\partial w}{\partial x_j}
\end{array}\right)+c_0 \left(\begin{array}{cc}  w+x\cdot \nabla w \\[2mm]
w+x\cdot \nabla w \end{array}\right)
\end{equation}
for some constants $c_0$ and $b_j$ with $j=0, 1, 2$.

Following (\ref{uniq:a-4}), one can check from (\ref{uniq:a-7}) that $(\xi_{1,k}, \xi_{2,k})$ satisfies
\begin{equation}\label{step-1:9}\arraycolsep=1.5pt
 \left\{\begin{array}{lll}
&  \Delta \xi_{1,k}-\varepsilon_{k}^2V_1(\varepsilon_{k} x+x_{2,k})\xi_{1,k}+\mu_{2,k} \varepsilon_{k}^2\xi_{1,k}\\[2mm]
&+\displaystyle\frac{a_k}{a^*}\big(\bar u_{2,k}^2+\bar u_{2,k}\bar u_{1,k}+\bar u_{1,k}^2\big)\xi_{1,k} \\[2mm]
 &+\displaystyle \beta _k  C_\infty^2\sigma _k^2\eps_k^2 \big[\bar  v_{1,k}^2\xi_{1,k}+\bar u_{2,k}( \bar v_{2,k}+\bar v_{1,k})\xi_{2,k}\big] =\displaystyle c_k\bar u_{1,k}  \,\ \mbox{in}\,\  \R^2,\\[2mm]
 &  \Delta \xi_{2,k}-\varepsilon_{k}^2V_2(\varepsilon_{k} x+x_{2,k})\xi_{2,k}+\mu_{2,k} \varepsilon_{k}^2\xi_{2,k}\\[2mm]
 &+\displaystyle b  C_\infty^2\sigma _k^2\eps_k^2 \big(\bar v_{2,k}^2+\bar v_{2,k}\bar v_{1,k}+\bar v_{1,k}^2\big)\xi_{2,k} \\[2mm]
 &+\displaystyle\frac{\beta _k  }{a^*}\big[\bar  u_{1,k}^2\xi_{2,k}+\bar v_{2,k}( \bar u_{2,k}+\bar u_{1,k})\xi_{1,k}\big] =\displaystyle c_k\bar v_{1,k} \,\ \mbox{in}\,\  \R^2.
\end{array}\right.
\end{equation}
Using (\ref{uniq:a-6}), the coefficient $c_{k}$ satisfies
\begin{equation}\label{step-1:10}
\arraycolsep=1.5pt
 \begin{array}{lll}
c_k:&=\displaystyle\frac{-\varepsilon_{k}^2(\mu_{2,k}-\mu_{1,k})}{\| \bar u_{2,k}-  \bar u_{1,k}\|_{L^\infty(\R^2)}+\| \bar v_{2,k}- \bar v_{1,k}\|_{L^2(\R^2)}}\\[3mm]
&= \displaystyle \frac{a_k}{2(a^*)^2}\inte ( \bar u_{2,k}^2+ \bar u_{1,k}^2)( \bar u_{2,k}+\bar u_{1,k})\xi_{1,k}\\[3mm]
&\quad +\displaystyle \frac{b}{2}C_\infty^4\sigma _k^4\eps_k^4\inte ( \bar v_{2,k}^2+ \bar v_{1,k}^2)( \bar v_{2,k}+\bar v_{1,k})\xi_{2,k}\\[3mm]
\quad &\quad +\displaystyle\frac{\beta _k}{a^*}C_\infty^2\sigma _k^2\eps_k^2\inte \Big[\bar v_{2,k}^2( \bar  u_{2,k}+ \bar u_{1,k})\xi_{1,k}+\bar u_{1,k}^2( \bar  v_{2,k}+ \bar v_{1,k})\xi_{2,k}\Big].
\end{array}\end{equation}
Since $\xi_{1,k}$ is bounded uniformly in $\R^2$ and $\xi_{2,k}$ is bounded uniformly  in $L^2(\R^2)$, the standard elliptic regularity theory (cf. \cite[Corollary 7.11]{GT}) then implies from (\ref{step-1:10}) that $\|\xi _{1,k}\|_{C^{\alpha }_{loc}(\R^2)}\le C$ for some $\alp \in (0,1)$, where the constant $C>0$ is independent of $k$. Therefore, up to a subsequence if necessary, we have $(\xi_{1,k}, \xi_{2,k})\to (\xi_{10}, \xi_{20})$ in $C_{loc}(\R^2)$ as $k\to\infty$, where the vector function $(\xi_{10}, \xi_{20})$ satisfies
\begin{equation}\label{uniq:limit-A1}
\arraycolsep=1.5pt
\left\{\begin{array}{lll}
\qquad \quad \quad\, \Delta \xi _{10}-  \xi_{10}+3u_0^2\xi_{10}&=&\displaystyle \frac{2}{a^*} \Big( \inte u_0^3\xi_{10}\Big)w    \ \ \mbox{in}\,\  \R^2,\\ [3mm]
 \Delta  \xi_{20}-    \xi_{20}  +   u_0^2   \xi_{20}+2  u_0 v_0  \xi_{10}&=&\displaystyle \frac{2}{a^*} \Big( \inte u_0^3\xi_{10}\Big)w   \ \ \mbox{in}\,\  \R^2,
\end{array}\right.
\end{equation}
and $ (u_0, v_0)= (w, w)$. We then obtain from (\ref{uniq:limit-A1}) that there exist constants $b_1$, $b_2$ and $c_0$ such that
\[
 \xi _{10}=b_1\frac{\partial w}{\partial x_1}+b_2\frac{\partial w}{\partial x_2}+c_0 (w+x\cdot \nabla w),\,\ c_0=\frac{1}{a^*}\inte w^3\xi_{10}.
\]
We thus derive from (\ref{uniq:limit-3}) and (\ref{uniq:limit-A1}) that $(\xi_{10}, \xi_{20})$ satisfies (\ref{uniq:limit-A3}) for some constants $c_0$ and $b_j$ with $j=0, 1, 2$, and Step 1 is thus established.

\vskip 0.05truein

\noindent{\em  Step 2.} We claim that if $\delta >0$ is small, we then have the following Pohozaev-type identities
 \begin{equation}\label{5.2:10}\arraycolsep=1.5pt\begin{array}{lll}
 o(e^{-\frac{C\delta}{\eps_k}})&=&\displaystyle\eps _k^{3+p_1} \intB  \frac{\partial V_1\big(x+\frac{x_{2,k}}{\eps_k}\big)}{\partial  x_j}\big(\bar u_{2,k}+\bar u_{1,k}\big)   \xi _{1,k}\\[3mm]
&&+\displaystyle\eps _k^{3+p_2} (a^*C_\infty^2\sigma _k^2\eps_k^2)\intB \frac{\partial V_2\big(x+\frac{x_{2,k}}{\eps_k}\big)}{\partial  x_j}\big(\bar v_{2,k}+\bar v_{1,k}\big)\xi _{2,k}dx
\end{array} \end{equation}
as $k\to\infty$, where $j=1,\,2$.

To prove the above claim, multiply the first equation of (\ref{5.2:0}) by $\frac{\partial \hat u_{i,k}}{\partial  x_j}$, where $i,j=1,2$, and integrate over $B_\delta (x_{2,k})$, where $\delta >0$ is small and given by (\ref{5.2:6}). It then gives that
\begin{equation}\arraycolsep=1.5pt\begin{array}{lll}
&&-\eps _k^2\displaystyle\intB\frac{\partial \hat u_{i,k}}{\partial  x_j}\Delta \hat u_{i,k}+\eps _k^2\displaystyle\intB V_1(x)\frac{\partial \hat u_{i,k}}{\partial  x_j} \hat u_{i,k}\\[4mm]
&=&\mu_{i,k}\eps _k^2\displaystyle\intB \frac{\partial \hat u_{i,k}}{\partial  x_j} \hat u_{i,k}+\displaystyle\frac{a_k}{a^*}\intB \frac{\partial \hat u_{i,k}}{\partial  x_j} \hat u_{i,k}^3\\[4mm]
&&
+\displaystyle \beta _k C_\infty^2\sigma _k^2\eps_k^2\intB \frac{\partial \hat u_{i,k}}{\partial  x_j} \hat u_{i,k}\hat v_{i,k}^2\\[4mm]
&=&\displaystyle\frac{\mu_{i,k}}{2}\eps _k^2\intPB \hat u_{i,k}^2\nu _jdS+\displaystyle\frac{a_k}{4a^*}\intPB \hat u_{i,k}^4\nu _jdS\\[4mm]
&&+\displaystyle\frac{1}{2}\beta _k C_\infty^2\sigma _k^2\eps_k^2\intB \frac{\partial \hat u_{i,k}^2}{\partial  x_j}  \hat v_{i,k}^2,
\end{array}\label{5.2:7}
\end{equation}
where $\nu =(\nu _1,\nu _2)$ denotes the outward unit normal of $\partial B_\delta (x_{2,k})$.
Note that
\[\arraycolsep=1.5pt\begin{array}{lll}
 &&-\eps _k^2\displaystyle\intB\frac{\partial \hat u_{i,k}}{\partial  x_j}\Delta \hat u_{i,k}\\[4mm]&=&-\eps _k^2\displaystyle\intPB\frac{\partial \hat u_{i,k}}{\partial  x_j}\frac{\partial \hat u_{i,k}}{\partial  \nu}dS+\eps _k^2\displaystyle\intB\nabla \hat u_{i,k}\cdot\nabla\frac{\partial \hat u_{i,k}}{\partial  x_j}\\[4mm]
 &=&-\eps _k^2\displaystyle\intPB\frac{\partial \hat u_{i,k}}{\partial  x_j}\frac{\partial \hat u_{i,k}}{\partial  \nu}dS+\displaystyle\frac{1}{2}\eps _k^2\intPB |\nabla \hat u_{i,k}|^2\nu _jdS,
\end{array}
\]
and
\[
 \eps _k^2\displaystyle\intB V_1(x)\frac{\partial \hat u_{i,k}}{\partial  x_j} \hat u_{i,k}= \frac{\eps _k^2}{2}\intPB V_1(x)\hat u_{i,k}^2\nu _jdS-\frac{\eps _k^2}{2}\intB \frac{\partial V_1(x)}{\partial  x_j}\hat u_{i,k}^2.
\]
We then derive from (\ref{5.2:7}) that
\begin{equation}\arraycolsep=1.5pt\begin{array}{lll}
&&\displaystyle\eps _k^2 \intB \frac{\partial V_1(x)}{\partial  x_j}\hat u_{i,k}^2+\displaystyle \beta _k C_\infty^2\sigma _k^2\eps_k^2\intB \frac{\partial \hat u^2_{i,k}}{\partial  x_j} \hat v_{i,k}^2\\[4mm]
&=&-2\eps _k^2\displaystyle\intPB\frac{\partial \hat u_{i,k}}{\partial  x_j}\frac{\partial \hat u_{i,k}}{\partial  \nu}dS+\displaystyle \eps _k^2\intPB |\nabla \hat u_{i,k}|^2\nu _jdS \\[4mm]
&& +\displaystyle\eps _k^2 \intPB V_1(x)\hat u_{i,k}^2\nu _jdS-\displaystyle \mu_{i,k}\eps _k^2\intPB \hat u_{i,k}^2\nu _jdS\\[4mm]
&&-\displaystyle\frac{ a_k}{2a^*}\intPB \hat u_{i,k}^4\nu _jdS:=\mathcal{C}_i.
\end{array}\label{5.2:8A}
\end{equation}
Similarly, we derive from the second equation of (\ref{5.2:0}) that
\[\arraycolsep=1.5pt\begin{array}{lll}
&&\displaystyle\eps _k^2 \intB \frac{\partial V_2(x)}{\partial  x_j}\hat v_{i,k}^2+\displaystyle\frac{\beta _k }{a^*}\intB \frac{\partial \hat v^2_{i,k}}{\partial  x_j} \hat u_{i,k}^2\\[4mm]
&=&-2\eps _k^2\displaystyle\intPB\frac{\partial \hat v_{i,k}}{\partial  x_j}\frac{\partial \hat v_{i,k}}{\partial  \nu}dS+\displaystyle \eps _k^2\intPB |\nabla \hat v_{i,k}|^2\nu _jdS \\[4mm]
&& +\displaystyle\eps _k^2 \intPB V_2(x)\hat v_{i,k}^2\nu _jdS-\displaystyle \mu_{i,k}\eps _k^2\intPB \hat v_{i,k}^2\nu _jdS\\[4mm]
&&-\displaystyle\frac{ b}{2} C_\infty^2\sigma _k^2\eps_k^2   \intPB \hat v_{i,k}^4\nu _jdS,
\end{array}
\]
which then implies that
 \begin{equation}\arraycolsep=1.5pt\begin{array}{lll}
&&\displaystyle\eps _k^2 \intB \frac{\partial V_2(x)}{\partial  x_j}\hat v_{i,k}^2-\displaystyle\frac{\beta _k}{a^*}\intB \frac{\partial \hat u^2_{i,k}}{\partial  x_j} \hat v_{i,k}^2\\[4mm]
&=&-2\eps _k^2\displaystyle\intPB\frac{\partial \hat v_{i,k}}{\partial  x_j}\frac{\partial \hat v_{i,k}}{\partial  \nu}dS+\displaystyle \eps _k^2\intPB |\nabla \hat v_{i,k}|^2\nu _jdS \\[4mm]
&& +\displaystyle\eps _k^2 \intPB V_2(x)\hat v_{i,k}^2\nu _jdS-\displaystyle \mu_{i,k}\eps _k^2\intPB \hat v_{i,k}^2\nu _jdS\\[4mm]
&&- \displaystyle\frac{ b}{2} C_\infty^2\sigma _k^2\eps_k^2  \intPB \hat v_{i,k}^4\nu _jdS-\displaystyle\frac{\beta _k}{a^*}\intPB   \hat v^2_{i,k}  \hat u_{i,k}^2\nu_j dS:=\mathcal{D}_i.
\end{array}\label{5.2:8B}
\end{equation}

Following (\ref{5.2:8A}) and (\ref{5.2:8B}), we thus have
\begin{equation}\arraycolsep=1.5pt\begin{array}{lll}
&\displaystyle\eps _k^2 \intB  \frac{\partial V_1(x)}{\partial  x_j}\big(\hat u_{2,k}+\hat u_{1,k}\big) \hat \xi _{1,k}\\[3mm]
=&-\displaystyle\eps _k^2 (a^*C_\infty^2\sigma _k^2\eps_k^2)\intB \frac{\partial V_2(x)}{\partial  x_j}\big(\hat v_{2,k}+\hat v_{1,k}\big) \hat \xi _{2,k}\\
&+\mathcal{I}^u_k+(a^*C_\infty^2\sigma _k^2\eps_k^2)\,\mathcal{I}^v_k,
\label{5.2:8}
\end{array}\end{equation}
where we denote
\[\arraycolsep=1.5pt\begin{array}{lll}
\mathcal{I}^u_k&=&-2\displaystyle \eps _k^2\intPB \Big[\frac{\partial \hat u_{2,k}}{\partial  x_j}\frac{\partial \hat \xi_{1,k}}{\partial  \nu}+\frac{\partial \hat \xi_{1,k}}{\partial  x_j}\frac{\partial \hat u_{1,k}}{\partial  \nu}\Big]dS\\[4mm]
&&+\eps _k^2\displaystyle\intPB\nabla \hat \xi_{1,k} \cdot\nabla \big(\hat u_{2,k}+\hat u_{1,k}\big)\nu _jdS +\displaystyle \eps _k^2 \intPB V_1(x)\big(\hat u_{2,k}+\hat u_{1,k}\big) \hat \xi _{1,k} \nu _jdS\\[4mm]
&&-\hat c_k\displaystyle\intPB \hat u_{2,k}^2\nu _jdS-\displaystyle \mu_{1,k}\eps _k^2\intPB \big(\hat u_{2,k}+\hat u_{1,k}\big) \hat \xi _{1,k}\nu _jdS \\[4mm]
&& -\displaystyle\frac{ a_k}{2a^*}\intPB \big(\hat u^2_{2,k}+\hat u^2_{1,k}\big) \big(\hat u_{2,k}+\hat u_{1,k}\big)\hat \xi _{1,k}\nu _jdS,
\end{array}\label{5.2:9A}
\]
and
\[\arraycolsep=1.5pt\begin{array}{lll}
\mathcal{I}^v_k&=&-2\displaystyle \eps _k^2\intPB \Big[\frac{\partial \hat v_{2,k}}{\partial  x_j}\frac{\partial \hat \xi_{2,k}}{\partial  \nu}+\frac{\partial \hat \xi_{2,k}}{\partial  x_j}\frac{\partial \hat v_{1,k}}{\partial  \nu}\Big]dS\\[4mm]
&&+\eps _k^2\displaystyle\intPB\nabla \hat \xi_{2,k} \cdot\nabla \big(\hat v_{2,k}+\hat v_{1,k}\big)\nu _jdS+\displaystyle \eps _k^2 \intPB V_2(x)\big(\hat v_{2,k}+\hat v_{1,k}\big) \hat \xi _{2,k} \nu _jdS \\[4mm]
&&-\hat c_k\displaystyle\intPB \hat v_{2,k}^2\nu _jdS -\displaystyle \mu_{1,k}\eps _k^2\intPB \big(\hat v_{2,k}+\hat v_{1,k}\big) \hat \xi _{2,k}\nu _jdS \\[4mm]
&& -\displaystyle\frac{ b}{2} C_\infty^2\sigma _k^2\eps_k^2\intPB \big(\hat v^2_{2,k}+\hat v^2_{1,k}\big) \big(\hat v_{2,k}+\hat v_{1,k}\big)\hat \xi _{2,k}\nu _jdS\\[4mm]
&& -\displaystyle\frac{\beta _k }{a^*}\intPB   \big[\hat u^2_{2,k}  (\hat v_{2,k}+\hat v_{1,k})\hat \xi _{2,k}+\hat v^2_{1,k}  (\hat u_{2,k}+\hat u_{1,k})\hat \xi _{1,k}\big]\nu_jdS.
\end{array}\label{5.2:9B}
\]
Here the coefficient $\hat c_k$ is defined by
\begin{equation}
\arraycolsep=1.5pt
 \begin{array}{lll}
\hat c_k:&=\displaystyle\frac{-\varepsilon_{k}^2(\mu_{2,k}-\mu_{1,k})}{\| \hat u_{2,k}-  \hat u_{1,k}\|_{L^\infty(\R^2)}+\frac{1}{\varepsilon_{k}}\| \hat v_{2,k}- \hat v_{1,k}\|_{L^2(\R^2)}}\\[3mm]
&= \displaystyle \frac{a_k}{2(a^*)^2\eps ^2_k}\inte ( \hat u_{2,k}^2+ \hat u_{1,k}^2)( \hat u_{2,k}+\hat u_{1,k})\hat \xi_{1,k}\\[3mm]
&\quad +\displaystyle \frac{b}{2\eps ^2_k}C_\infty^4\sigma _k^4\eps_k^4\inte ( \hat v_{2,k}^2+ \hat v_{1,k}^2)( \hat v_{2,k}+\hat v_{1,k})\hat \xi_{2,k}\\[3mm]
\quad &\quad +\displaystyle\frac{\beta _k}{a^*\eps ^2_k}C_\infty^2\sigma _k^2\eps_k^2\inte \Big[\hat v_{2,k}^2( \hat  u_{2,k}+ \hat u_{1,k})\hat \xi_{1,k}+\hat u_{1,k}^2( \hat  v_{2,k}+ \hat v_{1,k})\hat \xi_{2,k}\Big],
\end{array}
\label{5.2:9FF}
\end{equation}
due to (\ref{uniq:a-6}).
Since (\ref{uniq:B-7}) gives that
\[
\|\hat \xi_{1,k}\|_\infty\le 1\ \ \text{and}\,\ \inte |\hat \xi_{2,k}|^2\le  \eps_k^2,
\]
we have
\[\arraycolsep=1.5pt\begin{array}{lll}
\Big|\displaystyle\inte \hat  u^2_{1,k}(\hat  v_{2,k}+\hat  v_{1,k})\hat \xi_{2,k}\Big|&\le &\displaystyle\inte \hat  u^2_{1,k}(\hat  v_{2,k}+\hat  v_{1,k})|\hat \xi_{2,k}|\\[3mm]
&\le &\displaystyle\Big(\inte \hat  u^4_{1,k}(\hat  v_{2,k}+\hat  v_{1,k})^2\Big)^{\frac{1}{2}}\Big(\displaystyle\inte |\hat \xi_{2,k}|^2\Big)^{\frac{1}{2}}\le C\eps_k^2.
\end{array}\]
The above argument then yields that there exists a constant $C>0$ such that
\begin{equation}
|\hat c_{k}|\le   C \ \ \text{uniformly in}\, \ k.
\label{5.2:ck}
\end{equation}

Applying Lemma \ref{lem4.3}, if $\delta >0$ is small, we then deduce that
\[\arraycolsep=1.5pt\begin{array}{lll}
&&\displaystyle \eps _k^2\intPB \Big|\frac{\partial \hat u_{2,k}}{\partial  x_j}\frac{\partial \hat \xi_{1,k}}{\partial  \nu}\Big|dS\\[4mm]
&\le &\displaystyle \eps _k\Big(\intPB \Big|\frac{\partial \hat u_{2,k}}{\partial  x_j}\Big|^2dS\Big)^{\frac{1}{2}}\Big(\eps _k^2\intPB \Big|\frac{\partial \hat \xi_{1,k}}{\partial  \nu}\Big|^2dS\Big)^{\frac{1}{2}}\le C\eps _k^2e^{-\frac{C\delta}{\eps_k}}\,\ \mbox{as} \,\ k\to\infty,
\end{array}\label{5.2:9a}
\]
due to the fact that $\nabla \hat u_{2,k}(\eps _kx+x_{2,k})$ decays exponentially as mentioned soon after (\ref{uniq:a-6}), where $C>0$ is independent of $k$. Similarly, we have
\[
\eps _k^2\intPB \Big|\frac{\partial \hat \xi_{1,k}}{\partial  x_j}\frac{\partial \hat u_{1,k}}{\partial  \nu} \Big|dS\le C\eps _k^2e^{-\frac{C\delta}{\eps_k}}\,\ \mbox{as} \,\ k\to\infty,
\]
and
\[
\eps _k^2\Big|\displaystyle\intPB \nabla \hat \xi_{1,k} \cdot\nabla \big(\hat u_{2,k}+\hat u_{1,k}\big)\nu _jdS\Big|\le C\eps _k^2e^{-\frac{C\delta}{\eps_k}}\,\ \mbox{as} \,\ k\to\infty.
\]
On the other hand, we also get that
\[\arraycolsep=1.5pt\begin{array}{lll}
&&\Big|\displaystyle \eps _k^2 \intPB V_1(x)\big(\hat u_{2,k}+\hat u_{1,k}\big) \hat \xi _{1,k} \nu _jdS\Big|+\Big|\intPB \big(\hat u_{2,k}+\hat u_{1,k}\big) \hat \xi _{1,k}\nu _jdS\Big| \\[4mm]
&& +\Big|\displaystyle \intPB \big(\hat u^2_{2,k}+\hat u^2_{1,k}\big) \big(\hat u_{2,k}+\hat u_{1,k}\big)\hat \xi _{1,k}\nu _jdS\Big|+\Big| \displaystyle\intPB \hat u_{2,k}^2\nu _jdS\Big|\\[4mm]
&&=o(e^{-\frac{C\delta}{\eps_k}})\,\ \mbox{as} \,\ k\to\infty,
\end{array}
\label{5.2:9b}
\]
where the exponential decay of $\hat u_{i,k}$ is also used. We thus conclude from above that
\begin{equation}\label{5.2:9aB}
\mathcal{I}^u_k+(a^*C_\infty^2\sigma _k^2\eps_k^2)\mathcal{I}^v_k=o(e^{-\frac{C\delta}{\eps_k}}) \,\ \mbox{as} \,\ k\to\infty,
\end{equation}
where $C>0$ is independent of $k$. It now follows from (\ref{5.2:8}) and (\ref{5.2:9aB}) that the claim (\ref{5.2:10}) holds for $j=1,\,2$.

\vskip 0.05truein

\noindent{\em  Step 3.} The constants $b_1=b_2=c_0=0$ in (\ref{uniq:limit-A3}), i.e., $\xi_{10}=0$ and $\xi_{20}=b_0w$ for some constant $b_0$.

Using the integration by parts,  we first note that
\begin{equation}\arraycolsep=1.5pt
\begin{array}{lll}
&&-\displaystyle\eps _k^2 \intB \big[(x-x_{2,k})\cdot \nabla \hat u_{i,k}\big] \Delta \hat u_{i,k} \\
&=&- \eps _k^2\displaystyle \intPB \frac{\partial\hat u_{i,k} }{\partial \nu }(x-x_{2,k})\cdot \nabla \hat u_{i,k}\\
&&+\displaystyle\eps _k^2 \intB \nabla \hat u_{i,k}\nabla \big[(x-x_{2,k})\cdot \nabla \hat u_{i,k}\big] \\
&=& - \eps _k^2\displaystyle \intPB \frac{\partial\hat u_{i,k} }{\partial \nu }(x-x_{2,k})\cdot \nabla \hat u_{i,k}
\\
&&+\displaystyle \frac{\eps _k^2}{2}\intPB \big[(x-x_{2,k})\cdot \nu \big]|\nabla \hat u_{i,k}|^2.
\end{array}\label{5.3:3}
\end{equation}
Multiplying the first equation of (\ref{5.2:0}) by $ (x-x_{2,k})\cdot \nabla \hat u_{i,k} $, where $i=1,2$, and integrating over $B_\delta (x_{2,k})$, where $\delta >0$ is small as before, we deduce that for $i=1,2,$
\[\arraycolsep=1.5pt\begin{array}{lll}
&&-\displaystyle\eps _k^2 \intB \big[(x-x_{2,k})\cdot \nabla \hat u_{i,k}\big] \Delta \hat u_{i,k} \\[4mm]
&=& \displaystyle\eps _k^2 \intB \big[\mu_{i,k}-V_1(x)\big] \hat u_{i,k}\big[(x-x_{2,k})\cdot \nabla \hat u_{i,k}\big]\\[4mm]
&& +
\displaystyle \frac{ a_k}{a^*} \intB \hat u_{i,k}^3\big[(x-x_{2,k})\cdot \nabla \hat u_{i,k}\big]\\[4mm]
&&
+\displaystyle \frac{\beta _k }{2} C_\infty^2\sigma _k^2\eps_k^2\intB \hat v_{i,k}^2\big[(x-x_{2,k})\cdot \nabla \hat u^2_{i,k}\big]\\[4mm]
&=& -\displaystyle\frac{\eps _k^2}{2} \intB\hat u_{i,k}^2\Big\{2\big[\mu_{i,k}-V_1(x)\big]-(x-x_{2,k})\cdot \nabla V_1(x)\Big\}\\[4mm]
&&+\displaystyle\frac{\eps _k^2}{2} \intPB\hat u_{i,k}^2\big[\mu_{i,k}-V_1(x)\big](x-x_{2,k})\nu dS\\[4mm]
&&-\displaystyle \frac{a_k}{2a^*} \intB \hat u_{i,k}^4+\displaystyle \frac{  a_k}{4a^*} \intPB \hat u_{i,k}^4(x-x_{2,k})\nu dS\\[4mm]
&&-\displaystyle  \beta _k C_\infty^2\sigma _k^2\eps_k^2 \intB \hat u_{i,k}^2\hat v_{i,k}^2-\displaystyle \frac{\beta _k }{2} C_\infty^2\sigma _k^2\eps_k^2\intB \hat u_{i,k}^2\big[(x-x_{2,k})\cdot \nabla \hat v^2_{i,k}\big]\\[4mm]
&&+\displaystyle \frac{\beta _k }{2}C_\infty^2\sigma _k^2\eps_k^2 \intPB \hat u_{i,k}^2\hat v_{i,k}^2 (x-x_{2,k})\nu dS.
\end{array}\]
Since $x\cdot \nabla V_1(x)=p_1V_1(x)$, this yields that
\begin{equation}\arraycolsep=0.5pt
\begin{array}{lll}
&&-\displaystyle\eps _k^2 \intB \big[(x-x_{2,k})\cdot \nabla \hat u_{i,k}\big] \Delta \hat u_{i,k} \\[4mm]
&=& -\mu_{i,k}\displaystyle\eps _k^2 \inte \hat u_{i,k}^2+\displaystyle\frac{2+p_1}{2}\eps _k^2 \inte V_1(x)\hat u_{i,k}^2-\displaystyle\frac{\eps _k^2}{2} \inte \big[x_{2,k}\cdot \nabla V_1(x)\big]\hat u_{i,k}^2\\[4mm]
&&-\displaystyle \frac{  a_k}{2a^*} \inte \hat u_{i,k}^4-\displaystyle  \beta _k C_\infty^2\sigma _k^2\eps_k^2 \inte \hat u_{i,k}^2\hat v_{i,k}^2\\[4mm]
&&-\displaystyle \frac{\beta _k }{2}C_\infty^2\sigma _k^2\eps_k^2 \intB \hat u_{i,k}^2\big[(x-x_{2,k})\cdot \nabla \hat v^2_{i,k}\big]+I_i,
\end{array}\label{5.3:1}
\end{equation}
where the lower order term $I_i$ satisfies
\begin{equation}\arraycolsep=1.5pt
\begin{array}{lll}
I_i&=& \mu_{i,k}\displaystyle\eps _k^2 \int _{\R^2\backslash B_\delta (x_{2,k})} \hat u_{i,k}^2-\displaystyle\frac{2+p_1}{2}\eps _k^2 \int _{\R^2\backslash B_\delta (x_{2,k})} V_1(x)\hat u_{i,k}^2
\\[4mm]
&&+\displaystyle \frac{a_k}{2a^*} \int _{\R^2\backslash B_\delta (x_{2,k})} \hat u_{i,k}^4+\displaystyle\frac{1}{2} \eps _k^2\int _{\R^2\backslash B_\delta (x_{2,k})}  \big[x_{2,k}\cdot \nabla V_1(x)\big] \hat u_{i,k}^2\\[4mm]
&&+\displaystyle\frac{\eps _k^2}{2} \intPB\hat u_{i,k}^2\big[\mu_{i,k}-V_1(x)\big](x-x_{2,k})\nu dS\\[4mm]
&&+\displaystyle \frac{  a_k}{4a^*} \intPB \hat u_{i,k}^4(x-x_{2,k})\nu dS+\displaystyle \beta _k C_\infty^2\sigma _k^2\eps_k^2 \int _{\R^2\backslash B_\delta (x_{2,k})} \hat u_{i,k}^2\hat v_{i,k}^2\\[2mm]
&&+\displaystyle \frac{\beta _k }{2} C_\infty^2\sigma _k^2\eps_k^2\intPB \hat u_{i,k}^2\hat v_{i,k}^2 (x-x_{2,k})\nu dS,\,\ i=1,2.
\end{array}\label{5.3:2}
\end{equation}
Similarly, we have
\begin{equation}\arraycolsep=1.5pt
\begin{array}{lll}
&&-\displaystyle\eps _k^2 \intB \big[(x-x_{2,k})\cdot \nabla \hat v_{i,k}\big] \Delta \hat v_{i,k} \\[2mm]
=&& - \eps _k^2\displaystyle \intPB \frac{\partial\hat v_{i,k} }{\partial \nu }(x-x_{2,k})\cdot \nabla \hat v_{i,k}\\
&&+\displaystyle \frac{\eps _k^2}{2}\intPB \big[(x-x_{2,k})\cdot \nu \big]|\nabla \hat v_{i,k}|^2,
\end{array}\label{5.3v:3}
\end{equation}
and  the second equation of (\ref{5.2:0}) yields that
\begin{equation}\arraycolsep=1.5pt
\begin{array}{lll}
&&-\displaystyle\eps _k^2 \intB \big[(x-x_{2,k})\cdot \nabla \hat v_{i,k}\big] \Delta \hat v_{i,k} \\[2mm]
&=& -\mu_{i,k}\displaystyle\eps _k^2 \inte \hat v_{i,k}^2+\displaystyle\frac{2+p_2}{2}\eps _k^2 \inte V_2(x)\hat v_{i,k}^2-\displaystyle\frac{\eps _k^2}{2} \inte \big[x_{2,k}\cdot \nabla V_2(x)\big]\hat v_{i,k}^2\\[2mm]
&&-\displaystyle \frac{b}{2}C_\infty^2\sigma _k^2\eps_k^2 \inte \hat v_{i,k}^4+\displaystyle \frac{\beta _k}{2a^*} \intB \hat u_{i,k}^2\big[(x-x_{2,k})\cdot \nabla \hat v^2_{i,k}\big]+II_i,
\end{array}\label{5.3v:1}
\end{equation}
where the lower order term $II_i$ satisfies
\begin{equation}\arraycolsep=1.5pt
\begin{array}{lll}
II_i&=& \mu_{i,k}\displaystyle\eps _k^2 \int _{\R^2\backslash B_\delta (x_{2,k})} \hat v_{i,k}^2-\displaystyle\frac{2+p_2}{2}\eps _k^2 \int _{\R^2\backslash B_\delta (x_{2,k})} V_2(x)\hat v_{i,k}^2\\[2mm]
&&+\displaystyle \frac{b}{2}C_\infty^2\sigma _k^2\eps_k^2 \int _{\R^2\backslash B_\delta (x_{2,k})} \hat v_{i,k}^4+\displaystyle\frac{1}{2} \eps _k^2\int _{\R^2\backslash B_\delta (x_{2,k})}  \big[x_{2,k}\cdot \nabla V_2(x)\big] \hat v_{i,k}^2\\[2mm]
&&+\displaystyle\frac{\eps _k^2}{2} \intPB\hat v_{i,k}^2\big[\mu_{i,k}-V_2(x)\big](x-x_{2,k})\nu dS\\[2mm]
&&+\displaystyle \frac{b}{4}C_\infty^2\sigma _k^2\eps_k^2 \intPB \hat v_{i,k}^4(x-x_{2,k})\nu dS,\,\ i=1,2.
\end{array}\label{5.3v:2}
\end{equation}

Since it follows from (\ref{uniq:a-6H}) and (\ref{uniq:a-2}) that
\[\arraycolsep=1.5pt\begin{array}{lll}
a^*\eps_k^4 \,e(a_k, b,\beta _k )&=&\mu_{i,k}\eps _k^2\Big[\displaystyle \inte \hat u_{i,k}^2+\displaystyle  a^*C_\infty^2\sigma _k^2\eps_k^2\inte \hat v_{i,k}^2\Big]+\displaystyle \frac{a_k}{2a^*} \inte \hat u_{i,k}^4\\[3mm]
&&+\displaystyle \frac{b}{2}a^*C_\infty^4\sigma _k^4\eps_k^4\inte \hat v_{i,k}^4+\displaystyle \beta _k C_\infty^2\sigma _k^2\eps_k^2\inte \hat u_{i,k}^2\hat v_{i,k}^2,
\end{array}\]
using (\ref{5.3:3})--(\ref{5.3v:2}), we then conclude from above that
\begin{equation} \arraycolsep=1.5pt
\begin{array}{lll}
&&a^*\eps_k^4 \,e(a_k, b,\beta _k )-\displaystyle\frac{2+p_1}{2}\eps _k^2 \inte V_1(x)\hat u_{i,k}^2\\
&&-\displaystyle\frac{2+p_2}{2} \eps _k^2(a^*C_\infty^2\sigma _k^2\eps_k^2 )\inte V_2(x)\hat v_{i,k}^2
\\[3mm]
&&+\displaystyle\frac{\eps _k^2}{2} \inte \big[x_{2,k}\cdot \nabla V_1(x)\big]\hat u_{i,k}^2+\displaystyle\frac{\eps _k^2}{2}(a^*C_\infty^2\sigma _k^2\eps_k^2 ) \inte \big[x_{2,k}\cdot \nabla V_2(x)\big]\hat v_{i,k}^2
\\[3mm]
&=&I_i+(a^*C_\infty^2\sigma _k^2\eps_k^2 )II_i+\eps _k^2\displaystyle \intPB \frac{\partial\hat u_{i,k} }{\partial \nu }(x-x_{2,k})\cdot \nabla \hat u_{i,k}\\[4mm]
&&-\displaystyle \frac{\eps _k^2}{2}\intPB \big[(x-x_{2,k})\cdot \nu \big]|\nabla \hat u_{i,k}|^2\\[4mm]
&&+\eps _k^2(a^*C_\infty^2\sigma _k^2\eps_k^2 )\displaystyle \intPB \frac{\partial\hat v_{i,k} }{\partial \nu }(x-x_{2,k})\cdot \nabla \hat v_{i,k}\\[4mm]
&&-\displaystyle \frac{\eps _k^2}{2}(a^*C_\infty^2\sigma _k^2\eps_k^2 )\intPB \big[(x-x_{2,k})\cdot \nu \big]|\nabla \hat v_{i,k}|^2
:=\mathcal{B}_i,\,\ i=1,2,
 \label{5.3:A2}
\end{array}\end{equation}
which then implies that
\begin{equation} \arraycolsep=1.5pt\begin{array}{lll}
&&\displaystyle\frac{2+p_1}{2}\eps _k^{2} \inte V_1(x)(\hat u_{2,k}+\hat u_{1,k})\hat \xi _{1,k} \\
&&-\displaystyle\frac{\eps _k^2}{2} \inte \big[x_{2,k}\cdot \nabla V_1(x)\big](\hat u_{2,k}+\hat u_{1,k})\hat \xi _{1,k} \\[3mm]
&& +\displaystyle\frac{2+p_2}{2}\eps _k^{2}(a^*C_\infty^2\sigma _k^2\eps_k^2 ) \inte V_2(x)(\hat v_{2,k}+\hat v_{1,k})\hat \xi _{2,k} \\[3mm]
&&-\displaystyle\frac{\eps _k^2}{2}(a^*C_\infty^2\sigma _k^2\eps_k^2 ) \inte \big[x_{2,k}\cdot \nabla V_2(x)\big](\hat v_{2,k}+\hat v_{1,k})\hat \xi _{2,k}:=-T_k.
 \label{5.3:A3}
\end{array}\end{equation}
We shall prove in the appendix that $T_k$ satisfies
\begin{equation}
T_k:=\frac{\mathcal{B}_2-\mathcal{B}_1}{\|\hat u_{2,k}- \hat  u_{1,k}\|_{L^\infty(\R^2)}+\frac{1}{\eps _k }\|\hat  v_{2,k}- \hat v_{1,k}\|_{L^2(\R^2)}}=o(e^{-\frac{C\delta}{\eps_k}}),
 \label{5.3:exponent}
\end{equation}
where $\mathcal{B}_i$ is defined in (\ref{5.3:A2}) for $i=1, 2$. We thus conclude from (\ref{5.3:A3}) and (\ref{5.3:exponent}) that
\begin{equation} \arraycolsep=1.5pt\begin{array}{lll}
&&\displaystyle\frac{2+p_1}{2}  \inte V_1\big(x+\frac{x_{2,k}}{\eps_k}\big)(\bar u_{2,k}+\bar  u_{1,k})  \xi _{1,k} \\[3mm]
&&-\displaystyle\frac{1}{2} \inte \Big[\frac{x_{2,k}}{\eps_k}\cdot \nabla V_1\big(x+\frac{x_{2,k}}{\eps_k}\big)\Big](\bar  u_{2,k}+\bar  u_{1,k})  \xi _{1,k} \\[3mm]
&& +\displaystyle\frac{2+p_2}{2}\eps _k^{p_2-p_1}(a^*C_\infty^2\sigma _k^2\eps_k^2 ) \inte V_2\big(x+\frac{x_{2,k}}{\eps_k}\big)(\bar  v_{2,k}+\bar  v_{1,k})  \xi _{2,k} \\[3mm]
&&-\displaystyle\frac{\eps _k^{p_2-p_1}}{2}(a^*C_\infty^2\sigma _k^2\eps_k^2 ) \inte \Big[\frac{x_{2,k}}{\eps_k}\cdot \nabla V_2\big(x+\frac{x_{2,k}}{\eps_k}\big)\Big](\bar  v_{2,k}+\bar  v_{1,k})  \xi _{2,k}\\
&&=o(e^{-\frac{C\delta}{\eps_k}}).
 \label{5.3:AA3}
\end{array}\end{equation}

We next establish Step 3 as follows.
Since $p_1\le p_2$ and $\sigma _k^2\eps_k^2\to 0$ as $k\to\infty$, we then conclude from (\ref{5.2:10}) and (\ref{5.3:AA3}) that
\[
\inte V_1\big(x+\frac{x_{2,k}}{\eps_k}\big)(\bar u_{2,k}+\bar  u_{1,k})  \xi _{1,k}=o(1)\,\ \mbox{as} \,\ k\to\infty.
\]
Following this, we then obtain from (\ref{1:H}) that
\[\arraycolsep=1.5pt\begin{array}{lll}
0&=&2\displaystyle\inte V_1 (x+y_0)w  \xi _{10}\\
&=&2c_0\displaystyle\inte V_1 (x+y_0)\big( w^2+\displaystyle\frac{1}{2}x\cdot\nabla w^2\big)\\
&=&2c_0\Big\{\displaystyle\inte V_1 (x+y_0)w^2-\displaystyle\frac{1}{2}\inte w^2\big[ 2V_1 (x+y_0)+x\cdot \nabla V_1 (x+y_0) \big] \Big\}\\[4mm]
&=&-\displaystyle p_1 c_0\inte V_1 (x+y_0)w^2+c_0\displaystyle\inte   w^2\big[y_0\cdot \nabla V_1 (x+y_0)\big]\\[4mm]
&=&-\displaystyle p_1 c_0\inte V_1 (x+y_0)w^2=-\displaystyle p_1 c_0H_1(y_0),
\end{array}\]
which therefore implies that $c_0=0$.
Using $c_0=0$, we further derive from (\ref{uniq:limit-A3}) and (\ref{5.2:10}) that
\[ \arraycolsep=1.5pt\begin{array}{lll}
0=2\displaystyle\inte  \frac{\partial V_1(x+y_0)}{\partial x_j}u_0\,\xi_{10}&=&\displaystyle2\inte \frac{\partial V_1(x+y_0)}{\partial x_j}u_0\Big(\sum ^2_{i=1}b_i\frac{\partial u_0}{\partial x_i}\Big)\\[3mm]
&=&-\displaystyle\sum ^2_{i=1}b_i\inte  \frac{\partial ^2 V_1(x+y_0)}{\partial x_j\partial x_i}u_0^2,\quad j=1,\,2,
\end{array}\]
which then gives that $b_1=b_2=0$  in (\ref{uniq:limit-A3}), due to the non-degeneracy assumption (\ref{1:H}). Therefore, we have  $c_0=b_1=b_2=0$, which implies that $\xi_{10}=0$ and $\xi_{20}=b_0w$ for some constant $b_0$.

\vskip 0.05truein

\noindent{\em  Step 4.} There exist two constants $b_{11}$ and $b_{12}$ such that $\xi_{1,k}$ satisfies
\begin{equation}\label{uniq:3TT}
  \xi_{1,k}=\Big[-b_0w+\sum ^2_{i=1}b_{1i}\frac{\partial w}{\partial x_i}\Big](a^*C_\infty^2\sigma _k^2\eps_k^2)+o(\sigma _k^2\eps_k^2) \, \ \text{as} \, \ k\to\infty,
\end{equation}
where the constant $b_0$ is the same as that of $\xi_{20}=b_0w$ given in (\ref{uniq:limit-A3}).

Actually, similar to the proof of (3.6) in \cite{GLW}, one can obtain from (\ref{step-1:9}) that
\begin{equation}\label{uniq:3TP}
\xi_{1,k}= (a^*C_\infty^2\sigma _k^2\eps_k^2)\xi_1+o(\sigma _k^2\eps_k^2),\, \ \text{as} \, \ k\to\infty,
\end{equation}
where $\xi_1$ is a unique solution of $\nabla \xi_1(0)=0$ and
\[\arraycolsep=1.5pt\begin{array}{lll}
\Delta \xi_1-\xi_1+3w^2\xi_1- \displaystyle\frac{2}{a^*}\Big(\inte w^3\xi_1\Big)w&=&-2w^2\xi_{20}+ \displaystyle\frac{2}{a^*}\Big(\inte w^3\xi_{20}\Big)w\\[3mm]
&=&-2b_0w^3+4b_0w \  \ \text{in}\  \ \R^2,
\end{array}\]
since $\xi_{20}=b_0w$.
One can check that $\xi_1$ satisfies \begin{equation}\label{uniq:3TPM}
\xi_1=-b_0w+\sum ^2_{i=1}b_{1i}\frac{\partial w}{\partial x_i}
\end{equation}
for some constants $b_{11}$ and $b_{12}$, where the constant $b_0$ is the same as that of $\xi_{20}=b_0w$ given in (\ref{uniq:limit-A3}).
Therefore, the estimate (\ref{uniq:3TT}) now follows from (\ref{uniq:3TP}) and (\ref{uniq:3TPM}).

\vskip 0.02truein

\noindent{\em  Step 5.}  $b_0=0$ in (\ref{uniq:limit-A3}), i.e., $\xi_{10}=\xi_{20}=0$.

We shall consider separately the following two cases:

\noindent\text{\em Case 1: $p_1<p_2$. }
In this case, we follow from (\ref{5.2:10}) and Step 4 that
\begin{equation}
b_{11}\frac{\partial ^2 H_1(y_0)}{\partial x_1\partial x_j}+b_{12}\frac{\partial ^2 H_1(y_0)}{\partial x_2\partial x_j}=0,\ \ j=1, 2,
 \label{5.3:BB3}
\end{equation}
which then implies that $b_{11}=b_{12}=0.$
It thus yields from (\ref{uniq:3TT}) and (\ref{5.3:AA3}) that
\begin{equation}\label{5.3:C1}\arraycolsep=1.5pt\begin{array}{lll}
0&=&-2(2+p_1)H_1(y_0)b_0+\displaystyle\Big(y_0\cdot \frac{\partial \nabla H_1(y_0)}{\partial x_1}\Big)b_{11}+\displaystyle\Big(y_0\cdot \frac{\partial \nabla H_1(y_0)}{\partial x_2}\Big)b_{12}\\[2mm]&=& -2(2+p_1)H_1(y_0)b_0,
\end{array}\end{equation}
which gives that $b_0=0$, since $H_1(y_0)>0$. Therefore, we have $\xi_{10}=\xi_{20}=0$ in this case.

\vskip 0.02truein

\noindent\text{\em Case 2: $p_1=p_2$. }
In this case,  we deduce from (\ref{5.2:10}) and Step 4 that
\begin{equation}
2b_0\frac{\partial H_2(y_0)}{\partial x_j}-b_{11}\frac{\partial ^2 H_1(y_0)}{\partial x_1\partial x_j}-b_{12}\frac{\partial ^2 H_1(y_0)}{\partial x_2\partial x_j}=0,\ \ j=1, 2.
 \label{5.3:D3}
\end{equation}
However, it yields from (\ref{uniq:3TT}) and (\ref{5.3:AA3}) that
\begin{equation}\label{5.3:D1}\arraycolsep=1.5pt\begin{array}{lll}
&&\Big\{2(2+p_1)\big[H_2(y_0)-H_1(y_0)\big]-2y_0\cdot \nabla H_2(y_0)\Big\} b_0\\[3mm]
&&+\displaystyle\Big(y_0\cdot \frac{\partial \nabla H_1(y_0)}{\partial x_1}\Big)b_{11}+\displaystyle\Big(y_0\cdot \frac{\partial \nabla H_1(y_0)}{\partial x_2}\Big)b_{12}=0.
\end{array}\end{equation}
We thus obtain from (\ref{5.3:D3}) and (\ref{5.3:D1}) that
\begin{equation}\label{5.3:D2}
\begin{pmatrix}
M_{11}  &  y_0\cdot \frac{\partial \nabla H_1(y_0)}{\partial x_1}   & y_0\cdot \frac{\partial \nabla H_1(y_0)}{\partial x_2} \\[2mm]
 2\frac{\partial H_2(y_0)}{\partial x_1} & -\frac{\partial ^2 H_1(y_0)}{\partial x_1\partial x_1} & -\frac{\partial ^2 H_1(y_0)}{\partial x_1\partial x_2}\\[2mm]
 2\frac{\partial H_2(y_0)}{\partial x_2} &  -\frac{\partial ^2 H_1(y_0)}{\partial x_1\partial x_2} & -\frac{\partial ^2 H_1(y_0)}{\partial x_2\partial x_2}
 \end{pmatrix}
\begin{pmatrix}
b_0\\[2mm] b_{11}\\[2mm] b_{12}
 \end{pmatrix}
 =0,
\end{equation}
where $M_{11}=2(2+p_1)\big[H_2(y_0)-H_1(y_0)\big]-2y_0\cdot \nabla H_2(y_0)$.
In this case, if  $H_2(y_0)\not =H_1(y_0)$,  we then derive from the non-degeneracy assumption (\ref{1:H}) that
\[ \arraycolsep=1.5pt\begin{array}{lll}
&&  \begin{vmatrix}
 2(2+p_1)\big[H_2(y_0)-H_1(y_0)\big]-2y_0\cdot \nabla H_2(y_0)\ &\ y_0\cdot \frac{\partial \nabla H_1(y_0)}{\partial x_1} \ &\ y_0\cdot \frac{\partial \nabla H_1(y_0)}{\partial x_2} \\[2mm]
 2\frac{\partial H_2(y_0)}{\partial x_1} & -\frac{\partial ^2 H_1(y_0)}{\partial x_1\partial x_1} & -\frac{\partial ^2 H_1(y_0)}{\partial x_1\partial x_2}\\[2mm]
 2\frac{\partial H_2(y_0)}{\partial x_2} &  -\frac{\partial ^2 H_1(y_0)}{\partial x_1\partial x_2} & -\frac{\partial ^2 H_1(y_0)}{\partial x_2\partial x_2}
 \end{vmatrix}\\[10mm]
 &=&
  \begin{vmatrix}
 2(2+p_1)\big[H_2(y_0)-H_1(y_0)\big] \ &\ 0 \ &\ 0 \\[2mm]
 2\frac{\partial H_2(y_0)}{\partial x_1} & -\frac{\partial ^2 H_1(y_0)}{\partial x_1\partial x_1} & -\frac{\partial ^2 H_1(y_0)}{\partial x_1\partial x_2}\\[2mm]
 2\frac{\partial H_2(y_0)}{\partial x_2} &  -\frac{\partial ^2 H_1(y_0)}{\partial x_1\partial x_2} & -\frac{\partial ^2 H_1(y_0)}{\partial x_2\partial x_2}
 \end{vmatrix}
 \\[10mm]
 &=&
2(2+p_1)\big[H_2(y_0)-H_1(y_0)\big]\det \Big(\frac{\partial ^2 H_1(y_0)}{\partial x_i\partial x_j}\big)\not =0,
\end{array}\]
which thus implies that  $b_0=0$. Therefore, we also have $\xi_{10}=\xi_{20}=0$ in this case.

\vskip 0.05truein

\noindent{\em  Step 6.} $\xi_{10}=\xi_{20}=0$ cannot occur.

Finally, let $ x_k\in\R^2$ satisfy  \begin{equation}\label{step6:D}
|\xi_{1,k}(x_k)| +\sqrt{\|\xi_{2,k}\|^2_2}=\|\xi_{1,k}(x)\|_\infty+\sqrt{\|\xi_{2,k}\|^2_2}=1.
\end{equation}
If $\xi_{1,k}\to \xi_{10}\not\equiv 0$ uniformly on $\R^2$ as $k\to\infty$, it then contradicts to the fact that $\xi_{10}\equiv 0$ on $\R^2$.

We now assume that $\xi_{1,k}\to \xi_{10}\equiv 0$ uniformly on $\R^2$ as $k\to\infty$. Since $(u_{i,k}, v_{i,k})$ decays exponentially as $|x|\to\infty$ for $i=1$ and $2$, the linear elliptic theory applied to (\ref{step-1:9}) gives that $\xi_{2,k}$ is also bounded uniformly in $\R^2$. Let $ y_k\in\R^2$ satisfy $|\xi_{2,k}(y_k)|=\|\xi_{2,k}(x)\|_\infty$. Applying the maximum principle to (\ref{step-1:9}) then yields that $|x_k|\le C$ and $|y_k|\le C$ uniformly in $k$, due to the exponential decay of $(u_{i,k}, v_{i,k})$.
By the comparison principle, one can get from (\ref{step-1:9}) that $\xi_{i,k}$ decays exponentially for $i=1$ and $2$, see \cite{GWZZ,GZZ} for similar proofs. Following these, we then conclude from (\ref{step6:D}) that there exists a large $R>0$ such that $\int_{B_R(y_k)}|\xi_{2,k}|^2dx\ge \frac{1}{2}$  uniformly in $k>0$. This further implies that $\xi_{2,k}\to \xi_{20}\not\equiv 0$ uniformly on $\R^2$ as $k\to\infty$, a contradiction again.
Therefore, the proof of Theorem \ref{thm1.3} is now complete for the first case where $u_{2,k}\not \equiv u_{1,k}$ and $v_{2,k}\not \equiv v_{1,k}$ in $\R^2$.

\vskip 0.05truein

(2).\, We next consider the second case where $u_{2,k} \equiv u_{1,k}$ and $v_{2,k}\not \equiv v_{1,k}$ in $\R^2$. In this case, we have $\xi_{1,k}\equiv 0$ and $\xi_{2,k}\not \equiv 0$ in $\R^2$, and the second equation of  (\ref{step-1:9}) gives that
\[
\xi_{2,k}\to \xi_{20}:=b_0w  \ \, \text{uniformly on}\ \, \R^2\ \, \text{as}\ \, k\to\infty
\]
for some constant $b_0$. Moreover, one can also get that both (\ref{5.2:10}) and (\ref{5.3:AA3}) hold true with $\xi_{1,k}\equiv 0$, and it follows from (\ref{5.2:10})   that
\begin{equation}\label{step6:M}
2\inte  \frac{\partial V_2(x+y_0)}{\partial x_j}w\xi_{20}=0,\,\ j=1,\, 2.
\end{equation}
On the other hand, we derive from (\ref{5.3:AA3}) that
\begin{equation} \arraycolsep=1.5pt\begin{array}{lll}
&&\displaystyle (2+p_2)  \inte V_2\big(x+\frac{x_{2,k}}{\eps_k}\big)(\bar  v_{2,k}+\bar  v_{1,k})  \xi _{2,k} \\[3mm]
&&-\displaystyle\inte\Big[\frac{x_{2,k}}{\eps_k}\cdot \nabla V_2\big(x+\frac{x_{2,k}}{\eps_k}\big)\Big](\bar  v_{2,k}+\bar  v_{1,k}) \xi _{2,k}=o(e^{-\frac{C\delta}{\eps_k}}).
 \label{5.F:AA3}
\end{array}\end{equation}
We thus reduce from (\ref{step6:M}) and (\ref{5.F:AA3}) that
\[
0=2(2+p_2)\inte V_2(x+y_0)w\xi_{20}=2b_0(2+p_2)\inte V_2(x+y_0)w^2,
\]
which then implies that $b_0=0$. Following this fact, the argument of the above Step 6 then leads to a contradiction. Therefore, Theorem \ref{thm1.3} is also proved for the second case.

\vskip 0.05truein

(3).\, As for the last case where $v_{2,k}\equiv v_{1,k}$ and $u_{2,k} \not\equiv u_{1,k}$ in $\R^2$, we have $\xi_{2,k}\equiv 0$ and $\xi_{1,k}\not \equiv 0$ in $\R^2$. Following the first equation of  (\ref{step-1:9}), we can derive that there exist some constants $b_1, b_2$ and $c_0$ such that
\[
\xi_{1,k}\to \xi_{10}:=b_1\frac{\partial w}{\partial x_1}+b_2\frac{\partial w}{\partial x_2}+c_0 (w+x\cdot \nabla w),\,\ c_0=\frac{1}{a^*}\inte w^3\xi_{10}
\]
uniformly on $\R^2$ as $k\to\infty$.  In this case, one can get that both (\ref{5.2:10}) and (\ref{5.3:AA3}) hold true with $\xi_{2,k}\equiv 0$, from which one can further deduce that $\xi_{10}\equiv 0$ in $\R^2$, see also \cite[Theorem 1.3]{GLW} for similar arguments.  Following this fact, the argument of the above Step 6 then leads again to a contradiction. This completes the proof of Theorem \ref{thm1.3}.
\qed

\appendix

\section{The proofs of Lemma \ref{lem4.3} and (\ref{5.3:exponent})}

In this appendix we shall address the proofs of Lemma \ref{lem4.3} and (\ref{5.3:exponent}).

\begin{proof}[\bf Proof of Lemma \ref{lem4.3}.] Following (\ref{uniq:B-1}), one can check from (\ref{5.2:0}) that $(\hat \xi_{1,k}, \hat \xi_{2,k})$ satisfies
\begin{equation}\label{step-1:A9}\arraycolsep=0.1pt
 \left\{\begin{array}{lll}
&  \varepsilon_{k}^2\Delta   \hat \xi_{1,k}-\varepsilon_{k}^2V_1(x)\hat \xi_{1,k}+\mu_{2,k} \varepsilon_{k}^2\hat \xi_{1,k}+\displaystyle\frac{a_k}{a^*}\big(\hat u_{2,k}^2+\hat u_{2,k}\hat u_{1,k}+\hat u_{1,k}^2\big)\hat \xi_{1,k} \\[2mm]
 &\qquad +\displaystyle \beta _k  C_\infty^2\sigma _k^2\eps_k^2 \big[\hat  v_{1,k}^2\hat \xi_{1,k}+\hat u_{2,k}( \hat v_{2,k}+\hat v_{1,k})\hat \xi_{2,k}\big] =\displaystyle \hat c_k\hat u_{1,k}  \,\ \mbox{in}\,\  \R^2,\\[2mm]
 &  \varepsilon_{k}^2\Delta \hat \xi_{2,k}-\varepsilon_{k}^2V_2(x)\hat \xi_{2,k}+\mu_{2,k} \varepsilon_{k}^2\hat \xi_{2,k}+\displaystyle b  C_\infty^2\sigma _k^2\eps_k^2 \big(\hat v_{2,k}^2+\hat v_{2,k}\hat v_{1,k}+\hat v_{1,k}^2\big)\hat \xi_{2,k} \\[2mm]
 &\qquad\qquad\quad  +\displaystyle\frac{\beta _k  }{a^*}\big[\hat  u_{1,k}^2\hat \xi_{2,k}+\hat v_{2,k}( \hat u_{2,k}+\hat u_{1,k})\hat \xi_{1,k}\big] =\displaystyle \hat c_k\hat v_{1,k} \,\ \mbox{in}\,\  \R^2,
\end{array}\right.
\end{equation}
where the coefficient $\hat c_{k}$ satisfies (\ref{5.2:9FF}).
Multiplying the first equation of (\ref{step-1:A9}) by $\hat \xi_{1,k}$ and integrating over $\R^2$, we obtain from (\ref{uniq:a-6H}) and (\ref{5.2:9FF}) that
\begin{equation}\arraycolsep=0.5pt\begin{array}{lll}
I_0&=&\displaystyle \eps ^2_k\inte |\nabla \hat  \xi_{1,k}|^2 -\mu_{2,k}\eps ^2 _k\inte  |\hat \xi_{1,k}|^2+\eps ^2_k\inte V_1(x)|\hat \xi_{1,k}|^2 \\[4mm]
&=&\displaystyle \frac{a_k}{a^*}\inte \big(\hat  u_{2,k}^2+\hat  u_{2,k}\hat  u_{1,k}+\hat  u_{1,k}^2\big)|\hat \xi_{1,k}|^2\\[4mm]
&&+\displaystyle\beta _k  C_\infty^2\sigma _k^2\eps_k^2\inte \Big[\hat v_{1,k}^2|\hat \xi_{1,k}|^2+\hat  u_{2,k}(\hat  v_{2,k}+\hat  v_{1,k})\hat \xi_{1,k}\hat \xi_{2,k}\Big]\\[4mm]
&&-\displaystyle\inte \hat u_{1,k}\hat \xi_{1,k}\Big\{\displaystyle \frac{a_k}{2(a^*)^2\eps_k^2}\inte ( \hat  u_{2,k}^2+ \hat  u_{1,k}^2)( \hat  u_{2,k}+\hat  u_{1,k})\hat \xi_{1,k}\\[3mm]
&&\qquad +\displaystyle \frac{b}{2\eps_k^2}C_\infty^4\sigma _k^4\eps_k^4\inte ( \hat  v_{2,k}^2+ \hat  v_{1,k}^2)( \hat  v_{2,k}+\hat  v_{1,k})\hat \xi_{2,k}\\[3mm]
&&\qquad +\displaystyle\frac{\beta _k}{a^*\eps_k^2}C_\infty^2\sigma _k^2\eps_k^2\inte \Big[\hat  v_{2,k}^2( \hat   u_{2,k}+ \hat  u_{1,k})\hat \xi_{1,k}+\hat  u_{1,k}^2( \hat   v_{2,k}+ \hat  v_{1,k})\hat \xi_{2,k}\Big]
\Big\}.
\end{array}\label{AA5.2:6}
\end{equation}
Recall that
\[\arraycolsep=1.5pt\begin{array}{lll}
\Big|\displaystyle\inte \hat  u_{2,k}(\hat  v_{2,k}+\hat  v_{1,k})\hat \xi_{1,k}\hat \xi_{2,k}\Big|&\le &\displaystyle\inte \hat  u_{2,k}(\hat  v_{2,k}+\hat  v_{1,k})|\hat \xi_{2,k}|\\[2mm]
&\le &\displaystyle\Big(\inte \hat  u^2_{2,k}(\hat  v_{2,k}+\hat  v_{1,k})^2\Big)^{\frac{1}{2}}\Big(\displaystyle\inte |\hat \xi_{2,k}|^2\Big)^{\frac{1}{2}}\le C\eps_k^2,
\end{array}\]
due to the fact that
\[\|\hat \xi_{1,k}\|_\infty\le 1\ \ \text{and}\,\ \inte |\hat \xi_{2,k}|^2\le  \eps_k^2.\]
Using above estimates, we derive from (\ref{AA5.2:6}) that
\begin{equation}
     I'_1:=\eps ^2_k\inte |\nabla \hat  \xi_{1,k}|^2+\frac{1}{2} \inte |\hat \xi_{1,k}|^2+ \eps ^2_k\inte V_1(x)|\hat \xi_{1,k}|^2\le I_0<C_1\eps ^2_k
\label{5.2:5}
\end{equation}
as $k\to\infty$ holds for some constant $C_1>0$.
Applying \cite[Lemma 4.5]{Cao}, we then conclude that for any $x_0\in\R^2$, there exist a small constant $\delta >0$ and $C_2>0$  such that
\[
    \int_{\partial B_\delta (x_0)} \Big[ \eps ^2_k |\nabla \hat  \xi_{1,k}|^2+ \frac{1}{2}  |\hat \xi_{1,k}|^2+ \eps ^2_k  V_1(x)|\hat \xi_{1,k}|^2\Big]dS\le C_2I'_1\le C_1C_2\eps ^2_k\,\ \mbox{as} \,\ k\to\infty,
\]
which therefore implies that (\ref{5.2:6}) holds true for $i=1$.

Similarly, applying the above argument to the second equation of (\ref{step-1:A9}), one can obtain that (\ref{5.2:6}) holds true for $i=2$, and we are therefore done.
\end{proof}


\vskip 0.1truein

\begin{proof}[\bf Proof of (\ref{5.3:exponent}).]
Following (\ref{5.3:A2}), we have  for small $\delta >0$,
\begin{eqnarray}\arraycolsep=1.5pt\begin{array}{lll}
T_k&=&\displaystyle\frac{\mathcal{B}_2-\mathcal{B}_1}{\|\hat u_{2,k}- \hat  u_{1,k}\|_{L^\infty(\R^2)}+\frac{1}{\eps _k }\|\hat  v_{2,k}- \hat v_{1,k}\|_{L^2(\R^2)}}
\\[4mm]
&=&
\displaystyle\frac{(I_2-I_1)+(a^*C_\infty^2\sigma _k^2\eps_k^2 )(II_2-II_1)}{\|\hat u_{2,k}- \hat  u_{1,k}\|_{L^\infty(\R^2)}+\frac{1}{\eps _k }\|\hat  v_{2,k}- \hat v_{1,k}\|_{L^2(\R^2)}}\\[2mm]
&&-\displaystyle \frac{\eps _k^2}{2}\intPB \big[(x-x_{2,k})\cdot \nu \big]\big(\nabla \hat u_{2,k}+\nabla \hat u_{1,k}\big)\nabla\hat \xi_{1,k}\\[2mm]
&&+\displaystyle  \eps _k^2 \intPB \Big\{\big[(x-x_{2,k})\cdot \nabla \hat u_{2,k}\big]\big(\nu \cdot \nabla\hat \xi_{1,k}\big)\\
&&\qquad \qquad\qquad\qquad \qquad\qquad +\big(\nu \cdot \nabla \hat u_{1,k}\big)\big[(x-x_{2,k})\cdot \nabla \hat \xi_{1,k}\big]\Big\}\\[2mm]
&&-\displaystyle \frac{\eps _k^2}{2}(a^*C_\infty^2\sigma _k^2\eps_k^2 )\intPB \big[(x-x_{2,k})\cdot \nu \big]\big(\nabla \hat v_{2,k}+\nabla \hat v_{1,k}\big)\nabla\hat \xi_{2,k}\\[2mm]
&&+\displaystyle  \eps _k^2 (a^*C_\infty^2\sigma _k^2\eps_k^2 )\intPB \Big\{\big[(x-x_{2,k})\cdot \nabla \hat v_{2,k}\big]\big(\nu \cdot \nabla\hat \xi_{2,k}\big)\\[2mm]
&&\qquad \qquad\qquad\qquad \qquad\qquad +\big(\nu \cdot \nabla \hat v_{1,k}\big)\big[(x-x_{2,k})\cdot \nabla \hat \xi_{2,k}\big]\Big\}
\\[2mm]
&=&\displaystyle\frac{(I_2-I_1)+(a^*C_\infty^2\sigma _k^2\eps_k^2 )(II_2-II_1)}{\|\hat u_{2,k}- \hat  u_{1,k}\|_{L^\infty(\R^2)}+\frac{1}{\eps _k }\|\hat  v_{2,k}- \hat v_{1,k}\|_{L^2(\R^2)}}+o(e^{-\frac{C\delta}{\eps_k}})\,\ \mbox{as} \,\ k\to\infty,
\end{array} \label{5.3:5}\end{eqnarray}
due to (\ref{5.2:6}), where the second equality follows by applying the argument of estimating  (\ref{5.2:9aB}). Here $I_i$ and $II_i$ satisfy (\ref{5.3:2}) and (\ref{5.3v:2}), respectively.

Using the arguments of estimating (\ref{5.2:9aB}) again, along with the exponential decay mentioned soon after (\ref{uniq:a-6}), we also derive from (\ref{5.3:2}) that for small $\delta >0$,
\begin{eqnarray*}
\arraycolsep=1.5pt\begin{array}{lll}
II&=&\displaystyle\frac{I_2-I_1}{\|\hat u_{2,k}- \hat  u_{1,k}\|_{L^\infty(\R^2)}+\frac{1}{\eps _k }\|\hat  v_{2,k}- \hat v_{1,k}\|_{L^2(\R^2)}}\\[4mm]
&=&\mu_{2,k}\displaystyle\eps _k^2 \int _{\R^2\backslash B_\delta (x_{2,k})} \big(\hat u_{2,k}+\hat u_{1,k}\big)\hat \xi_{1,k}-\displaystyle\frac{2+p_1}{2}\eps _k^2 \int _{\R^2\backslash B_\delta (x_{2,k})} \big(\hat u_{2,k}+\hat u_{1,k}\big)V_1\hat \xi_{1,k}\\[4mm]
&&+\displaystyle \frac{  a_k}{2a^*} \int _{\R^2\backslash B_\delta (x_{2,k})} \big(\hat u_{2,k}^2+\hat u_{1,k}^2\big)\big(\hat u_{2,k}+\hat u_{1,k}\big)\hat \xi_{1,k}\\[4mm]
&&+\displaystyle   \beta _k C_\infty^2\sigma _k^2\eps_k^2 \int _{\R^2\backslash B_\delta (x_{2,k})} \big[\hat v_{2,k}^2(\hat u_{2,k}+\hat u_{1,k})\hat\xi_{1,k}+\hat u_{1,k}^2 (\hat v_{2,k}+\hat v_{1,k}\big)\hat \xi_{2,k}\big]
\\[4mm]
&&+\hat c_k \displaystyle\int _{\R^2\backslash B_\delta (x_{2,k})}  \hat u_{1,k}^2+\displaystyle\frac{1}{2} \eps _k^2\int _{\R^2\backslash B_\delta (x_{2,k})}   \big[x_{2,k}\cdot \nabla V_1(x)\big]\big(\hat u_{2,k}+\hat u_{1,k}\big)\hat \xi_{1,k}\\[4mm]
&&+\displaystyle \frac{ a_k}{4a^*} \intPB   \big(\hat u_{2,k}^2+\hat u_{1,k}^2\big)\big(\hat u_{2,k}+\hat u_{1,k}\big)\hat \xi_{1,k}(x-x_{2,k})\nu dS\\[4mm]
&&-\displaystyle\frac{\eps _k^2}{2} \intPB \big(\hat u_{2,k}+\hat u_{1,k}\big)\hat \xi_{1,k}V_1(x)(x-x_{2,k})\nu dS\\[4mm]
&&+\displaystyle\frac{\mu_{2,k}\eps _k^2}{2} \intPB   \big(\hat u_{2,k}+\hat u_{1,k}\big)\hat \xi_{1,k}(x-x_{2,k})\nu dS\\[4mm]
&&
+\hat c_k \displaystyle\intPB   \hat u_{1,k}^2 (x-x_{2,k})\nu dS+\displaystyle \frac{  \beta _k }{2 }C_\infty^2\sigma _k^2\eps_k^2 \\[4mm]
&&\cdot\displaystyle\int _{\partial B_\delta (x_{2,k})} \big[\hat v_{2,k}^2(\hat u_{2,k}+\hat u_{1,k})\hat\xi_{1,k}+\hat u_{1,k}^2 (\hat v_{2,k}+\hat v_{1,k}\big)\hat \xi_{2,k}\big](x-x_{2,k})\nu dS
\end{array}
\end{eqnarray*}
as $ k\to\infty$, and hence
\begin{equation}\arraycolsep=1.5pt\begin{array}{lll}
II&=& \hat c_k \Big[ \displaystyle\int _{\R^2\backslash B_\delta (x_{2,k})}  \hat u_{1,k}^2+\displaystyle \intPB   \hat u_{1,k}^2 (x-x_{2,k})\nu dS\Big]\\[4mm]
&&+\displaystyle\frac{1}{2} \eps _k^2\int _{\R^2\backslash B_\delta (x_{2,k})}  \big[x_{2,k}\cdot \nabla V_1(x)\big]\big(\hat u_{2,k}+\hat u_{1,k}\big)\hat \xi_{1,k}+o(e^{-\frac{C\delta}{\eps_k}}),
\end{array} \label{5.3:10}
\end{equation}
as $k\to\infty$, where $\hat c_k$ is defined by (\ref{5.2:9FF}) and satisfies (\ref{5.2:ck}).
Therefore, we deduce from (\ref{5.3:5}) and (\ref{5.3:10}) that
\[
II=\displaystyle\frac{I_2-I_1}{\|\hat u_{2,k}- \hat  u_{1,k}\|_{L^\infty(\R^2)}+\frac{1}{\eps _k }\|\hat  v_{2,k}- \hat v_{1,k}\|_{L^2(\R^2)}}=o(e^{-\frac{C\delta}{\eps_k}}) \,\ \mbox{as} \,\ k\to\infty,
\]
and similarly we also have
\[
\displaystyle\frac{(a^*C_\infty^2\sigma _k^2\eps_k^2 )(II_2-II_1)}{\|\hat u_{2,k}- \hat  u_{1,k}\|_{L^\infty(\R^2)}+\frac{1}{\eps _k }\|\hat  v_{2,k}- \hat v_{1,k}\|_{L^2(\R^2)}}=o(e^{-\frac{C\delta}{\eps_k}}) \,\ \mbox{as} \,\ k\to\infty.
\]
Therefore, we conclude from (\ref{5.3:5}) that $T_k=o(e^{-\frac{C\delta}{\eps_k}})$ as $k\to\infty$, which completes the proof of  (\ref{5.3:exponent}).
\end{proof}

\bibliographystyle{amsplain}

\end{document}